%% file: random_regular.tex
\PassOptionsToPackage{dvipsnames}{xcolor}
\documentclass[letterpaper, 11pt]{article}

\usepackage[letterpaper, margin=1in, includeheadfoot]{geometry}
\usepackage{times, amsmath, amssymb, amsthm, mathtools, thmtools, lineno}
\usepackage{tikz}
\usepackage{enumerate}
\usepackage[color=gray!30]{todonotes}
\usepackage{hyperref}
\usepackage[dvipsnames]{xcolor}
\hypersetup{
	colorlinks=true,
	pdfpagemode=UseNone,
	citecolor=OliveGreen,
	linkcolor=NavyBlue,
	urlcolor=Magenta,
	pdfstartview=FitW
}
\newtheorem{theorem}{Theorem}
\newtheorem{definition}[theorem]{Definition}

\newtheorem{corollary}[theorem]{Corollary}
\newtheorem{remark}[theorem]{Remark}
\newtheorem{lemma}[theorem]{Lemma}

\usepackage{xfrac}
\usepackage{algorithm}
\usepackage[noend]{algpseudocode}
\usepackage[color=gray!30]{todonotes}
\usepackage{xcolor,colortbl}
\usepackage{booktabs}
\usepackage{multirow}
\usepackage{array}
\usepackage{tikz}
\usetikzlibrary{decorations.pathreplacing, calc, shapes, arrows.meta, decorations.markings}
\usetikzlibrary{calc,arrows.meta,shapes.geometric,decorations.pathmorphing}

\usepackage{xspace}

\title{Logarithmic Mixing of Random Walks on Dynamical Random Cluster Models}
\author{}
\author{Andreas Galanis\thanks{Department of Computer Science, University of Oxford, Oxford OX1 3QD, UK} \and Leslie Ann Goldberg\footnotemark[1] \and Xandru Mifsud\footnotemark[1]}
\date{May 5, 2026}

\DeclareGraphicsExtensions{.pdf,.png,.jpg}

\newcommand{\TVD}[2]{||#1 - #2||_{\mathrm{TV}}}
\newcommand{\tmix}{t_{\mathrm{mix}}}
\newcommand{\pmin}{p_{\min}}
\newcommand{\pmax}{p_{\max}}
\newcommand{\pu}{p_\mathrm{u}(q, d)}
\newcommand{\hmin}{h_{\min}}
\newcommand{\hmax}{h_{\max}}
\newcommand{\Hmin}{H_{\min}(d, n)}
\newcommand{\Hmax}{H_{\max}(d, n)}
\newcommand{\imin}{i_{\min}}
\newcommand{\imax}{i_{\max}}
\newcommand{\amin}{\alpha_{\min}}
\newcommand{\amax}{\alpha_{\max}}

\newcommand{\kmin}{k_{\min}}
\newcommand{\kmax}{k_{\max}}
\newcommand{\Cmax}{C_{\max}}
\newcommand{\Cmin}{C_{\min}}
\newcommand{\Cmid}{C_{\mathrm{mid}}}
\newcommand{\Cburn}{C_{\mathrm{burn}}}
\newcommand{\Tmax}{T_{\mathrm{max}}}
\newcommand{\W}{\mathcal{W}}

\newcommand{\rr}{{r(q, d, n)}}

\newcommand{\rrg}{\mathcal{G}(d, n)}
\newcommand{\Cw}{\mathsf{C}_{\mathrm{w}}}
\newcommand{\Ce}{\mathsf{C}_e}
\newcommand{\Vroots}{V_{\mathrm{roots}}}

\newcommand{\Range}[1]{\mathcal{I}\left( #1 \right)}

\newcommand{\Vgood}{V_{\mathrm{good}}}
\def\calT{\mathcal{T}}

\allowdisplaybreaks

\begin{document}
	\maketitle
    \begin{abstract}
We study random walks on dynamically evolving graphs, where the environment is given by a time-dependent subset of the edges of an underlying graph. Concretely, following the recently introduced framework of Lelli and Stauffer, we consider a random walk interacting with a dynamical random-cluster environment, in which edges are updated with rate $\mu>0$ according to Glauber dynamics with parameters $p$ and $q$, and the walker moves at rate 1 but may only traverse edges that are present at the time of the move. This setting introduces strong dependencies between the walk and the environment, as edge-update probabilities depend on the global connectivity structure.

We focus on the case where the underlying graph is a random $d$-regular graph and the parameters lie in the subcritical regime $p < \pu$ where it is known that the Glauber dynamics mixes quickly. Our main result is to show that for any $\varepsilon >0$ and all $q \ge 1$, for all $p$ in the subcritical regime, the mixing time of the joint process is $\Theta(\log n)$ (in continuous time) whenever $\mu\geq \varepsilon \log n$. This matches the mixing time of the simple random walk on a static random regular graph, showing that in this regime the evolving environment does not slow down mixing. Our proof is based on a coupling argument that uses path-count techniques to overcome the dependencies in the edge dynamics by controlling the structure of the environment along typical trajectories.
\end{abstract}
	\thispagestyle{empty}
	\newpage
	\setcounter{page}{1}
	\section{Introduction}

    Random walks on graphs are a central object of study in computer science and probability theory, providing a fundamental tool for algorithms, sampling procedures, and stochastic processes on discrete structures. Classical theory, however, largely concentrates on static underlying graphs, whereas in many complex networks the underlying structure changes over time. This dynamic evolution introduces new challenges in characterising the mixing properties of random walks. In this paper, we study random walks on a dynamically evolving environment and analyse how changes in the underlying structure affect the walk; the ``environment" for the random walk is an evolving edge set~$(\eta_t)$ of an underlying
graph $G$ (later, we will focus on random regular graphs).

We consider a well-studied model in which the graph evolves as a random process on a fixed underlying graph $G=(V,E)$. More precisely, the system is described by a continuous-time Markov process $(\eta_t,X_t)_{t\geq 0}$ where $\eta_t\in \{0,1\}^E$ encodes the set of present edges at time $t$   and $X_t\in V$ denotes the location of the walker at time~$t$.   The dynamics are driven by $|E| + 1$ independent Poisson clocks -- one clock $\Cw$ with rate~$1$ which updates the walker $X_t$, and  
a clock $\Ce$ with rate~$\mu>0$ for each edge $e\in E$.
When an edge clock $\Ce$ rings, the environment is updated at edge~$e$. Let $(\eta_{t-},X_{t-})$  denote the state  just before time $t$. Whenever the walker clock $\Cw$ rings, say at time $t$, the walker picks a  neighbour $u$ of  $X_{t-}$ in $G$ u.a.r. If the edge $\{X_{t-},u\}$ is present in  $\eta_{t-}$, then $X_t=u$. Otherwise, the walker remains at its current location, i.e., $X_t=X_{t-}$. 

Key to the dynamic graph evolution are the edge updates, which are triggered whenever an edge clock $\Ce$ rings.
Previous work has mainly focused on the independent edge-update setting, based on the percolation model, where the edge $e$ is included with probability $p$ for some fixed $p\in (0,1)$, independently from the rest of the configuration $\eta_t$.     Here, we focus on  the recently introduced framework of Lelli and Stauffer in which the update of an edge is correlated with the current connected-component structure, according to the random cluster model. For a parameter $q>1$, let $\pmin \coloneqq \min\{p, \frac{p}{q(1-p) + p}\}$ and $\pmax \coloneqq \max\{p, \frac{p}{q(1-p) + p}\}$. If $e$ is a cut edge in $\eta_{t-}$, we set $\eta_t(e)=1$ w.p.  $\pmin$ and $\eta_t(e)=0$ otherwise. If $e$ is not a cut edge in $\eta_{t-}$, we set $\eta_t(e)=1$ w.p.  $\pmax$ and $\eta_t(e)=0$ otherwise.

It can be shown that for $d$-regular graphs $G=(V,E)$ and arbitrary $q>0$ the process $(\eta_t,X_t)$ converges to the distribution $\pi_{G,p,q}\times \pi_V$, where $\pi_V$ is the uniform distribution over $V$ and $\pi_{G,p,q}$ is the random cluster measure where $\kappa(\eta)$ is the number of connected components in~$\eta$ and
\[\pi_{G,p,q}(\eta)\propto p^{|\eta|}(1-p)^{|E|-|\eta|}q^{\kappa(\eta)}\ \mbox{  for all }\eta\in \{0,1\}^E.\] We are interested in the mixing time $\tmix^{(\mu, p, q)} (G):= \max_{\eta_0,X_0}\inf \big\{t:\ \|(\eta_t,X_t)-\pi_{G,p,q}\times \pi_V\|_{TV}\leq \tfrac{1}{4}\}$ to get close to stationary from an arbitrary starting state.  A standard argument yields  that $\tmix^{(\mu, p, q)} (G)=\Omega(\log n)$ for all graphs $G$ with diameter $\Omega(\log n)$, since the walker can move graph distance at most one per walker-clock ring; we will primarily focus on obtaining matching upper bounds in the dynamic setting.

A key feature of the model is the interplay between the walk and the edge dynamics. The process $\eta_t$ corresponds to Glauber dynamics for the random-cluster model with update rate $\mu$. When $q=1$, this reduces to the independent-edge setting, where the environment mixes once each edge has been updated. For $q>1$, however, the behaviour is substantially more complex: the mixing time of Glauber dynamics is known to be exponential when $p>\pu$ and it remains open whether rapid mixing holds for all $p\leq \pu$. For random $d$-regular graphs, Blanca and Gheissari~\cite{Blanca2021} showed that the mixing time for $p<\pu$ is  $O(\tfrac{1}{\mu}\log n)$. This motivates our focus on random regular graphs and the subcritical regime $p<\pu$.

A key difficulty in analysing the dynamic random walk in the subcritical regime was already highlighted in the work of Lelli and Stauffer on the torus. In this regime, the configuration typically consists of small connected components that evolve over time. Although this prevents the walker from being trapped in moderately-large regions, the iterative reformation of components introduces subtle dependencies between the walk and the environment. In particular, the walker is confined to its current component, while the component structure itself influences the edge dynamics, making it challenging to control the mixing behaviour of the joint process.

Our main result establishes that the full joint process mixes in $O(\log n)$ time for all $p$ up to the critical threshold $\pu$,  matching the mixing time of the simple random walk on a static random regular graph. Specifically, for $q\geq 1$, $p<\pu$ and sufficiently large $\mu$, we show that the mixing time of the joint process on a random $d$-regular graph is $O(\log n)$ (with high probability over the choice of the random $d$-regular graph). 
	
\begin{restatable}{restatabletheorem}{mainThm} \label{thm:main}
Fix integer $d\geq 3$ and reals $q\geq 1$ and $p,\varepsilon > 0$ such that 
$p < \pu$.   
With probability $1 - o_n(1)$ over $G \sim \rrg$, for every edge update rate $\mu \geq \varepsilon \log n$, we have $\tmix^{(\mu, p, q)} (G) = \Theta\left(\log n\right)$.
\end{restatable}
In particular, our result shows that, throughout the subcritical regime and for sufficiently fast edge updates, the evolving environment does not slow down mixing beyond the static case. Our assumption $\mu = \Omega(\log n)$ for the edge updates ensures that the environment mixes in
$O(1)$ time, placing us in a regime where the dynamics of the walk, rather than
the environment, govern the mixing behaviour. In this regime, we show that the
walk mixes optimally, matching the static case. Extending these results to
slower environments, in particular $\mu = O(1)$, is an interesting open
problem, since the environment may then become the dominant source of mixing
bottlenecks.  
\vskip 0.2cm

\noindent \textbf{Related Work.}
 The independent edge-update setting was first considered by Peres, Stauffer, and Steif in \cite{Peres2015}, where they established the mixing time on the $d$-dimensional torus of side-length $n$ for the subcritical regime. Subsequent work in this setting obtained sharp results on the mixing time (and other related quantities) for various families of graphs when $q = 1$, notably 
the work of Peres, Sousi, and Steif~\cite{Peres2020} on the torus  and Sousi and Thomas~\cite{Sousi_2020}  on the complete graph in the supercritical regime. Hermon and Sousi~\cite{Hermon_2020} devised, for $q=1$, a general comparison principle between the dynamic and static settings on arbitrary underlying~graphs.

The analysis of dynamic random walks on evolving random cluster configurations  for arbitrary $q>0$ was initiated by Lelli and
Stauffer~\cite{Lelli2024}, where they showed optimal mixing-time results on the torus for sufficiently small $p>0$ and arbitrary $\mu>0$ (resembling to a certain extent the subcritical regime we consider here). Their proof develops a non-Markovian coupling
guided by a multi-scale space-time analysis of the environment, alternating between
random-walk coupling steps in favourable regions and identity coupling near
unfavourable regions. Our proof of Theorem~\ref{thm:main} is also based on a coupling approach, but adapts the path-count method of Lubetzky and Sly~\cite{Lubetzky_2010} to the dynamic setting; we discuss the main obstacles in the next section.

We finally remark that our notion of ``dynamic graph" differs from the notion of dynamic graphs studied in other algorithmic settings where vertex/edge updates can be arbitrary. In the random walk context, such updates are typically not ergodic unless some extra assumptions are made. We refer the reader to the work of Sauerwald and Zanetti \cite{sauerwald_et_al:LIPIcs.ICALP.2019.93} for progress on this front (see also \cite{Leran}) and related references. Various other properties beyond mixing time have been studied in this setting. For other dynamical MCMC variants and  applications, see \cite{Andres_2024, Avena:2025aa, 10.2307/26542447, AVENA20193360, Avena2025, BISKUP2018985}.

\section{Proof Overview} \label{sec:extended_proof_overview}
In the rest of the paper we fix 
an integer~$d\geq 3$ and reals $q\geq 1$ and $p, \varepsilon > 0$ such that 
$p < \pu$. Since $q \geq 1$, $\pmin = \tfrac{p}{q(1-p) + p}$ and $\pmax = p$. 
We assume that $G=(V,E)$ is sampled from $\rrg$, so it is $d$-regular and has $n$ vertices. For vertices $u, v \in V$,  $d(u, v)$ denotes the distance between $u$ and $v$ in $G$. Given $v \in V$ and a positive integer~$R$, the \textit{ball of radius $R$} around $v$, denoted by $B_R(v)$, is the set of vertices with distance at most $R$ from $v$. The \textit{boundary $\partial B_R(v)$} is the set of vertices with distance exactly $R$ from $v$. From now on, fix $R = R(d,n) \coloneqq \frac{1}{5}\log_{d-1} n$. 
Where the base is not specified, logarithms are base~$e$.\vskip 0.2cm

\noindent \textbf{High-Level Strategy.}  
We prove Theorem~\ref{thm:main} via a coupling argument, inspired by the approach of Lubetzky and Sly~\cite{Lubetzky_2010} in the static setting. There, a main component in the analysis relies on the fact that random regular graphs are locally tree-like, allowing the walk to be approximated by a biased random walk on the infinite $d$-regular tree, whose behaviour can be tracked precisely.

In our dynamic setting, this approach breaks down. The local neighbourhood of the walker evolves over time and can be highly irregular; more importantly, edge updates are no longer independent and the probability that an edge opens depends on its cut status, which in turn depends more globally on the structure of the current configuration.

Our approach to bypass this difficulty is to establish precise lower bounds on the probability that the walker follows prescribed trajectories (paths). To this end, we decompose such trajectory events into realisations of the walker clock and transition choices, together with compatible environment updates. The main challenge is to control the dependence introduced by the random-cluster dynamics, where edge refresh probabilities differ between cut and non-cut edges. Thus, instead of the local tree analysis of the static setting, we need to devise a path-probability lower bound argument that handles more robustly the dependencies that arise in the dynamic setting.

We show that, with high probability, for a large set of ``good" starting vertices, typical trajectories of length $c\log n$ (for appropriate constant $c$) avoid small cycles and explore regions that remain tree-like over the relevant time. Environment-wise, we invoke sparsity results from  ~\cite{Blanca2021} to show that along such trajectories the edges encountered by the walker are, with high probability, cut edges at their last refresh time and hence are available to the walker with probability $\pmin$. Consequently, most trajectories see cut edges and behave like a walk with bias parameter $\pmin$, giving the key technical estimate that allows us to control transition probabilities despite the dependencies.

The proof is completed via a three-phase coupling. We couple two copies of the chain, one started from an arbitrary initial state and one from stationarity. In the first phase, the chains proceed independently and we show that with constant probability the walker reaches a ``good" vertex where we can control the neighbourhood structure and presence of small cycles. In the second phase, we exploit the trajectory estimates above to ensure that the marginal distributions of the two walkers have substantial overlap. Finally, in a short third phase, we use the rapid mixing of the environment (which occurs in $O(1)$ time in our regime) to couple the full system. Altogether, this yields a coupling with constant success probability in $O(\log n)$ time.

\paragraph{\normalfont\textbf{Lower bounding walker trajectory probabilities.}} To carry out the previous strategy, a key ingredient in our analysis is to obtain an accurate lower bound on the probability that the walker follows a given trajectory over a prescribed time interval. 

For $\alpha \in \mathbb{N}$, we let $\overrightarrow{x_\alpha} \in V^{\alpha + 1}$  be a walk $(x_0, \dots, x_\alpha)$ on the underlying graph $G$ (with repeated vertices allowed) and  use $\ell(\overrightarrow{x_\alpha})$  to denote the number of stationary transitions in the walk (i.e., the number of indices $i$ such that $x_i=x_{i+1}$). For a time interval $(t, t')$, $N_\mathrm{w}(t, t')$ denotes the number of rings of the walker clock~$\Cw$. 
	
	\begin{definition}[The walker trajectory event $\widehat\Xi \left(\protect\overrightarrow{x_\alpha}, t_0, T\right)$] \label{def:walker-trajectory-event}
		 Let $\alpha \in \mathbb{N}$ and consider a walk $\overrightarrow{x_\alpha} \in V^{\alpha + 1}$. For all $t_0 \geq 0$ and $T > 0$, $\widehat\Xi \left(\overrightarrow{x_\alpha}, t_0, T\right)$ is the event that $N_\mathrm{w}(t_0, t_0 + T) = \alpha$ and the walker follows the walk $\overrightarrow{x_\alpha}$ during the time interval $(t_0, t_0 + T)$.
	\end{definition}

When an edge refreshes in the random cluster configuration corresponding to the environment, it opens with probability~$\pmin$ if it is a cut edge and with probability~$\pmax$ otherwise. If $q > 1$ then $\pmin < \pmax$ so updates depend on the current configuration. Hence bounding the probability of a prescribed trajectory requires controlling whether or not examined edges were cut edges at their last refresh. 
	
An edge fails to be a cut edge if it lies on a cycle (of the random cluster configuration) whose edges are all open at the time that the edge is refreshed. In the regime when $p < \pu$, this is highly unlikely for long cycles, after a sufficiently long ``burn-in" time of the environment dynamics. We take a different approach for short cycles.  By restricting to walks that avoid vertices on small cycles, namely cycles of order $O(\log \log n)$, we ensure that, with high probability, the examined edges are cut edges when they are refreshed. Consequently, along such trajectories, the walker effectively experiences the minimum opening probability $\pmin$.
	
	\begin{definition}[$r$, $r$-acyclic walk] \label{def:r-acyclic-walk}		
		Let $r = \rr \coloneqq \frac{3 \log_{d-1} \log n}{\log_{d-1} (2 /  (1 + \pu))}$. A walk on a graph is said to be \textit{$r$-acyclic} if it never visits a cycle of length less than $r$.
	\end{definition}

In Section~\ref{sec:xxx} we prove Lemma~\ref{lem:te_prob_lw_bdd}, which gives a lower bound on the probability that the walker follows a prescribed $r$-acyclic trajectory over a moderate time interval, after the environment has undergone a sufficiently long burn-in. The proof exercises tight control over the probability that examined edges are cut edges, and this yields a probability bound in which the stationary transitions are weighted by $1 - \pmin$  rather than the naive $1 - \pmax$. This improved bound saves a polynomial factor\footnote{If we omit the cut edge analysis and allow the effective random cluster parameters to differ so that some are $\pmin$ and others are $\pmax$, then the lower bound on the probability in Lemma \ref{lem:te_full_prob_lw_bdd} becomes smaller by a polynomial in $n$, which leads to a corresponding factor in the coupling time.} in the coupling time (and hence in the mixing time bound), and is the main point of the cut edge analysis.
	
\begin{restatable}{restatablelemma}{teProbLwBdd} \label{lem:te_prob_lw_bdd}
\sloppy There is a constant $C_0 \in (0, 1)$ such that, with probability $1 - o_n(1)$ over $G \sim \rrg$, for every $\mu \geq \varepsilon \log n $, every $T > 0$ such that $({1}/{(50 \ \pmin)}) \log_{d - 1} n \leq T \leq ({4}/{\pmin}) \log_{d - 1} n$, every $\alpha \in \mathbb{N}$ such that $\alpha \leq ({4}/{\pmin}) \log_{d - 1} n$, every~$r$-acyclic walk $\overrightarrow{x_\alpha} \in V(G)^{\alpha + 1}$, and every $\eta \in \{0, 1\}^{E(G)}$,
\begin{align*}
 \mathbb{P} [\widehat\Xi(\overrightarrow{x_\alpha}, \Cburn, T) & \mid (\eta_0, X_0) = (\eta, x_0)] \geq
C_0 \mathbb{P}[N_\mathrm{w}(\Cburn, \Cburn + T) = \alpha] \left(\frac{\pmin^{\alpha - \ell(\overrightarrow{x_\alpha})} (1 - \pmin)^{\ell(\overrightarrow{x_\alpha})}}{d^{\alpha - \ell(\overrightarrow{x_\alpha})}}\right).
		\end{align*}
	\end{restatable}

\paragraph{\normalfont\textbf{Counts of $r$-acyclic walks between ``good" vertices.}} 

Let $G=(V,E)$.
Our coupling hinges on the walker quickly hitting a ``good" set of vertices $\Vgood \subseteq V$ of size $n - o(n)$ and then becoming almost uniformly mixed over it in $O(\log n)$ time. We will introduce two properties, namely \textit{$(h, i)$-sparse} and \textit{$k$-root} vertices, which will be central to our definition of $\Vgood$. We first define some auxiliary notions needed for these properties.

It will be convenient to view a walk on a regular graph $G$ as a walk on a non-backtracking walk tree (NBWT) of $G$ (but allowing transitions from any node to its parent). Given $G = (V, E)$ and $u \in V$, the NBWT of $G$ at $u$, denoted by $\mathcal{T}_u$, has nodes corresponding to finite non-backtracking walks starting at $u$, each labelled by its terminal vertex. The root is labelled $u$ since it corresponds to the trivial walk of length $0$, and a node corresponding to $(u = x_0, x_1, \dots, x_k)$ has children given by the extensions $(u = x_0, x_1, \dots, x_k, x_{k + 1})$ such that $\{x_k, x_{k+1}\} \in E$ and $x_{k+1} \neq x_{k-1}$. 

In the dynamic setting, a walk from $u$ can then be naturally interpreted as a walk on $\mathcal{T}_u$. We say that it is $(h,i)$-\emph{constrained} for non-negative integers~$h$ and~$i$ if it ends at depth $h$ in~$\calT_u$ and has $h + 2i$ transitions that change the depth by plus or minus one.
	
\begin{definition}[$(h, i)$-constrained walk] \label{def:h-i-constrained-walk}
Let $h, i \in \mathbb{N}$. An \textit{$(h, i)$-constrained walk} on a rooted tree is a walk of length $h + 2i$ starting at the root of the tree, with exactly $h + i$ transitions increasing the depth by $+1$ and $i$ transitions decreasing the depth by~$-1$. By $\omega_{h, i}$ we denote the number of $(h,i)$-constrained walks on the rooted infinite $d$-regular tree, while by $\widetilde{\omega}_{h, i}$ we denote the number of $(h,i)$-constrained walks on the rooted infinite $(d - 1)$-ary tree.
	\end{definition}

If for a suitably large range of values for $h, i \in \mathbb{N}$ we have that almost all $(h, i)$-constrained walks from a vertex $u$ are $r$-acyclic, then Lemma \ref{lem:te_prob_lw_bdd} can be applied for a large range of $\alpha \in \mathbb{N}$ over almost all length-$\alpha$ walks from $u$. This motivates the following definition.

\begin{definition}[$(h, i)$-sparse vertex] \label{def:h-i-sparse}
A vertex $u \in V$ is said to be \textit{$(h, i)$-sparse} if at least $(1 - 1/(\log n)^3)\, \omega_{h,i}$ of the $(h, i)$-constrained walks from $u$ are $r$-acyclic.
	\end{definition}

In Appendix~\ref{sec:rrg_prop2} we will prove Lemma \ref{lem:rrg_glbl_geom}, which establishes that, with probability $1 - o_n(1)$ over $G \sim \rrg$, almost all vertices are $(h, i)$-sparse for all $h, i \leq ({4}/{\pmin}) \log_{d - 1} n$. Therefore, for every such vertex $u$ and every $\alpha \in \mathbb{N}$ such that $\alpha \leq ({4}/{\pmin}) \log_{d - 1} n$, Lemma \ref{lem:te_prob_lw_bdd} applies over almost all length-$\alpha$ walks from $u$. 

\begin{restatable}{restatablelemma}{rrgGlblGeom} \label{lem:rrg_glbl_geom} With probability $1 - o_n(1)$ over $G \sim \rrg$, for all 
$h, i \in \mathbb{N}$ such that $h, i \leq ({4}/{\pmin}) \log_{d - 1} n$,
there are $n - O\left((d-1)^r (\log n)^8\right)$ vertices which are $\left(h, i\right)$-sparse.
	\end{restatable}
	
\begin{definition}[$k$, $k$-root] \label{def:k-root}
Let $k = k(d, n) \coloneqq \left\lfloor\log_{d-1}\log n\right\rfloor$.
A vertex $v$ of a $d$-regular $n$-vertex graph~$G$ is
a \textit{$k$-root} if, and only if,   
the induced subgraph $G[B_k(v)]$ is a tree. $\Vroots$ denotes the set of all $k$-roots in~$G$.
\end{definition}

From now on fix $\Hmin \coloneqq \left\lfloor\log_{d-1} n\right\rfloor + 2 \left\lfloor\log_{d-1}\log n\right\rfloor$ and $\Hmax \coloneqq \left\lfloor\log_{d-1} n\right\rfloor + \left\lfloor\frac{1}{10}\log_{d-1} n\right\rfloor - 1$. A result from Lubetzky and Sly (\cite{Lubetzky_2010}, Lemmas 3.2 and 3.5) shows that, with probability $1 - o_n (1)$ over $G \sim \rrg$, almost all vertices are $k$-roots and for all $u, v \in \Vroots$ and $h, i \in \mathbb{N}$ satisfying $d(u, v) > 2 k$ and $h \in [\Hmin, \Hmax]$, at least a $\left(\frac{1 - o_n(1)}{n}\right)$ fraction of all $(h, i)$-constrained walks from $u$ end at $v$. In other words, for a $k$-root $u$ and $h \in [\Hmin, \Hmax]$, the terminal vertex of a uniformly chosen $(h, i)$-constrained walk from $u$ is close to uniform over $n - o(n)$ vertices which are $k$-roots. 
This underpins our rationale for viewing walks of a given length from a vertex $u$ as families of $(h, i)$-constrained walks from $u$.
We are now in a position to define the notion of a ``good" vertex, central to our analysis.

\begin{definition}[Good vertex] \label{def:good-vertex}
A vertex $u \in V$ is  \textit{good} if it is a $k$-root and it is $(h, i)$-sparse for all $h, i \in \mathbb{N}$ such that $h, i \leq \left({4}/{\pmin}\right) \log_{d - 1} n$. 
$\Vgood$ is the set of all good vertices in the graph~$G$.
	\end{definition}
	
	\paragraph{\normalfont\textbf{Coupling phases.}}
	
We next describe our coupling in more detail. We will fix the times
	\begin{align*}
		T_1 &\coloneqq 1 + \Cburn + \left({d}/{(d - 2)}\right) \left({1}/{\pmin}\right) \left({1}/{40}\right) \log_{d - 1} n, \\
		T_2 &\coloneqq T_1 + \Cburn + \left({d}/{(d - 2)}\right) \left({1}/{\pmin}\right) \left(\left({81}/{80}\right) \log_{d- 1} n + \log_{d-1} \log n\right),
	\end{align*}
and an additional time $T_3$ which we will define in the proof of Theorem~\ref{thm:main} --- it will be $T_2 + O(1)$.
 	
	Consider two copies of our full chain, $M \coloneqq (\eta_t, X_t)$ started at a given state $(\eta, u)$ and $M' \coloneqq (\eta'_t, X'_t)$ started from the stationary distribution $\pi_{G, p, q} \times \pi_V$. We will consider a coupling of these two chains over three phases, of lengths $T_1$, $T_2 - T_1$ and $T_3 - T_2$. We will first run the chains independently until time $T_1$. 
	
	We will show that at time $T_1$ the walker $X$ in chain $M$ is in $\Vgood$ with probability $\Omega(1)$. Now, given $v \in \Vgood$ and $\hat\eta \in \{0, 1\}^E$, we consider a maximal coupling between (1) the chain $M$ under the conditional distribution where $\left(\eta_{T_1}, X_{T_1}\right) = (\hat\eta, v)$ and (2) the chain $M'$ under the stationary distribution.
    We will run this   maximal coupling until time~$T_3$. We will show that by this point, the chain $M$ is close to stationarity.
	
	Lemma~\ref{lem:phase1_walk_counts}, which is proved in Section~\ref{sec:phase1_walk_counts},  establishes that from every vertex $u$, among walks of length at most ${3R}/{5}$ which end at a distance $\frac{R}{10} \leq h \leq \frac{R}{5}$ from $u$, a constant fraction of them avoid cycles of length less than $r = \rr$ (except possibly at $u$ if it is contained in such a small cycle) and end at a vertex in~$\Vgood$.
	
	\begin{restatable}{restatablelemma}{phaseOneWalkCounts} \label{lem:phase1_walk_counts}
With probability $1 - o_n(1)$ over $G \sim \rrg$, for all vertices $u$,  all $h, i \in \mathbb{N}$ such that ${R}/{10} \leq h \leq {R}/{5}$, and all $i \leq {R}/{5}$, there are at least $d - 2$ neighbours $w$ of $u$ such that from each $w$ there are ${\widetilde{\omega}_{h, i}} / {2}$ walks which are $(h, i)$-constrained, $r$-acyclic and end at a vertex in $\Vgood$.
	\end{restatable}
	
	Consequently, from the trajectory probability lower bound in Lemma \ref{lem:te_prob_lw_bdd}, we obtain that, with probability $\Omega(1)$, the walker $X$ is at a ``good" vertex at time $T_1$. This is captured in Lemma~\ref{lem:phase1_bdd}, which is proved in Section~\ref{sec:phase1_bdd}.
	
	\begin{restatable}{restatablelemma}{phaseOneBound} \label{lem:phase1_bdd}
\sloppy There is a constant $C_1 \in (0, 1)$ such that, with probability $1 - o_n(1)$ over $G \sim \rrg$, for all $\mu \geq \varepsilon \log n $ and all $(\eta, u) \in V(G) \times  \{0, 1\}^{E(G)}$, 
$\mathbb{P}\left[X_{T_1} \in \Vgood \mid \left(\eta_0, X_0\right) = \left(\eta, u\right)\right] \geq C_1$.
	\end{restatable}
	
We will show that, once the dynamics has reached a vertex $v \in \Vgood$ at time $T_1$,  
there is a set $S_v \subseteq V$ with $|S_v| = n - o(n)$ such that for all $x \in S_v$, for suitable values of $h, i \in \mathbb{N}$, there are $\Omega\left(\left({1}/{n}\right) \omega_{h, i}\right)$ walks from $v$ to $x$ which are $(h, i)$-constrained and $r$-acyclic. This  allows us to invoke Lemma \ref{lem:te_prob_lw_bdd} over all admissible values of $h$ and $i$, giving us a lower bound on the success probability of phase two by aggregating over all such trajectories. 
This is Lemma~\ref{lem:sparse_good_set_size}, which is 
proved in Section~\ref{sec:rrg_prop3} using Lemmas~\ref{lem:rrg_glbl_geom}~and~\ref{lem:kroot_path_counts}.
	
	\begin{restatable}{restatablelemma}{sparseGoodSetSize} \label{lem:sparse_good_set_size}
With probability $1 - o_n(1)$ over $G \sim \rrg$, for any $u \in \Vgood$ there exists $S_u \subseteq V$ such that $\left|S_u\right| = n - o(n)$ and 
$S_u$ has the following property. 
For all $h\in \mathbb{N}$ such that $\Hmin \leq h \leq \Hmax$, all $i \leq \left({4}/{\pmin}\right) \log_{d - 1} n$, and all   $v \in S_u$, there are at least $({1}/{4n}) \widetilde{\omega}_{h, i}$ walks from $u$ to $v$ which are $(h, i)$-constrained and $r$-acyclic.
	\end{restatable}
	
Consequently, we obtain Lemma~\ref{lem:phase2_bdd}, which is proved in Section~\ref{sec:phase2_bdd}, and
shows
that the walker $X$ is almost uniformly mixed over the set $S_v$ by time $T_2$.
	
	\begin{restatable}{restatablelemma}{phaseTwoBound} \label{lem:phase2_bdd}
There is a constant $C_2 \in (0, 1)$ such that,  with probability $1 - o_n(1)$ over $G \sim \rrg$, for all $\mu \geq \varepsilon \log n $ and all $v \in \Vgood$, there is a set $S_v \subseteq V(G)$ such that $|S_v| = n - o(n)$ and, for all $\left(\hat\eta, x\right) \in \{0, 1\}^{E(G)} \times S_v$, 
$\mathbb{P}\left[X_{T_2} = x \mid \left(\eta_{T_1}, X_{T_1}\right) = \left(\hat\eta, v\right)\right] \geq {C_2}/{n}$.
	\end{restatable}
	
	\subsection{Proof of Theorem \ref{thm:main}}
	
	\mainThm*
	
\begin{proof}
Let $G=(V,E)$. By Lemma \ref{lem:phase1_bdd}, 
there is a constant $C_1 \in (0,1)$ such that,
with probability $1-o_n(1)$ over the choice of~$G$, 
for all $u \in V$ and $\eta \in \{0, 1\}^E$, 
\begin{equation} \label{eq:main1}
\mathbb{P}\left[X_{T_1} \in \Vgood \mid \left(\eta_0, X_0\right) = \left(\eta, u\right)\right] \geq C_1.
\end{equation}
By Lemma \ref{lem:phase2_bdd} there is a constant $C_2\in (0,1)$ such that, with probability $1-o_n(1)$,
for all $v \in \Vgood$ there is a set $S_v \subseteq V$ such that $|S_v| = n - o(n)$ and for all $x \in S_v$ and $\hat\eta \in \{0, 1\}^E$,
		\begin{equation} \label{eq:main2}
			\mathbb{P}\left[X_{T_2} = x \mid \left(\eta_{T_1}, X_{T_1}\right) = \left(\hat\eta, v\right)\right] \geq {C_2}/{n}.
		\end{equation}

Blanca and Gheissari established in \cite{Blanca2021} that, with probability $1-o_n(1)$ over the choice of $G$, the mixing time of the discrete-time random cluster Glauber dynamics with parameters $(p, q)$ on $G$ satisfies $\tmix^{\mathrm{GD}, (p, q)} (G) = \Theta(n \log n)$. 
This gives a $\Theta(\log n)$ mixing time in continuous time, and a $\Theta(1)$ mixing time in our setting with $\mu \geq \epsilon \log n$.
Consequently, there is a constant $\tau > 0$ such that, with probability $1 - o_n(1)$ over the choice of $G$, for every $\mu \geq \varepsilon \log n $ and every initial distribution $\eta_0 \sim \rho$, the law of the environment dynamics $(\eta_t^{(\rho)})_{t \geq 0}$ at time $\tau$ is within total variation distance~$1/8$ of $\pi_{G, p, q}$. 

Let $T_3 \coloneqq T_2 + \tau$.
Given $v, x \in V$ and $\hat\eta \in \{0, 1\}^E$, let $\rho \coloneqq \rho\left(\hat{\eta}, v, x\right)$ be the law of $\eta_{T_2}$ conditioned on $(\eta_{T_1}, X_{T_1}) = (\hat\eta, v)$ and $X_{T_2} = x$. Let \[F_{\hat{\eta}, v, x} \! \coloneqq \! \left\{\omega \in \{0, 1\}^E \colon \mathbb{P}\left[\eta_{T_3} = \omega \mid \eta_{T_2} \sim \rho\left(\hat\eta, v, x\right)\right] \geq (1/2) \pi_{G, p, q}(\omega)\right\}.\] 
For all $\omega\in \{0,1\}^E\setminus F_{\hat\eta,v,x}$, 
$\min\{\mathbb{P}\left[\eta_{T_3} = \omega \mid \eta_{T_2} \sim \rho(\hat\eta, v, x)\right], \pi_{G, p, q}(\omega)\}$ is at most 
$(1/2)\pi_{G,p,q}(\omega)$ and for $\omega\in F_{\hat\eta,v,x}$ this quantity is at most $\pi_{G,p,q}(\omega)$. Let $\mathcal{D}(\eta^{(\rho)}_{\tau})$ be the distribution of the environment dynamics, started from the initial distribution $\rho$, at time~$\tau$. Since $\TVD{\mathcal{D}(\eta^{(\rho)}_{\tau})}{\pi_{G, p, q}} \leq 1/8$ for our choice of $\tau$, and 
since $\TVD{\mathcal{D}(\eta^{(\rho)}_{\tau})}{\pi_{G, p, q}} \! = 1 - \!\!\!\!\!\! \sum\limits_{\omega \in \{0, 1\}^E} \!\!\!\!\!\! \min\left\{\mathbb{P}\left[\eta_{T_3} = \omega \mid \eta_{T_2} \sim \rho\right], \pi_{G, p, q}(\omega)\right\}$, we have that $\pi_{G, p, q} \left(F_{\hat{\eta}, v, x}\right) \geq 3/4$. Since the environment evolves independently of the walker, if the walker clock does not ring between $T_2$ and $T_3$, the walker remains at the same vertex while the environment mixes. Thus, for all $v, x \in V$, $\hat\eta \in \{0, 1\}^E$ and $\omega \in F_{\hat{\eta}, v, x}$,
		\begin{align*}
			&\ \mathbb{P}\left[\left(\eta_{T_3}, X_{T_3}\right) = (\omega, x) \mid \left(\eta_{T_1}, X_{T_1}\right) = \left(\hat\eta, v\right)\right] \\
			\geq& \ \mathbb{P}\left[\eta_{T_3} = \omega, N_{\mathrm{w}}(T_2, T_3) = 0 \mid X_{T_2} = x, \left(\eta_{T_1}, X_{T_1}\right) = \left(\hat\eta, v\right)\right] \mathbb{P}\left[X_{T_2} = x \mid \left(\eta_{T_1}, X_{T_1}\right) = \left(\hat\eta, v\right)\right] \\
			=& \ \mathbb{P}\left[N_{\mathrm{w}}(T_2, T_3) = 0\right]\mathbb{P}\left[\eta_{T_3} = \omega \mid X_{T_2} = x, \left(\eta_{T_1}, X_{T_1}\right) = \left(\hat\eta, v\right)\right] \mathbb{P}\left[X_{T_2} = x \mid \left(\eta_{T_1}, X_{T_1}\right) = \left(\hat\eta, v\right)\right] \\
			\geq& \ (1/2) e^{-\tau} \pi_{G, p, q}(\omega)  \mathbb{P}\left[X_{T_2} = x \mid \left(\eta_{T_1}, X_{T_1}\right) = \left(\hat\eta, v\right)\right]
		\end{align*}
where the last inequality follows from the fact that $\omega \in F_{\hat{\eta}, v,x}$ and the fact that the probability of a clock of rate $1$ not ringing in $\tau$ time is $e^{-\tau}$. Combining this with (\ref{eq:main2}), 
we get that for all $v \in \Vgood$ there is a set $S_v \subseteq V$ satisfying $|S_v| = n - o(n)$ such that, for all $\left(\hat\eta, x\right) \in \{0, 1\}^E \times S_v$ and $\omega \in F_{\hat\eta, v, x}$,
		\begin{equation} \label{eq:main4}
			\mathbb{P}\left[\left(\eta_{T_3}, X_{T_3}\right) = (\omega, x) \mid \left(\eta_{T_1}, X_{T_1}\right) = \left(\hat\eta, v\right)\right] \geq \left(\frac{C_2 e^{-\tau}}{2 n}\right) \pi_{G, p, q}(\omega).
		\end{equation}
		
Under a maximal coupling from time $T_1$ to time $T_3$, from (\ref{eq:main4}) 
and from the fact that $\pi_V(x) = 1/n$,
we have that
\begin{align*}
&\ \mathbb{P}\left[\left(\eta_{T_3}, X_{T_3}\right) = \left(\eta'_{T_3}, X'_{T_3}\right) \mid \left(\eta_{T_1}, X_{T_1}\right) = \left(\hat\eta, v\right)\right]\\
=&\ \sum_{(\omega, x) \in \{0, 1\}^E \times V} \min\left\{\mathbb{P}\left[\left(\eta_{T_3}, X_{T_3}\right) = (\omega, x) \mid \left(\eta_{T_1}, X_{T_1}\right) = \left(\hat\eta, v\right)\right], \pi_{G, p, q} (\omega) \pi_V (x)\right\} \\
\geq&
\sum_{x \in S_v} \sum_{\omega \in F_{\hat\eta,v,x}} 
\min \{\frac{C_2 e^{-\tau}}{2 n} \pi_{G,p,q}(\omega), \pi_{G,p,q}(\omega) (1/n)\}
\geq \ \frac{C_2 e^{-\tau}}{2 n} \sum_{x \in S_v} \pi_{G, p, q} \left(F_{\hat{\eta}, v, x}\right). 
\end{align*}
Since $|S_v| = n - o(n)$ and $\pi_{G, p, q} (F_{\hat{\eta}, v, x}) \geq \frac{3}{4}$ for all $x$,
this is at least  
$3 C_2 e^{-\tau}({n - o(n)})/(8n)
			\geq {C_2 e^{-\tau}}/{8}$.
Therefore, the probability of coupling at time $T_3$ is
		\begin{align*}
			& \ \mathbb{P}\left[\left(\eta_{T_3}, X_{T_3}\right) = \left(\eta'_{T_3}, X'_{T_3}\right) \mid \left(\eta_0, X_0\right) = (\eta, u)\right] \\
			\geq& \ \sum_{v \in \Vgood} \sum_{\hat\eta \in \{0, 1\}^E} \mathbb{P}\left[\left(\eta_{T_3}, X_{T_3}\right) = \left(\eta'_{T_3}, X'_{T_3}\right), \left(\eta_{T_1}, X_{T_1}\right) = (\hat\eta, v) \mid \left(\eta_0, X_0\right) = (\eta, u)\right] \\ 
			\geq& \ \sum_{v \in \Vgood} \sum_{\hat\eta \in \{0, 1\}^E} \frac{C_2 e^{-\tau}}{8} \ \mathbb{P} \left[\left(\eta_{T_1}, X_{T_1}\right) = (\hat\eta, v) \mid \left(\eta_0, X_0\right) = (\eta, u) \right] \\
			=& \ \sum_{v \in \Vgood} \frac{C_2 e^{-\tau}}{8} \ \mathbb{P} \left[X_{T_1} = v \mid \left(\eta_0, X_0\right) = (\eta, u) \right] \\
			=& \ \frac{C_2 e^{-\tau}}{8} \ \mathbb{P} \left[X_{T_1} \in \Vgood \mid \left(\eta_0, X_0\right) = (\eta, u) \right] 
			\geq \frac{C_1 C_2 e^{-\tau}}{8}
		\end{align*}
		where the last inequality follows from (\ref{eq:main1}).
		Hence, with probability $\Omega(1)$, the two chains couple at time $T_3$ and therefore the mixing time is $O(\log n)$.
	\end{proof}

    \section{Lower bounding walker trajectory probabilities: Proof of Lemma \ref{lem:te_prob_lw_bdd}}\label{sec:xxx}
    
	Consider a $d$-regular graph $G = (V, E)$. Given an edge $e \in E$ and $t' > t \geq 0$, by $\Ce (t, t')$ we denote the event that $\Ce$ rings in $(t, t')$. Furthermore, conditioned on the event $\Ce (t, t')$, let $\tau_e (t, t')$ denote the last ring time of $\Ce$ in $(t, t')$. Consider $\alpha \in \mathbb{N}$ and let $\overrightarrow{x_\alpha} = (x_0, x_1, \dots, x_{\alpha}) \in V^{\alpha + 1}$ be a walk on $G$, with loops allowed. Suppose that the first $\alpha$ walker clock rings occur at times $t_1, \dots, t_\alpha$. Let $\Delta_i \coloneqq (t_{i + 1} - t_i) \left(\frac{\varepsilon \log n}{\mu}\right)$. Note that since $\mu \geq \varepsilon \log n$, then $(t_{i + 1} - \Delta_i, t_{i + 1}) \subseteq (t_i, t_{i + 1})$.
    
    If the walker is at $x_i$ at time $t_i$ and $x_{i + 1} \neq x_i$, for the walker to transition to $x_{i + 1}$ at time $t_{i + 1}$, it suffices that the following events hold: the edge $\{x_i, x_{i + 1}\}$ refreshes during $(t_{i + 1} - \Delta_i, t_{i + 1})$ and it refreshes open during its last ring, and at time $t_{i + 1}$ the walker examines the edge $\{x_i, x_{i + 1}\}$, i.e., $V_{x_i} (t_{i+1}) = x_{i + 1}$ (and hence the walker traverses the edge since it is open). This motivates the following event definition.

    \begin{definition}[$\xi_{u \to u'} (t, t')$] \label{def:transition-event}
		Let $G = (V, E)$ be a $d$-regular graph. Let $t' > t \geq 0$ and $u, u' \in V$ such that $e \coloneqq \{u, u'\} \in E$. The event $\xi_{u \to u'} (t, t')$ is the event that all of the following hold: $V_u (t') = u'$, $\Ce (t, t')$, and $e$ is refreshed open at time $\tau_e(t, t')$.
	\end{definition}

    The following lemma, proved in Section \ref{sec:transition_event_lw_bdds}, gives a lower bound on the probability of the event $\xi_{u \to u'} (t, t')$, in terms of the edge refresh rate $\mu$ and the length of the time interval $(t, t')$. 

    \begin{restatable}{restatablelemma}{transitionLwBdd} \label{lem:transition_lw_bdd}	
		For every $d$-regular graph $G = (V, E)$ and for all $\mu > 0$, $t' > t \geq 0$, $\{u, u'\} \in E$ and $(\eta, v) \in \{0, 1\}^E \times V$, $
		\mathbb{P} \left[\xi_{u \to u'} \left(t,t'\right) \mid (\eta_t, X_t) = (\eta, v)\right] \geq ({\pmin}/{d}) (1-e^{-\mu (t' - t)})$.
	\end{restatable}

    On the other hand, suppose that the walker is at $x_i$ at time $t_i$ and that $x_{i + 1} = x_i$. For the walker to remain at $x_i$ at time $t_{i + 1}$, it suffices that for every neighbour $u_{i + 1}$ of $x_i$, if the walker examines the edge $\{x_i, u_{i + 1}\}$ at time $t_{i + 1}$ (i.e., $V_{x_i} (t_{i+1}) = u_{i + 1}$), then it holds that the edge $\{x_i, u_{i + 1}\}$ was refreshed during $(t_{i + 1} - \Delta_i, t_{i + 1})$ and it was refreshed closed during its last ring (and hence the walker remains at $x_i$ since the examined edge is closed). This motivates the following event definition.

    \begin{definition}[$\xi_{u \not\to u'} (t, t')$] \label{def:stationary-transition-event}
		Let $G = (V, E)$ be a $d$-regular graph. Let $t' > t \geq 0$ and $u, u' \in V$ such that $e \coloneqq \{u, u'\} \in E$. The event $\xi_{u \not\to u'} (t, t')$ is the event that all of the following hold: $V_u (t') = u'$, $\Ce (t, t')$, and $e$ is refreshed closed at time $\tau_e(t, t')$.
	\end{definition}

To obtain good lower bounds on the probability of $\xi_{u \not\to u'} (t, t')$, we need good lower bounds on the probability that $e = \{u, u'\}$ refreshes closed at its last refresh time $\tau_e(t, t')$ during $(t, t')$. Our analysis handles three separate cases,
depending on whether $e$ lies on a cycle, and on the length of that cycle.    In particular, we will require the following event for handling long cycles (see Section~\ref{sec:st_events} for further details of all the cases that arise).

    Given a configuration $\eta \in \{0, 1\}^E$ and a set of edges $H \subseteq E$, let $\eta^H$ be the configuration such that $\eta^H (e) = 1$ if $e \in H$, and $\eta^H (e) = \eta(e)$ if $e \in E \setminus H$. Given $v \in V$, let $E_v \coloneqq E\left(B_R (v)\right)$, and consider the continuous-time random cluster Glauber dynamics $\left(\eta_{t}^{E_v}\right)_{t \geq 0}$ with parameters $(p, q)$, that is, the dynamics where all edges in $E_v$ are wired open. We define $S_{t} \left(B_R (v), K\right)$ as the event that, in the graph $\left(V,\eta_{t}^{E_v} \setminus E_v\right)$, at most $K$ vertices  in  $\partial B_R (v)$ are in non-trivial components, and let $S_{t} (R, K) \coloneqq \cap_{v \in V} S_{t} (B_R(v), K)$. By $S_{[t', t'']} (R, K)$ we denote the event that $S_t (R, K)$ holds for all $t \in [t', t'']$. See Definition \ref{def:external-K-sparse-event} in Section \ref{sec:sparsity} for further details; we will eventually prove a slight adaptation of a result from Blanca and Gheissari (\cite{Blanca2021}, Theorem 5), which establishes that, with high probability over $G \sim \rrg$, the event $S_t (R, K)$ holds for a significantly long time interval after some burn-in time $\Cburn$.

    The following lemma, proved in Section \ref{sec:transition_event_lw_bdds}, gives a lower bound on the probability of $\xi_{u \not\to u'} (t, t')$, in terms of the edge refresh rate $\mu$, the length of the time interval $(t, t')$, and the probability of the $K$-sparsity event $S_{[t, t']} \left(R, K\right)$.

    \begin{restatable}{restatablelemma}{stationaryLwBdd} \label{lem:stationary_lw_bdd}
Let $n \in \mathbb{N}$ be sufficiently large and let $G = (V, E)$ be a $d$-regular graph on $n$ vertices. For all $\mu > 0$, $t' > t \geq 0$ such that $\mu (t' - t) > 2 \log 2$, all $e = \{u, u'\} \in E$ such that $u$ does not lie on a cycle of length less than~$r$ and $B_R(u)$ has at most one cycle, and all $(\eta, v) \in \{0, 1\}^E \times V$,
		\begin{align*}
			&\ \mathbb{P} \left[\xi_{u \not\to u'} (t, t') \mid (\eta_t, X_t) = (\eta, v)\right] \\
			\geq&\ \frac{1-\pmin}{d} \left(1 - \frac{1}{(\log n)^2} - 2(K + 1) e^{-\mu(t' - t)/2}\right) \mathbb{P}\left[S_{[t, t']} \left(R, K\right) \mid \left(\eta_t, X_t\right) = \left(\eta, v\right)\right].
		\end{align*}
	\end{restatable}

     Let $\mathcal{L}\left(\overrightarrow{x_\alpha}\right)$ be the set of indices such that $i \in \mathcal{L}\left(\overrightarrow{x_\alpha}\right)$ if, and only if, $x_i = x_{i+1}$. We call a sequence $\overrightarrow{u_\alpha} = (u_1, \dots, u_\alpha) \in V^\alpha$ a \textit{realisation} of the walk $\overrightarrow{x_\alpha}$ if, and only if, \[
	   u_{i+1} \in \begin{cases}
		\left\{u \in V\colon \{x_i, u\} \in E\right\}, & \mathrm{if} \ i \in \mathcal{L}\left(\overrightarrow{x_\alpha}\right) \\ 
		\{x_{i+1}\}, & \mathrm{otherwise}
	   \end{cases}.
	\]
	
	We denote the set of all realisations by $\mathcal{R}\left(\overrightarrow{x_\alpha}\right)$. Let $\ell\left(\overrightarrow{x_\alpha}\right) = \left|\mathcal{L}\left(\overrightarrow{x_\alpha}\right)\right|$. When there is no room for ambiguity, we will simply write $\mathcal{R}$, $\mathcal{L}$ and $\ell$. Observe that the number of realisations $|\mathcal{R}|$ for a walk $\overrightarrow{x_\alpha}$ is $d^\ell$.

    Thus, given that the walker is at $x_0$ at time $t_0$ (where $t_0 < t_1$), in order for the walker to follow the prescribed trajectory $\overrightarrow{x_\alpha} = (x_0, \dots, x_\alpha)$ at the ring times $t_1, \dots, t_\alpha$, it suffices that for some realisation $\overrightarrow{u_\alpha} \in \mathcal{R}$, for each $i \in \{0, \dots, \alpha - 1\}$, the appropriate transition event occurs between times $t_i$ and $t_{i + 1}$. Specifically, if $i \notin \mathcal{L}$, then the event $\xi_{x_i \to x_{i + 1}}(t_i, t_{i + 1})$ holds. On the other hand, if $i \in \mathcal{L}$, then the event $\xi_{x_i \not\to u_{i + 1}}(t_i, t_{i + 1})$ holds, where $u_{i + 1}$ is the neighbour examined at time $t_{i + 1}$ as specified by the realisation~$\overrightarrow{u_\alpha}$.
	
	Given a realisation $\overrightarrow{u_\alpha}$ of a walk $\overrightarrow{x_\alpha}$, $t_0 \geq 0$ and $\alpha$ times $\overrightarrow{t_\alpha} = (t_1, \dots, t_\alpha)$ such that $t_0 < t_1 < \dots < t_\alpha$, let $\Delta_i \coloneqq (t_{i + 1} - t_i) \left(\frac{\varepsilon \log n}{\mu}\right)$ for all $0 \leq i < \alpha$ and define the \textit{trajectory realisation event} \[\Xi \! \left(\overrightarrow{x_\alpha}, \overrightarrow{t_\alpha} ; \overrightarrow{u_\alpha}\right)
	\coloneqq \Biggl(\bigcap_{i \not\in \mathcal{L}} \xi_{x_i \to x_{i+1}} \left(t_{i + 1} - \Delta_i, t_{i+1}\right)\Biggr) \cap \Biggl(\bigcap_{i \in \mathcal{L}} \xi_{x_i \not\to u_{i+1}} \left(t_{i + 1} - \Delta_i, t_{i+1}\right)\Biggr).\]

    From Lemmas \ref{lem:transition_lw_bdd} and \ref{lem:stationary_lw_bdd} we can deduce the following lower bound on the probability of $\Xi\left(\overrightarrow{x_\alpha}, \overrightarrow{t_\alpha} ; \overrightarrow{u_\alpha}\right)$, which we prove in Section \ref{sec:lem4a}.

    \begin{restatable}{restatablelemma}{LemmaFourA} \label{lem:lem4a}
For all sufficiently large $n \in \mathbb{N}$, every $\mu \geq \varepsilon \log n $, every $\Tmax \geq T > 0$ and $t_0 > 0$, every $d$-regular $n$-vertex graph $G$ such that every radius-$R$ ball has at most one cycle, every $\alpha \in \mathbb{N}$, every $r$-acyclic walk $\overrightarrow{x_\alpha} \in V(G)^{\alpha + 1}$ and realisation $\overrightarrow{u_{\alpha}} \in \mathcal{R}$, every $\overrightarrow{t_\alpha} = (t_1, t_2, \dots, t_\alpha)$ satisfying $t_0 < t_1 < t_2 < \dots < t_\alpha < t_0 + T$ and $(t_{i + 1} - t_i) (\varepsilon \log n) > 2 \log 2$, and every $\eta \in \{0, 1\}^{E(G)}$, 
        \begin{align*}
            &\ \mathbb{P} \left[\Xi\left(\overrightarrow{x_\alpha}, \overrightarrow{t_\alpha} ; \overrightarrow{u_\alpha}\right) \mid (\eta_0, X_0) = (\eta, x_0)\right] \\
            \geq& \frac{\pmin^{\alpha - \ell} (1 - \pmin)^\ell}{d^\alpha} \prod_{i = 0}^{\alpha - 1} \left(1 - \tfrac{1}{(\log n)^2} - \frac{2(K+1)}{n^{(\varepsilon / 2) T_i}}\right) \! - \sum_{j \in \mathcal{L}} \mathbb{P}\left[\neg S_{[t_{j + 1} - \Delta_j, t_{j + 1}]} (R, K) \mid (\eta_0, X_0) = (\eta, x_0)\right]
        \end{align*}
        where, for all $0 \leq i \leq \alpha - 1$, $T_i \coloneqq t_{i + 1} - t_i$ and $\Delta_i \coloneqq T_i \left(\frac{\varepsilon \log n}{\mu}\right)$.
    \end{restatable}

    We define the \textit{trajectory event} $\Xi\left(\overrightarrow{x_\alpha}, \overrightarrow{t_\alpha}\right)
	\coloneqq \bigcup_{\overrightarrow{u_\alpha} \in \mathcal{R}} \Xi\left(\overrightarrow{x_\alpha}, \overrightarrow{t_\alpha} ; \overrightarrow{u_\alpha}\right)$. We note the following remark.

    \begin{remark} \label{rem:te_disjoint}
    For any two distinct realisations $\overrightarrow{u_\alpha}$ and $\overrightarrow{w_\alpha}$ of a walk $\overrightarrow{x_\alpha}$, the trajectory realisation events $\Xi\left(\overrightarrow{x_\alpha}, \overrightarrow{t_\alpha} ; \overrightarrow{u_\alpha}\right)$ and $\Xi\left(\overrightarrow{x_\alpha}, \overrightarrow{t_\alpha} ; \overrightarrow{w_\alpha}\right)$ are disjoint, since there is at least one position $i \in \mathcal{L}$ such that $u_{i+1} \neq w_{i+1}$, and therefore the events $V_{x_i}(t_{i+1}) = u_{i+1}$ and $V_{x_i}(t_{i+1}) = w_{i+1}$ are disjoint.
    \end{remark}
	
	Recall from Definition \ref{def:walker-trajectory-event} that $\widehat\Xi \left(\overrightarrow{x_\alpha}, t_0, T\right)$ is the event that $\Cw$ rings exactly $\alpha$ times during $(t_0, t_0 + T)$ and the walker follows the walk $\overrightarrow{x_\alpha}$. By our previous discussions, if the walker is at the initial vertex $x_0$ of the walk $\overrightarrow{x_\alpha}$ at time $t_0$, then if $\Cw$ rings exactly $\alpha$ times during $(t_0, t_0 + T)$ and for the $\alpha$ ring times $\overrightarrow{t_\alpha}$ in $(t_0, t_0 + T)$ of $\Cw$ the event $\Xi\left(\overrightarrow{x_\alpha}, \overrightarrow{t_\alpha}\right)$ holds, it follows that $\widehat\Xi \left(\overrightarrow{x_\alpha}, t_0, T\right)$ holds and at time $t_0 + T$ the walker is at the end vertex $x_\alpha$ of the walk, i.e., $X_{t_0 + T} = x_\alpha$. Consequently, Lemma \ref{lem:lem4a} will serve as the basis of the proof of Lemma \ref{lem:te_prob_lw_bdd}.

To pass from fixed ring times of the walker clock to the event that the clock rings $\alpha$ times in an interval of length $T$, we must average the fixed-time trajectory bounds over the simplex of possible walker clock ring times. The following estimate, key to the proof of Lemma \ref{lem:te_prob_lw_bdd}, says that after excluding very short gaps between successive walker clock rings, the loss from requiring examined edges to refresh during each gap is only  a constant factor. Lemma \ref{lem:simplex_integral} is proved in Appendix \ref{appendix:A}.
	
	\begin{restatable}{restatablelemma}{simplexIntegralLemma} \label{lem:simplex_integral}
Fix $C > 0$. There exists $c > 0$ such that the following holds: For all $c' > c$, all sufficiently large $n \in \mathbb{N}$, 
all
$T \geq ({1}/{(50 \ \pmin)}) \log_{d - 1} n$, and all $\alpha \in \mathbb{N}$ such that $\alpha \leq ({4}/{\pmin}) \log_{d - 1} n$, 
with $\delta \coloneqq {c'}/{(\log n)}$, 
\[\int_{\substack{T_0, \dots, T_\alpha \geq\delta \\ \sum_{i = 0}^\alpha T_i = T}} \ \prod_{i=0}^\alpha \left(1 - C n^{-(\varepsilon / 2) T_i}\right) \ \mathrm{d} \overrightarrow{T_\alpha} \geq \frac{1}{4} \frac{T^{\alpha}}{\alpha!} \left(1 - \frac{(\alpha + 1) \delta}{T}\right)^\alpha.\]
	\end{restatable}
	
	\begin{remark} \label{times_to_gaps_simplex}
		Given $t_0, T > 0$ and $\alpha \in \mathbb{N}$, the simplex $\{(t_1, \dots, t_\alpha) \colon t_0 < t_1 < \dots < t_\alpha < t_0 + T\}$ is naturally identified with the simplex $\{T_0, \dots, T_\alpha > 0 \colon \sum\limits_{i = 0}^\alpha T_i = T\}$ via the change of variables $T_i \coloneqq t_{i+1} - t_i$ for all $i \in \{0, 1, \dots, \alpha - 1\}$ and $T_\alpha = t_0 + T - t_\alpha$, which is linear with unit Jacobian and hence preserves $\alpha$-dimensional volume.
	\end{remark}
	
	We are now in a position to prove our main result for this section.
	
	\teProbLwBdd*
	
	\begin{proof}
		Let $\amax \coloneqq \lfloor ({4}/{\pmin}) \log_{d - 1} n\rfloor$ and $\Tmax \coloneqq ({4}/{\pmin}) \log_{d - 1} n$. By Lemma \ref{lem:sparse_sup_prob}, with probability $1 - O\left(n^{-1}\right)$ over $G \sim \rrg$, for all $\mu \geq \varepsilon \log n $, for all $t' > t \geq \Cburn$ such that $t' - t \leq \Tmax$ and for all $(\eta, v) \in \{0, 1\}^E \times V$,
		\begin{equation} \label{eq:wtp1}
			\mathbb{P} \left[\neg S_{[t' - \Delta, t']} (R, K) \mid (\eta_0, X_0) = (\eta, v)\right] = O\left(\tfrac{1}{n} \left(\tfrac{\pmin}{4 d}\right)^{\amax}\right)
		\end{equation}
		where $\Delta \coloneqq (t' - t) \left(\frac{\varepsilon \log n}{\mu}\right)$, noting that the environment dynamics evolve independently of the walker history. Furthermore, by Lemma 2.1 in \cite{Lubetzky_2010}, we have that, with probability $1 - o_n(1)$ over $G \sim \rrg$, every radius-$R$ ball in $G$ contains at most one cycle. 
        
        Hence, with probability $1 - o_n(1)$ over $G \sim \rrg$, both of these properties hold. Fix such a graph $G$. Fix $\eta \in \{0, 1\}^E$, $\mu \geq \varepsilon \log n $, $T > 0$ such that $({1}/{(50 \ \pmin)}) \log_{d - 1} n \leq T \leq \Tmax$, $\alpha \in \mathbb{N}$ such that $\alpha \leq \amax$, and let $\overrightarrow{x_\alpha} = (x_0, \dots, x_\alpha) \in V^{\alpha + 1}$ be an $r$-acyclic walk. 

        Let $\mathcal{N}_0$ and $\mathcal{N}_\alpha$ denote the events $N_\mathrm{w}[0, \Cburn] = 0$ and $N_\mathrm{w}(\Cburn, \Cburn + T) = \alpha$, respectively. These events are independent since they concern the ringing of the same Poisson clock over disjoint time intervals. Conditioned on the event $\mathcal{N}_0$, the walker does not move during $[0, \Cburn]$, and hence $X_{\Cburn} = x_0$. If, furthermore, the event $\mathcal{N}_\alpha$ holds and the $\alpha$ ordered ring times of $\Cw$ in $(\Cburn, \Cburn + T)$ are given by the vector $\overrightarrow{\tau_\alpha}$, then the event $\Xi\left(\overrightarrow{x_\alpha}, \overrightarrow{\tau_\alpha}\right)$ forces the walker to follow the trajectory $\overrightarrow{x_\alpha}$ during $(\Cburn, \Cburn + 
        T)$ such that $X_{\Cburn + T} = x_\alpha$. Therefore these events imply $\widehat\Xi \left(\overrightarrow{x_\alpha}, \Cburn, T\right)$.
		
		Conditioned on $\mathcal{N}_\alpha$, the $\alpha$ ordered ring times of $\Cw$ occurring during $(\Cburn, \Cburn + T)$ are uniformly distributed on the simplex $\mathfrak{T}_\alpha \coloneqq \{(t_1, \dots, t_\alpha) \colon \Cburn < t_1 < \dots < t_\alpha < \Cburn + T\}$ which has volume $\mathrm{Vol}\left(\mathfrak{T}_\alpha\right) = {T^\alpha} / {\alpha!}$. We therefore have that
		\begin{align*} 
			&\ \mathbb{P}\left[\widehat\Xi \left(\overrightarrow{x_\alpha}, \Cburn, T\right) \mid (\eta_0, X_0) = (\eta, x_0)\right] \nonumber \\
            \geq&\ \mathbb{P}\left[\mathcal{N}_0, \mathcal{N}_\alpha \mid (\eta_0, X_0) = (\eta, x_0)\right] \frac{\alpha!}{T^{\alpha}} \!\! \int\limits_{\overrightarrow{t_\alpha} \in \mathfrak{T}_\alpha} \!\!\! \mathbb{P} \left[\Xi(\overrightarrow{x_\alpha}, \overrightarrow{t_\alpha}) \mid \overrightarrow{\tau_\alpha} = \overrightarrow{t_\alpha}, \mathcal{N}_0, \mathcal{N}_\alpha, (\eta_0, X_0) = (\eta, x_0)\right] \mathrm{d} \overrightarrow{t_\alpha}.
		\end{align*}

        For fixed $\overrightarrow{t_\alpha} \in \mathfrak{T}_\alpha$, the event $\Xi(\overrightarrow{x_\alpha}, \overrightarrow{t_\alpha})$ is independent of the walker clock, and therefore we have that $\mathbb{P} \left[\Xi(\overrightarrow{x_\alpha}, \overrightarrow{t_\alpha}) \mid \overrightarrow{\tau_\alpha} = \overrightarrow{t_\alpha}, \mathcal{N}_0, \mathcal{N}_\alpha, (\eta_0, X_0) = (\eta, x_0)\right] = \mathbb{P} [\Xi(\overrightarrow{x_\alpha}, \overrightarrow{t_\alpha}) \mid (\eta_0, X_0) = (\eta, x_0)]$. Furthermore, since $\Cburn$ is constant, there exists some constant $C_1 \in (0, 1)$ such that $\mathbb{P}[\mathcal{N}_0] \geq C_1$ and hence, by independence of $\mathcal{N}_0$ and $\mathcal{N}_\alpha$, we have that $\mathbb{P}\left[\mathcal{N}_0, \mathcal{N}_\alpha \mid (\eta_0, X_0) = (\eta, x_0)\right] \geq C_1 \mathbb{P}\left[\mathcal{N}_\alpha\right]$.

        Combining everything together, it follows that
        \begin{equation}
            \mathbb{P}\left[\widehat\Xi \left(\overrightarrow{x_\alpha}, \Cburn, T\right) \mid (\eta_0, X_0) = (\eta, x_0)\right] \geq C_1 \mathbb{P}\left[\mathcal{N}_\alpha\right] \frac{\alpha!}{T^{\alpha}} \!\!\! \int\limits_{\overrightarrow{t_\alpha} \in \mathfrak{T}_\alpha} \!\!\! \mathbb{P} [\Xi(\overrightarrow{x_\alpha}, \overrightarrow{t_\alpha}) \mid (\eta_0, X_0) = (\eta, x_0)] \mathrm{d} \overrightarrow{t_\alpha}. \label{eq:wtp2}
        \end{equation}
		
		It will be convenient to define $t_{\alpha + 1} \coloneqq \Cburn + T$, and, given $\overrightarrow{t_\alpha} \in \mathfrak{T}_\alpha$, to define, for all $0 \leq i \leq \alpha$, $T_i \coloneqq t_{i + 1} - t_i$ and $\Delta_i \coloneqq T_i \left(\frac{\varepsilon \log n}{\mu}\right)$. By Lemma \ref{lem:simplex_integral}, there exists $c > 0$ such that for all $c' > c$, all $n$ sufficiently large, with $\delta \coloneqq \frac{c'}{\log n}$, 
		\begin{equation}
			\int_{\substack{T_0, \dots, T_\alpha \geq\delta \\ \sum_{i = 0}^\alpha T_i = T}} \ \prod_{i=0}^\alpha \left(1 - 4(K+1) n^{-(\varepsilon / 2) T_i}\right) \ \mathrm{d} \overrightarrow{T_\alpha} \geq \frac{1}{4} \frac{T^{\alpha}}{\alpha!} \left(1 - \frac{(\alpha + 1) \delta}{T}\right)^\alpha. \label{eq:wtp10}
		\end{equation}
		
		In particular, we will choose $c'$ sufficiently large such that $\delta (\varepsilon \log n) > 2 \log (8(K + 1))$. Let $\widetilde{\mathfrak{T}}_\alpha$ be the region of the simplex $\mathfrak{T}_\alpha$ such that $(t_1, \dots, t_\alpha) \in \mathfrak{T}_\alpha$ is in $\widetilde{\mathfrak{T}}_\alpha$ if, and only if, $\min\limits_{0 \leq i \leq \alpha} T_i \geq \delta$.
		
		Next fix a realisation $\overrightarrow{u_\alpha}$ and $\overrightarrow{t_\alpha} \in \widetilde{\mathfrak{T}}_\alpha$. For this choice of $\overrightarrow{t_\alpha}$ it follows that, for all $0 \leq i < \alpha$, since $T_i (\varepsilon \log n) > 2 \log (8(K + 1))$, then $T_i (\varepsilon \log n) > 2 \log 2$. By Lemma \ref{lem:lem4a} we have that,
        \begin{align}
            &\ \mathbb{P} \left[\Xi\left(\overrightarrow{x_\alpha}, \overrightarrow{t_\alpha} ; \overrightarrow{u_\alpha}\right) \mid (\eta_0, X_0) = (\eta, x_0)\right] \nonumber \\
            \geq& \frac{\pmin^{\alpha - \ell} (1 - \pmin)^\ell}{d^\alpha} \prod_{i = 0}^{\alpha - 1} \left(1 - \tfrac{1}{(\log n)^2} - \frac{2(K+1)}{n^{(\varepsilon / 2) T_i}}\right) \! - \sum_{j \in \mathcal{L}} \mathbb{P}\left[\neg S_{[t_{j + 1} - \Delta_j, t_{j + 1}]} (R, K) \mid (\eta_0, X_0) = (\eta, x_0)\right]. \label{eq:wtp6}
        \end{align}

        From (\ref{eq:wtp1}) we have, for all $j \in \mathcal{L}$, $\mathbb{P}\left[\neg S_{[t_{j + 1} - \Delta_j, t_{j + 1}]} (R, K) \mid (\eta_0, X_0) = (\eta, x_0)\right] = O\left(\frac{1}{n} \left(\frac{\pmin}{4 d}\right)^{\amax}\right)$. Furthermore, $|\mathcal{L}| \leq \alpha = O(\log n)$. Hence, in conjunction with (\ref{eq:wtp6}),
        \begin{align}
            & \ \mathbb{P} \left[\Xi\left(\overrightarrow{x_\alpha}, \overrightarrow{t_\alpha} ; \overrightarrow{u_\alpha}\right) \mid (\eta_0, X_0) = (\eta, x_0)\right] \nonumber \\
            \geq& \ \frac{\pmin^{\alpha - \ell} (1 - \pmin)^\ell}{d^\alpha} \prod_{i = 0}^{\alpha - 1} \left(1 - \tfrac{1}{(\log n)^2} - \frac{2(K+1)}{n^{(\varepsilon / 2) T_i}}\right) - \left(\frac{\pmin}{4d}\right)^{\amax} O\left(\tfrac{\log n}{n}\right). \label{eq:wtp7}
        \end{align}
		
		For all $0 \leq i < \alpha$, since $T_i (\varepsilon \log n) > 2 \log (8(K + 1))$ then \[2(K + 1)n^{- (\varepsilon / 2) T_i} = 2 (K + 1) e^{- T_i (\varepsilon / 2) \log n} \leq 1/4\] and hence for all $n$ sufficiently large we have that $1 - (\log n)^{-2} - 2(K + 1)e^{-\frac{\mu T_i}{2}} \geq 3/4 - (\log n)^{-2} \geq 1/4$. Furthermore, since $d \geq 3$ and $\pmin < \frac{1}{d-1}$, it follows that $1 - \pmin > \pmin$. Therefore, for all $n$ sufficiently large, we have that 
           $\pmin^{\alpha - \ell} (1 - \pmin)^\ell\prod_{i = 0}^{\alpha - 1}  \left(1 - \frac{1}{(\log n)^2} - 2(K+1)e^{-\frac{\mu T_i}{2}}\right) \geq \left(\pmin / 4\right)^{\amax}$. Hence from (\ref{eq:wtp7}), for all $n$ sufficiently large,
		\begin{align}
			\mathbb{P} \left[\Xi\left(\overrightarrow{x_\alpha}, \overrightarrow{t_\alpha} ; \overrightarrow{u_\alpha}\right) \mid (\eta_0, X_0) = (\eta, x_0)\right] \geq& (1 - O(\tfrac{\log n}{n})) \frac{\pmin^{\alpha - \ell} (1 - \pmin)^\ell}{d^\alpha} \prod_{i = 0}^{\alpha - 1} \left(1 - \tfrac{1}{(\log n)^2} - \frac{2(K+1)}{n^{(\varepsilon / 2) T_i}}\right) \nonumber \\
            >& \ \frac{1}{2} \frac{\pmin^{\alpha - \ell} (1 - \pmin)^\ell}{d^\alpha} \prod_{i = 0}^{\alpha - 1} \left(1 - \tfrac{1}{(\log n)^2} - \frac{2(K+1)}{n^{(\varepsilon / 2) T_i}}\right). \label{eq:wtp8}
		\end{align}
		
		By the arbitrariness of the realisation $\overrightarrow{u_\alpha}$ and the disjointness of trajectory realisation events from Remark \ref{rem:te_disjoint}, from (\ref{eq:wtp8}) we obtain that for all $n$ sufficiently large
		\begin{align}
			\mathbb{P} \left[\Xi\left(\overrightarrow{x_\alpha}, \overrightarrow{t_\alpha}\right) \mid (\eta_0, X_0) = (\eta, x_0)\right] \nonumber \geq& \! \sum_{\overrightarrow{u_\alpha} \in \mathcal{R}} \frac{1}{2} \frac{\pmin^{\alpha - \ell} (1 - \pmin)^\ell}{d^\alpha} \prod_{i = 0}^{\alpha - 1} \left(1 - \tfrac{1}{(\log n)^2} - \frac{2(K+1)}{n^{(\varepsilon / 2) T_i}}\right) \nonumber \\
			\geq& \ \frac{1}{4} \frac{\pmin^{\alpha - \ell} (1 - \pmin)^\ell}{d^{\alpha - \ell}} \prod_{i = 0}^{\alpha - 1} \left(1 - \frac{4(K+1)}{n^{(\varepsilon / 2) T_i}}\right)\label{eq:wtp9}
		\end{align}
		where the last inequality follows from the facts that $|\mathcal{R}| = d^\ell$, and that for all $x > 0$ and all $n$ sufficiently large the inequality $1 - \frac{1}{(\log n)^2} - 2(K+1)n^{-x} \geq \left(1 - \frac{1}{(\log n)^2}\right) \left(1 - 4(K+1)n^{-x} \right)$ holds and since $\alpha = O(\log n)$, then for all $n$ sufficiently large we have that $\left(1 - \frac{1}{(\log n)^2}\right)^\alpha > \frac{1}{2}$.
        
        In conjunction with (\ref{eq:wtp10}) and Remark \ref{times_to_gaps_simplex}, from (\ref{eq:wtp9}) we have that
		\begin{align}
			\int\limits_{\overrightarrow{t_\alpha} \in \mathfrak{T}_\alpha} \mathbb{P} \left[\Xi\left(\overrightarrow{x_\alpha}, \overrightarrow{t_\alpha}\right) \mid (\eta_0, X_0) = (\eta, x_0)\right] \mathrm{d} \overrightarrow{t_\alpha} \geq&\ \frac{1}{4} \frac{\pmin^{\alpha - \ell} (1 - \pmin)^\ell}{d^{\alpha - \ell}} \!\! \int\limits_{\overrightarrow{t_\alpha} \in \mathfrak{T}_\alpha} \!\! \prod_{i = 0}^{\alpha - 1} \left(1 - \frac{4(K+1)}{n^{(\varepsilon / 2) T_i}}\right) \mathrm{d} \overrightarrow{t_\alpha} \nonumber \\
			\geq& \ \frac{1}{16} \frac{T^\alpha}{\alpha!} \left(1 - \frac{(\alpha + 1) \delta}{T}\right)^\alpha \frac{\pmin^{\alpha - \ell} (1 - \pmin)^\ell}{d^{\alpha - \ell}}. \label{eq:wtp11}
		\end{align} 

        Since $\alpha \leq (4/\pmin) \log_{d-1} n, T \geq (1/(50 \ \pmin) \log_{d - 1} n)$ and $\delta = c' / \log n = (c' \log_{d-1} e) / \log_{d-1} n$, $\left(1 - \frac{(\alpha + 1) \delta}{T}\right)^\alpha \geq \left(1 - \left(\frac{c' \log_{d-1} e}{(1/(50 \ \pmin))}\right) \left(\frac{1 + (4/\pmin) \log_{d-1} n}{(\log_{d-1} n)^2}\right)\right)^{\amax} \geq \left(1 - \frac{400 c' \log_{d-1} e}{\log_{d-1} n}\right)^{(4/\pmin) \log_{d-1} n}$ for $n$ sufficiently large. Hence $\left(1 - \frac{(\alpha + 1) \delta}{T}\right)^\alpha$ is lower bounded by some constant $C_2 \in (0, 1)$ for $n$ sufficiently large. Therefore from (\ref{eq:wtp11}) it follows that for all $n$ sufficiently large,
        \begin{align}
			\frac{\alpha!}{T^{\alpha}} \int\limits_{\overrightarrow{t_\alpha} \in \mathfrak{T}_\alpha} \mathbb{P} \left[\Xi\left(\overrightarrow{x_\alpha}, \overrightarrow{t_\alpha}\right) \mid (\eta_0, X_0) = (\eta, x_0)\right] \mathrm{d} \overrightarrow{t_\alpha} \geq (C_2 / 16) \frac{\pmin^{\alpha - \ell} (1 - \pmin)^\ell}{d^{\alpha - \ell}}. \label{eq:wtp12}
		\end{align}
        
        The result follows by substituting (\ref{eq:wtp12}) in (\ref{eq:wtp2}).
	\end{proof}

    \subsection{Proof of Lemma \ref{lem:lem4a}} \label{sec:lem4a}

    \LemmaFourA*
    \begin{proof}
        We will introduce some compact notation. For all $i \in \mathcal{L}$, let $\xi_i \coloneqq \xi_{x_i \not\to u_{i+1}} \left(t_{i + 1} - \Delta_i, t_{i+1}\right)$ and $\gamma_i \coloneqq \frac{1-\pmin}{d} \left(1 - (\log n)^{-2} - 2(K + 1)n^{-(\varepsilon / 2) T_i}\right)$. For all $i \notin \mathcal{L}$, let $\xi_i \coloneqq \xi_{x_i \to x_{i+1}} \left(t_{i + 1} - \Delta_i, t_{i+1}\right)$ and $\gamma_i \coloneqq \frac{\pmin}{d} \left(1 - n^{-\varepsilon T_i}\right)$. For all $0 \leq i < \alpha$, let $\xi_{[i]} = \bigcap_{j = 0}^i \xi_j$. Observe that $\Xi\left(\overrightarrow{x_\alpha}, \overrightarrow{t_\alpha} ; \overrightarrow{u_\alpha}\right) = \xi_{[\alpha - 1]}$.
		
		By the properties of the graph $G$ and walk $\overrightarrow{x_\alpha}$, the radius-$R$ ball  around every vertex in the walk has at most one cycle and does not lie on a cycle of length less than $r$. Furthermore, for all $0 \leq i \leq \alpha - 1$, since $T_i (\varepsilon \log n) > 2 \log 2$ then $\mu \Delta_i > 2 \log 2$. Therefore, by Lemma \ref{lem:transition_lw_bdd} we have that if $i \notin \mathcal{L}$ then
		\begin{align*}
			\mathbb{P} \left[\xi_{[i]} \mid (\eta_0, X_0) = (\eta, x_0)\right] &= \mathbb{P} \left[\xi_i \mid \xi_{[i-1]}, (\eta_0, X_0) = (\eta, x_0)\right] \mathbb{P} \left[\xi_{[i-1]} \mid (\eta_0, X_0) = (\eta, x_0)\right] \\
			&\geq\ \gamma_i \ \mathbb{P} \left[\xi_{[i-1]} \mid (\eta_0, X_0) = (\eta, x_0)\right].
		\end{align*}
		
		Otherwise by Lemma \ref{lem:stationary_lw_bdd}, if $i \in \mathcal{L}$ we have that
		\begin{align*}
			&\ \mathbb{P} \left[\xi_{[i]} \mid (\eta_0, X_0) = (\eta, x_0)\right] \\
			=&\ \mathbb{P} \left[\xi_i  \mid \xi_{[i-1]}, (\eta_0, X_0) = (\eta, x_0)\right] \mathbb{P} \left[\xi_{[i-1]} \mid (\eta_0, X_0) = (\eta, x_0)\right] \\
			\geq&\ \gamma_i \ \mathbb{P} \left[S_{[t_{i + 1} - \Delta_i, t_{i+1}]} (R, K), \xi_{[i-1]} \mid (\eta_0, X_0) = (\eta, x_0)\right] \\
			\geq&\ \gamma_i \bigl(\mathbb{P} \left[\xi_{[i-1]} \mid (\eta_0, X_0) = (\eta, x_0)\right]\! - \mathbb{P}\left[\neg S_{[t_{i + 1} - \Delta_i, t_{i+1}]} (R, K) \mid (\eta_0, X_0) = (\eta, x_0)\right]\bigr).
		\end{align*}
		
        Therefore we inductively have that 
		\begin{equation*}
			\mathbb{P} \left[\xi_{[\alpha - 1]} \mid (\eta_0, X_0) = (\eta, x_0)\right] \geq \prod_{i = 0}^{\alpha - 1} \gamma_i - \sum_{j \in \mathcal{L}} \left(\prod_{i = j}^{\alpha - 1} \gamma_i \right) \mathbb{P}\left[\neg S_{[t_{j + 1} - \Delta_j, t_{j + 1}]} (R, K) \mid (\eta_0, X_0) = (\eta, x_0)\right].
		\end{equation*}
        
        Since $\gamma_i \in [0, 1]$ for all $0 \leq i \leq \alpha - 1$, we obtain that
        \begin{equation} \label{eq:l4a1}
			\mathbb{P} \left[\xi_{[\alpha - 1]} \mid (\eta_0, X_0) = (\eta, x_0)\right] \geq \prod_{i = 0}^{\alpha - 1} \gamma_i - \sum_{j \in \mathcal{L}} \mathbb{P}\left[\neg S_{[t_{j + 1} - \Delta_j, t_{j + 1}]} (R, K) \mid (\eta_0, X_0) = (\eta, x_0)\right].
		\end{equation}

        Now, for all $n$ sufficiently large, since for all $0 \leq i \leq \alpha - 1$ we have that $T_i \varepsilon \log n > 2 \log 2$, then $1 - n^{-\varepsilon T_i} \geq 1 - n^{-(\varepsilon / 2) T_i} \geq 1 - (\log n)^{-2} - 2(K+1) n^{-(\varepsilon / 2) T_i}$. Then for all $i \notin \mathcal{L}$ it follows that $\gamma_i = \frac{\pmin}{d} \left(1 - n^{-\varepsilon T_i}\right) \geq \frac{\pmin}{d} \left(1 - (\log n)^{-2} - 2(K+1)n^{-(\varepsilon / 2) T_i}\right)$. By this lower bound on $\gamma_i$ when $i \notin \mathcal{L}$ and the definition of $\gamma_i$ when $i \in \mathcal{L}$, from (\ref{eq:l4a1}) it follows that
        \begin{align*}
            &\ \mathbb{P} \left[\Xi\left(\overrightarrow{x_\alpha}, \overrightarrow{t_\alpha} ; \overrightarrow{u_\alpha}\right) \mid (\eta_0, X_0) = (\eta, x_0)\right] \\
            \geq& \frac{\pmin^{\alpha - \ell} (1 - \pmin)^\ell}{d^\alpha} \prod_{i = 0}^{\alpha - 1} \left(1 - \tfrac{1}{(\log n)^2} - \frac{2(K+1)}{n^{(\varepsilon / 2) T_i}}\right) \! - \sum_{j \in \mathcal{L}} \mathbb{P}\left[\neg S_{[t_{j + 1} - \Delta_j, t_{j + 1}]} (R, K) \mid (\eta_0, X_0) = (\eta, x_0)\right]
        \end{align*}
        as required.
    \end{proof}

	\section{Coupling success probability lower bounds: Lemmas \ref{lem:phase1_bdd} and \ref{lem:phase2_bdd}}
    \label{sec:te_full_lw_bdd}\label{sec:phase1_bdd}\label{sec:phase2_bdd}
	
In this section we will prove the main technical results for our coupling, namely Lemmas \ref{lem:phase1_bdd} and \ref{lem:phase2_bdd}. We will first prove Lemma \ref{lem:te_full_prob_lw_bdd} --- Lemmas \ref{lem:phase1_bdd} and Lemma \ref{lem:phase2_bdd} are 
applications of this lemma. Lemma \ref{lem:te_full_prob_lw_bdd} says that if a set $S$ of vertices has a positive fraction $\zeta \in (0, 1)$ of $r$-acyclic $(h, i)$-constrained walks from a vertex $u$, for suitable values of $h$ and $i$, then 
with probability proportional to~$\zeta$, 
after an appropriate amount of time, 
the walker transitions from $u$ to $S$.
    
In order to prove Lemma \ref{lem:te_full_prob_lw_bdd} (and hence Lemmas \ref{lem:phase1_bdd} and \ref{lem:phase2_bdd}), we first require the following two technical lemmas which are proved in  Appendix \ref{appendix:B}. Lemma \ref{lem:f_alpha_properties} will be used to show that the dominant contribution in the lower bound of Lemma \ref{lem:te_full_prob_lw_bdd} comes from walks of length $\Theta(\log n)$ which are $(h, i)$-constrained for ``typical" values of $h$ and $i$, namely values in some window of order $\Theta(\sqrt{\log n})$. Lemma \ref{lem:poisson_concentration} ensures that the walker clock has constant probability of ringing the desired number of times in the desired time window.
	
	\begin{restatable}{restatablelemma}{fAlphaLem} \label{lem:f_alpha_properties}
		Let $a, b\in (0, 1)$ such that $a + b < {1}/{2}$ and $\alpha \in \mathbb{N}$. For all $x, y \in \mathbb{R}_+$ such that $x + y < \alpha$, define: 
		\[f_\alpha(x,y) \coloneqq x \log\left({a}/{x}\right) + y\log \left({b}/{y}\right) + (\alpha - x - y)\log\left(\frac{1-a-b}{\alpha - x - y}\right).\]
		Then,
		\begin{enumerate}[(i)]
			\item $f_\alpha(x, y)$ attains a maximum at $(x, y) = \left(a \alpha, b \alpha\right)$.
			\item For $\alpha$ sufficiently large, for all $0 \leq x, y \leq \sqrt{\alpha}$, $f_\alpha\left(a \alpha, b \alpha\right) - f_\alpha\left(a (\alpha + x), b (\alpha + y)\right) \leq \frac{a+b}{2(1-a-b)}$.
		\end{enumerate}	
	\end{restatable}
	
	\begin{restatable}{restatablelemma}{poissonConcentration} \label{lem:poisson_concentration}
		Let $\Cmax > \Cmin > 0$ and $\Cmid > 0$ such that $\Cmid \leq \Cmax - \Cmin$. Let $n \in \mathbb{N}$ be sufficiently large and let $\kmin, \kmax \in \mathbb{N}$ such that $\Cmin \log n \leq \kmin < \kmax \leq \Cmax \log n$ and $\kmax - \kmin \geq \Cmid \log n$. Let $Y \sim \mathsf{Poisson}(1)$ and let $N_Y (T)$ denote the number of rings of $Y$ during an interval of length $T > 0$. Then, $\mathbb{P}\left[\kmin \leq N_Y \left(\frac{\kmin + \kmax}{2}\right) \leq \kmax\right] \geq \frac{1}{2}$.
	\end{restatable}
	
	We are now in a position to prove Lemma \ref{lem:te_full_prob_lw_bdd}.

    \begin{lemma} \label{lem:te_full_prob_lw_bdd}
		There exists a constant $C \in (0, 1)$ such that for all $n \in \mathbb{N}$ sufficiently large, with probability $1 - o_n(1)$ over $G \sim \rrg$, for all $\mu \geq \varepsilon \log n $, and for every $S \subseteq V$ and $u \in V$ such that there exists $\amin, \amax \in \mathbb{N}$ and $\zeta \in (0, 1)$ satisfying the following:
		\begin{enumerate}[(i)]
			\item $({1}/{(50 \ \pmin)}) \log_{d - 1} n \leq \amin < \amax \leq ({4}/{\pmin}) \log_{d - 1} n$ and $\amax - \amin \geq \frac{\log_{d-1} n}{100 \ \pmin}$
			\item for all $\alpha, h, i \in \mathbb{N}$ such that $\alpha \in [\amin, \amax]$, $h \in \left[\left((d - 2) / d\right) \pmin \alpha, \left((d - 2) / d\right) \pmin \left(\alpha + \sqrt{\alpha}\right)\right]$ and $i \in \left[(\pmin / d) \alpha, (\pmin / d) \left(\alpha + \sqrt{\alpha}\right)\right]$, there are at least $\zeta \ \widetilde{\omega}_{h, i}$ walks from $u$ to a vertex in $S$ which are $r$-acyclic and are $(h, i)$-constrained walks
		\end{enumerate}
		then for all $\eta \in \{0, 1\}^E$, $\mathbb{P} \left[X_{\Cburn + \frac{\amax + \amin}{2}} \in S \mid \left(\eta_0, X_0\right) = (\eta, u)\right] \geq C \zeta$.
    \end{lemma}
	
	\begin{proof}
		By Lemma \ref{lem:te_prob_lw_bdd}, there exists $C_0 \in (0, 1)$ such that for $n \in \mathbb{N}$ sufficiently large, with probability $1 - o_n(1)$ over $G = (V, E) \sim \rrg$, for every $\mu \geq \varepsilon \log n $, $T > 0$ such that $({1}/{(50 \ \pmin)}) \log_{d - 1} n \leq T \leq ({4}/{\pmin}) \log_{d - 1} n$ and $\alpha \in \mathbb{N}$ such that $\alpha \leq ({4}/{\pmin}) \log_{d - 1} n$, and for all $r$-acyclic walks $\overrightarrow{x_\alpha} \in V^{\alpha + 1}$ and $\eta \in \{0, 1\}^E$,
		\begin{equation} \label{eq:dcw0}
			\mathbb{P} \left[\widehat\Xi\left(\overrightarrow{x_\alpha}, \Cburn, T\right) \mid (\eta_0, X_0) = (\eta, x_0)\right] \geq C_0 \mathbb{P}\left[N_\mathrm{w}(\Cburn, \Cburn + T) = \alpha\right] \frac{\pmin^{\alpha - \ell}(1 - \pmin)^{\ell}}{d^{\alpha - \ell}}. 
		\end{equation}
		
		Fix such a graph $G$ and $\mu \geq\varepsilon \log n $. Let $S \subseteq V$ and $u \in V$ such that there exists $\amin, \amax \in \mathbb{N}$ and $\zeta \in (0, 1)$ satisfying conditions (i) and (ii) in the theorem. For convenience, for all $x > 0$ define the functions
		\begin{enumerate}[i.]
			\item $\hmin(x) \coloneqq \left(\frac{d-2}{d}\right) \pmin x$ and $\hmax(x) \coloneqq \left(\frac{d-2}{d}\right) \pmin \left(x + \sqrt{x}\right)$; 
			\item $\imin(x) \coloneqq \frac{\pmin}{d} x$ and $\imax(x) \coloneqq \frac{\pmin}{d} \left(x + \sqrt{x}\right)$.
		\end{enumerate}
		
		Fix $T \coloneqq \frac{\amin + \amax}{2}$, and for all $\alpha \in \mathbb{N}$ such that $\alpha \in [\amin, \amax]$, let \[\Range{\alpha} \coloneqq \left\{(h, i, \alpha - h - 2i) \in \mathbb{N}^3 \colon \hmin(\alpha) \leq h \leq \hmax(\alpha), \imin(\alpha) \leq i \leq \imax(\alpha)\right\}.\]
		
		Furthermore, for all $(h, i, \ell) \in \Range{\alpha}$, let $\W_{h, i, \ell}$ be the collection of walks of length $\alpha = h + 2i + \ell$ from $u$ to $S$ which are $r$-acyclic, have exactly $h + 2i$ transitions corresponding to an $(h, i)$-constrained walk, and exactly $\ell$ stationary transitions. Noting that there are $\binom{h + 2i + \ell}{h + 2i}$ ways of inserting $\ell$ stationary transitions along an $(h, i)$-constrained walk, by (ii) it follows that $\left|\W_{h, i, \ell}\right | \geq \zeta \binom{h + 2i + \ell}{h+2i} \widetilde{\omega}_{h, i}$.
		
		In particular then, by (\ref{eq:dcw0}), we have that for all $\alpha \in \mathbb{N}$ such that $\amin \leq \alpha \leq \amax$, for all $(h, i, \ell) \in \Range{\alpha}$, for all $\overrightarrow{x_\alpha} \in \W_{h, i, \ell}$ and for all $\eta \in \{0, 1\}^E$:
		\begin{equation} \label{eq:dcw1}
			\mathbb{P} \left[\widehat\Xi\left(\overrightarrow{x_\alpha}, \Cburn, T\right) \mid \left(\eta_0, X_0\right) = (\eta, u)\right] \geq C_0 \mathbb{P} \left[N_\mathrm{w} \left(\Cburn, \Cburn + T\right) = \alpha\right] \left(\frac{\pmin}{d}\right)^{h + 2i} (1 - \pmin)^\ell.
		\end{equation}
		
		Next fix $\eta \in \{0, 1\}^E$, $\alpha \in \mathbb{N}$ and a walk $\overrightarrow{x_\alpha}$ starting at $u$ and ending at some vertex in $S$. Recall that, by our discussion in Section 5, that if $\widehat \Xi\left(\overrightarrow{x_\alpha}, \Cburn, \Cburn + T\right)$ holds then the walker follows the trajectory $\overrightarrow{x_\alpha}$ during the time interval $(\Cburn, \Cburn + T)$ and is at the end-vertex of the walk, i.e. in $S$, at time $\Cburn + T$. We therefore have that,
		\begin{equation}
			\mathbb{P} \left[X_{\Cburn + T} \in S \mid \left(\eta_0, X_0\right) = (\eta, u)\right] 	\geq \mathbb{P} \biggl[ \ \bigcup_{\alpha = \amin}^{\amax} \bigcup_{(h, i, \ell) \in \Range{\alpha}} \bigcup_{\overrightarrow{x_\alpha} \in \W_{h, i, \ell}} \!\!\!\! \widehat\Xi\left(\overrightarrow{x_\alpha}, \Cburn, T \right)\bigg| \left(\eta_0, X_0\right) = (\eta, u) \biggr]. \label{eq:dcw3}
		\end{equation}
		
		Next note that for any $\alpha, \alpha' \in \mathbb{N}$ such that $\alpha \neq \alpha'$, the events $N_\mathrm{w}(\Cburn, \Cburn + T) = \alpha$ and $N_\mathrm{w}(\Cburn, \Cburn + T) = \alpha'$ are disjoint. Also, for any $(h, i, \ell), (h', i', l') \in \Range{\alpha}$ such that $(h, i, \ell) \neq (h', i', l')$, given any $\overrightarrow{x_\alpha} \in \W_{h, i, \ell}$ and $\overrightarrow{y_\alpha} \in \W_{h', i', l'}$, the events $\widehat\Xi\left(\overrightarrow{x_\alpha}, \Cburn, T\right)$ and $\widehat\Xi\left(\overrightarrow{y_\alpha}, \Cburn, T\right)$ are disjoint. This follows from the fact that either $h \neq h'$ and hence the number of non-stationary transitions is distinct, or else, if $h' = h$ then $l \neq l'$ and therefore the number of stationary transitions is distinct. 
		
		Lastly, we have that for any $\overrightarrow{x_\alpha}, \overrightarrow{y_\alpha} \in \W_{h, i, \ell}$ such that $\overrightarrow{x_\alpha} \neq \overrightarrow{y_\alpha}$, the events $\widehat\Xi\left(\overrightarrow{x_\alpha}, \Cburn, T\right)$ and $\widehat\Xi\left(\overrightarrow{y_\alpha}, \Cburn, T\right)$ are disjoint. This follows from the fact that either the non-stationary transitions of both walks differ in their indices, or else they match and hence the walks must differ in at least one vertex.
		
		Therefore from (\ref{eq:dcw3}) we have that,
		\begin{equation}
			\mathbb{P} \left[X_{\Cburn + T} \in S  \mid  \left(\eta_0, X_0\right) = (\eta, u)\right] \geq \sum_{\alpha = \amin}^{\amax} \sum_{(h, i, \ell) \in \Range{\alpha}} \sum_{\overrightarrow{x_\alpha} \in \W_{h, i, \ell}} \mathbb{P} \left[ \widehat\Xi\left(\overrightarrow{x_\alpha}, \Cburn, T\right) \mid  \left(\eta_0, X_0\right) = (\eta, u) \right]. \label{eq:dcw5}
		\end{equation}
		
		Together with (\ref{eq:dcw1}) and the fact that for all $(h, i, \ell) \in \Range{\alpha}$ we have that $|\W_{h, i, \ell}| > \zeta \binom{h + 2i + \ell}{h + 2i} \widetilde{\omega}_{h, i}$,
		\begin{align}
			&\ \mathbb{P} \left[X_{\Cburn + T} \in S \mid \left(\eta_0, X_0\right) = (\eta, u)\right] \nonumber \\
			\geq& \ C_0 \ \zeta \sum_{\alpha = \amin}^{\amax} \mathbb{P} \left[N_\mathrm{w} (\Cburn, \Cburn + T) = \alpha\right] \sum_{(h, i, \ell) \in \Range{\alpha}} \binom{h + 2i + \ell}{h+2i}\widetilde{\omega}_{h, i} \left(\frac{\pmin}{d}\right)^{h+2i} (1 - \pmin)^\ell. \label{eq:dcw6}
		\end{align}
		
		Let $a \coloneqq \left(1-\frac{1}{d}\right) \pmin$ and $b \coloneqq \frac{\pmin}{d}$. By the lower bound on $\widetilde{\omega}_{h,i}$ in Lemma \ref{lem:walk_counts}, our choice $a$ and $b$, Stirling's approximation for the factorial, along with the definition of $f_\alpha$ in Lemma \ref{lem:f_alpha_properties}, we have that for all $n$ sufficiently large:
		\begin{align*}
			&\sum_{(h, i, \ell) \in \Range{\alpha}} \binom{h + 2i + \ell}{h+2i}\widetilde{\omega}_{h, i} \left(\frac{\pmin}{d}\right)^{h+2i} (1 - \pmin)^\ell \\
			\geq& \sum\limits_{(h, i, \ell) \in \Range{\alpha}} \frac{h+1}{h+i+1} \frac{(h + 2i + \ell)!}{(h+i)! \ i! \ \ell!} \left(\frac{(d-1)\pmin}{d}\right)^{h+i} \left(\frac{\pmin}{d}\right)^{i} (1-\pmin)^\ell \\
			\geq& \ \frac{1}{2 \pi} \ \alpha^{\alpha} \! \sum\limits_{(h, i, \ell) \in \Range{\alpha}} \frac{h+1}{h+i+1} \sqrt{\frac{\alpha}{(h+i) \ i \ \ell}} \left(\frac{a}{h+i}\right)^{h+i} \left(\frac{b}{i}\right)^{i} \left(\frac{1-a-b}{\alpha - h - 2i}\right)^{\alpha - h - 2i} \\
			\geq& \ C_1 \ \alpha^{\alpha - 1} \! \sum\limits_{(h, i, \ell) \in \Range{\alpha}} \exp\Big(f_\alpha (h+i, i)\Big)
		\end{align*}
		where $C_1 \coloneqq \frac{1}{3200 \pi}$, since $\sqrt{\frac{\alpha}{(h+i) \ i \ \ell}} \geq \frac{1}{\alpha}$ and $\frac{h+1}{h+i+1} \geq \frac{\hmin\left(({1}/{(50 \ \pmin)}) \log n\right)+1}{\hmax\left(({4}/{\pmin}) \log_{d-1} n\right) + \imax\left(({4}/{\pmin}) \log_{d-1} n\right) + 1} \geq 2 \pi C_1$.
		
		Letting $C_2 \coloneqq \left(\frac{d-2}{2}\right)\left(\frac{\pmin}{d}\right)^2$, for all $n$ sufficiently large we have $|\Range{\alpha}| \geq C_2 \alpha$. Also, letting $C_3 \coloneqq \exp\left(-\frac{a+b}{2(1-a-b)}\right)$, by Lemma \ref{lem:f_alpha_properties} we have that for all $(h, i, \ell) \in \Range{\alpha}$: $\exp\left(f_\alpha (h + i, i)\right) \geq C_3 \exp\left(f_\alpha (a \alpha, b \alpha)\right) \geq C_3 \alpha^{-\alpha}$. Therefore,
		\[\sum\limits_{(h, i, \ell) \in \Range{\alpha}} \exp\left(f_\alpha (h + i, i)\right) \geq C_2 C_3 \alpha \exp\left(f_\alpha (a \alpha, b \alpha)\right) \geq \frac{C_2 C_3}{\alpha^{\alpha-1}}.\]
		
		Combining everything together, we obtain that:
		\begin{equation} \label{eq:dcw7}
			\sum_{(h, i, \ell) \in \Range{\alpha}} \binom{h + 2i + \ell}{h+2i}\widetilde{\omega}_{h, i} \left(\frac{\pmin}{d}\right)^{h+2i} (1 - \pmin)^\ell \geq C_1 C_2 C_3.
		\end{equation}
		
		From (\ref{eq:dcw6}) and (\ref{eq:dcw7}), defining $C \coloneqq \frac{C_0 C_1 C_2 C_3}{2}$, we have that
		\begin{align*}
			\mathbb{P} \left[X_{\Cburn + T} \in S \mid \left(\eta_0, X_0\right) = (\eta, u)\right] &\geq 2 C \zeta \sum_{\alpha = \amin}^{\amax} \mathbb{P} \left[N_\mathrm{w} (\Cburn, \Cburn + T) \alpha\right] \\
			&= 2 C \zeta  \cdot \mathbb{P} \left[N_\mathrm{w} (\Cburn, \Cburn + T) \in \left[\amin, \amax\right]\right] \\
			&\geq C \zeta
		\end{align*}
		where the last inequality follows from Lemma \ref{lem:poisson_concentration}. The result follows.
	\end{proof}
	
	We are now in a position to prove Lemmas \ref{lem:phase1_bdd} and \ref{lem:phase2_bdd}. Recall that we fixed the times
	\begin{align*}
		T_1 &\coloneqq (1 + \Cburn) + \left({d}/{(d - 2)}\right) \left({1}/{\pmin}\right) \left({1}/{40}\right) \log_{d - 1} n \\
		\mathrm{and} \quad T_2 &\coloneqq T_1 + \Cburn + \left({d}/{(d - 2)}\right) \left({1}/{\pmin}\right) \left(\left({81}/{80}\right) \log_{d- 1} n + \log_{d-1} \log n\right).
	\end{align*}
	
	\phaseOneBound*
	
	\begin{proof}
		By Lemma \ref{lem:phase1_walk_counts}, with probability $1 - o_n(1)$ over $G \sim \rrg$, for any vertex $u$ and $h, i \in \mathbb{N}$ such that ${R}/{10} \leq h \leq {R}/{5}$ and $i \leq {R}/{5}$, there are at least $d - 2$ neighbours of $u$ from which there are $\frac{\widetilde{\omega}_{h, i}}{2}$ walks which are $(h, i)$-constrained, $r$-acyclic and end at a vertex in $\Vgood$.
		
		Fix such a graph $G$, vertex $u \in V$, and define $\amin \coloneqq \left\lceil\frac{1}{50} \left({d}/{d - 2}\right) \left({1}/{\pmin}\right) \log_{d - 1} n\right\rceil$ and $\amax \coloneqq \left\lfloor\frac{3}{100} \left({d}/{d - 2}\right) \left({1}/{\pmin}\right) \log_{d - 1} n \right\rfloor$. Observe that this choice of $\amin$ and $\amax$ satisfies condition (i) of Lemma \ref{lem:te_full_prob_lw_bdd}. Furthermore, for all $\alpha, h, i \in \mathbb{N}$ such that $\alpha \in [\amin, \amax]$, $i \in \left[\frac{\pmin}{d} \alpha, \frac{\pmin}{d} \left(\alpha + \sqrt{\alpha}\right)\right]$ and $h \in \left[\left(\frac{d-2}{d}\right) \pmin \alpha, \left(\frac{d-2}{d}\right) \pmin \left(\alpha + \sqrt{\alpha}\right)\right]$, it follows that $h$ and $i$ satisfy ${R}/{10} \leq h \leq {R}/{5}$ and $i \leq {R}/{5}$. 
		
		Therefore there are $d - 2$ neighbours of $u$, say $u_1, \dots, u_{d-2}$, such that for all $1 \leq j \leq d - 2$, there are $\frac{\widetilde{\omega}_{h, i}}{2}$ walks from $u_j$ to a vertex in $\Vgood$ which are $r$-acyclic and $(h, i)$-constrained, and hence condition (ii) of Lemma \ref{lem:te_full_prob_lw_bdd} is satisfied for every $u_j$ with $S = \Vgood$ and $\zeta = 1/2$.
		
		Hence by Lemma \ref{lem:te_full_prob_lw_bdd}, with probability $1 - o_n(1)$ over the choice of $G$, noting that we invoke it at time $t = 1$ and that $T_1 = 1 + \Cburn + \frac{\amax + \amin}{2}$, there exists a constant $C \in (0, 1)$ such that for all $1 \leq j \leq d - 2$ we have that 
		\begin{equation} \label{eq:pob2}
			\mathbb{P} \left[X_{T_1} \in \Vgood \mid X_1 = u_j \right] \geq {C}/{2}.
		\end{equation}
		
		It remains to show that, with probability $\Omega(1)$, the walker transitions during $(0, 1)$ from $u$ to one of the neighbours $u_1, \dots, u_{d - 2}$ and then stays there until time $t = 1$. Recall the definition of the transition event $\xi_{u \to u'} (t, t')$ as in Definition \ref{def:transition-event} and let $\sigma_1$ denote the first ring time of the walker clock. If $N_{\mathrm{w}}(0, 1) = 1$ and $\xi_{u \to u_j} (0, \sigma_1)$ both hold for some $1 \leq j \leq d - 2$, the walker clock rings exactly once in $(0, 1)$, the edge $\{u, u_j\}$ refreshes open during $(0, \sigma_1)$, and at time $\sigma_1$ the walker examines that edge. Hence the walker transitions from $u$ to $u_j$; since there are no further rings in $(\sigma_1, 1)$, it remains there until time $t = 1$.
		
		Conditioned on $N_{\mathrm{w}}(0, 1) = 1$, the time $\sigma_1$ is uniformly distributed on $(0, 1)$.  Therefore, for all $1 \leq j \leq d - 2$ and for all $\eta \in \{0, 1\}^E$,
		\begin{equation} \label{eq:pob3}
			\mathbb{P} \left[X_1 = u_j \mid \left(\eta_0, X_0\right) = (\eta, u) \right] 
			\geq \int_0^1 \mathbb{P}\left[\xi_{u \to u_j} (0, x) \mid \left(\eta_0, X_0\right) = (\eta, u)\right] \mathbb{P}\left[N_{\mathrm{w}}(0, 1) = 1\right] \mathrm{d} x. 
		\end{equation}
		
		By Lemma \ref{lem:transition_lw_bdd} and the fact that $\mathbb{P}\left[N_{\mathrm{w}}(0, 1) = 1\right] = {1}/{e}$,
		\begin{equation} \label{eq:pob4}
			\left(\int_0^1 \mathbb{P}\left[\xi_{u \to u_j} (0, x) \mid \left(\eta_0, X_0\right) = (\eta, u)\right] \mathrm{d} x\right) \mathbb{P}\left[N_{\mathrm{w}}(0, 1) = 1\right] = \frac{\pmin}{e d} \left(1 - \frac{1 - e^{-\mu}}{\mu}\right). 
		\end{equation}
			
		Since $\mu = \Omega(\log n)$, $\left(1 - \frac{1 - e^{-\mu}}{\mu}\right) \geq {1}/{2}$ for $n$ sufficiently large. 	Moreover, for each fixed $x > 0$, the events $\xi_{u \to u_1}(0, x), \dots, \xi_{u \to u_{d - 2}}(0, x)$ are pairwise disjoint. Therefore for all $n$ sufficiently large, by (\ref{eq:pob3}) and (\ref{eq:pob4}),  we have that $\mathbb{P} \left[X_1 \in \{u_1, \dots, u_{d - 2}\} \mid \left(\eta_0, X_0\right) = (\eta, u) \right] \geq \left(\frac{d - 2}{d}\right) \frac{\pmin}{2 e}$. The result follows.
	\end{proof}
	
	\phaseTwoBound*
	
	\begin{proof}
		By Lemma \ref{lem:sparse_good_set_size}, with probability $1 - o_n(1)$ over $G \sim \rrg$, for any $v \in \Vgood$ there exists $S_v \subseteq V$ such that $\left|S_v\right| = n - o(n)$ and for all $h, i \in \mathbb{N}$ such that $\Hmin \leq h \leq \Hmax$ and $i \leq \left({4}/{\pmin}\right) \log_{d - 1} n$, the following holds: For every $x \in S_v$ there are at least $\frac{1}{4n} \widetilde{\omega}_{h, i}$ walks from $v$ to $x$ which are $(h, i)$-constrained and $r$-acyclic.
		
		Fix such a graph $G$, vertex $v \in \Vgood$, and let $\amin \coloneqq \left\lceil\frac{d}{d-2} \frac{1}{\pmin} \left(\log_{d-1} n + 2\log_{d-1}\log n \right)\right\rceil$ and $\amax \coloneqq \left\lfloor\frac{d}{d-2} \frac{1}{\pmin}(1+\frac{1}{40})\log_{d-1} n\right\rfloor$. Observe that this choice of $\amin$ and $\amax$ satisfies condition (i) of Lemma \ref{lem:te_full_prob_lw_bdd}. Furthermore, for all $\alpha, i, h \in \mathbb{N}$ such that $\alpha \in [\amin, \amax]$, $i \in \left[\frac{\pmin}{d} \alpha, \frac{\pmin}{d} \left(\alpha + \sqrt{\alpha}\right)\right]$ and $h \in \left[\left(\frac{d-2}{d}\right) \pmin \alpha, \left(\frac{d-2}{d}\right) \pmin \left(\alpha + \sqrt{\alpha}\right)\right]$, it follows that $h$ and $i$ satisfy $\Hmin \leq h \leq \Hmax$ and $i \leq \left({4}/{\pmin}\right) \log_{d - 1} n$. 
		
		Therefore there exists $S_v \subseteq V$ such that $|S_v| = n - o(n)$ and for all $x \in S_v$, there are $\frac{1}{4n}\widetilde{\omega}_{h, i}$ walks from $v$ to $x$ which are $r$-acyclic and $(h, i)$-constrained, and hence condition (ii) of Lemma \ref{lem:te_full_prob_lw_bdd} is satisfied with $S = \{x\}$ and $\zeta = {1}/{4n}$.
		
		Hence by Lemma \ref{lem:te_full_prob_lw_bdd}, with probability $1 - o_n(1)$ over the choice of $G$, noting that we invoke it at time $t = T_1$ and that $T_2 - T_1 = \Cburn + \frac{\amax + \amin}{2}$, there exists $C \in (0, 1)$ such that for all $x \in S_v$ and $\hat\eta \in \{0, 1\}^E$, we have that $\mathbb{P}\left[X_{T_2} = x \mid \left(\eta_{T_1}, X_{T_1}\right) = \left(\hat\eta, v\right)\right] \geq {C}/{4 n}$. The result follows.
	\end{proof}
	
	\bibliographystyle{plain}
	\bibliography{drwrc_ref}
    
    \appendix
    
\section{Detailed Chain Description}
    We next formalise the full Markov chain that we study.  For all positive integers~$d$ and~$n$,   $\mathcal{G}(d, n)$ denotes the family of (simple) $d$-regular graphs with~$n$ vertices. Given a graph $G = (V, E)$, 
an indicator vector $\eta \in \{0, 1\}^E$ and an edge $e \in E$, $\eta+e$ denotes
the indicator vector obtained from~$\eta$ by changing the entry of~$e$ to~$1$ (or leaving it alone if it is already~$1$). Similarly, $\eta-e$ is the configuration obtained from~$\eta$ by changing the entry of~$e$ to~$0$. 
Fix $G = (V, E) \sim \rrg$. We consider a continuous-time Markov chain $\left(\eta_t, X_t\right)$ with state space $\{0, 1\}^E \times V$. 
The indicator vector
$\eta_t\in \{0,1\}^E$ is the \emph{environment} at time~$t$.
It represents the configuration of edges in~$E$ which are open at time~$t$. 
The vertex~$X_t$ is the \emph{walker position} at time~$t$.

The dynamics are driven by $|E| + 1$ independent Poisson clocks -- one clock $\Cw$ with rate~$1$ which updates the walker $X_t$, and
for each edge $e\in E$,
a clock $\Ce$ with rate~$\mu$ which updates
the entry of $\eta_t$ corresponding to~$e$.
Let $\pmin \coloneqq \min\left\{p, \frac{p}{q(1-p) + p}\right\}$ and $\pmax \coloneqq \max\left\{p, \frac{p}{q(1-p) + p}\right\}$. If the clock~$\Ce$ rings at time $t$ then $\eta_{t-}$ denotes the environment just before
$\Ce$~rings.
The dynamics samples a uniform random variable $U_e (t)$ on $(0, 1)$ to determine whether  edge~$e$ is opened or closed in~$\eta_t$, depending on whether or not $e$~is a cut edge  in $\eta_{t-}$, according to the procedure {\sc Environment\_Update}$(\eta_{t-},e)$.
	\begin{algorithmic}
		\Function{environment\_update}{$\eta_{t-}, e$}
		\State $U_e (t) \sim \text{Uniform}(0, 1)$
		\If{$e$ is a cut edge in $\eta_{t-}$}
		\If{$U_e (t) \leq \pmin$} \Return $\eta_{t-} + e$
		\Else \ \Return $\eta_{t-} - e$
		\EndIf
		\Else
		\If{$U_e (t) \leq \pmax$} \Return $\eta_{t-} + e$
		\Else \ \Return $\eta_{t-} - e$
		\EndIf
		\EndIf
		\EndFunction
	\end{algorithmic}

Let $X_{t-}$ denote the location of the walker just before~$\Cw$ rings.
Whenever the walker clock~$\Cw$ rings, say at time $t$, the chain samples $V_{X_{t-}} (t)$ uniformly on the $d$ neighbours of $X_{t-}$ in $G$.\footnote{We will in fact assume that, whenever the walker clock~$\Cw$ rings (say at time $t$), for all $u\in V$ we sample  $V_u(t)$ to be a uniformly chosen neighbour of~$u$ in~$G$; the chain only uses $V_{X_{t-}}(t)$ but the extra samples will be technically convenient later.} If the edge $\{X_{t-}, V_{X_{t-}} (t)\}$ is open in $\eta_{t-}$, then $X_t = V_{X_{t-}} (t)$. Otherwise, the walker stays at the same location, i.e., $X_t = X_{t-}$.
The full Markov chain is thus described by 
the procedure {\sc full\_chain}$(\eta_0,X_0)$.
	\begin{algorithmic}
		\Function{full\_chain}{$\eta_0, X_0$}
		\If{$\Cw$ rings at time $t$}
		\State $\forall u \in V \colon V_u (t) \sim \text{Uniform}\left(v \in V \colon \{u, v\} \in E\right)$
		\If{$X_{t-}$ is adjacent to $V_{X_{t-}} (t)$ in $\eta_{t-}$} $X_t \leftarrow V_{X_{t-}} (t)$
		\EndIf
		\ElsIf{$\Ce$ rings at time $t$ for some $e \in E$} $\eta_t \leftarrow$ \Call{environment\_update}{$\eta_{t-}, e$}
		\EndIf
		\EndFunction
	\end{algorithmic}
	
Observe that the environment dynamics $(\eta_t)$ is a Markov chain whose stationary distribution is the random cluster measure $\pi_{G,p,q}$, but the walker position $(X_t)$ is \textit{not} Markovian since its transitions depend on the current configuration. Nevertheless, the full system $(\eta_t, X_t)$ is a Markov chain. 
Let $\pi^{\textrm{SRW}}_G$ denote the stationary distribution of the simple random walk on $G$. Then
the stationary distribution of the full system is
$\pi_{G, p, q} \times \pi^{\textrm{SRW}}_G$. Given a regular graph $G$, $\pi^{\textrm{SRW}}_G$ is simply the uniform distribution on the vertex set $V$, which we  denote by $\pi_V$. We use
$\tmix^{(\mu, p, q)} (G)$ to denote the mixing time of the full chain.

	\section{Sparse boundaries in burned-in environments} \label{sec:sparsity}
	
	Given a graph $G = (V, E)$, a configuration $\eta \in \{0, 1\}^E$ and a set of edges $H \subseteq E$, recall that by $\eta^H$ we denote the configuration $\eta^H (e) = 1$ if $e \in H$, else $\eta^H (e) = \eta(e)$ if $e \in E \setminus H$. 
	
	\begin{definition}[External $K$-sparsity events] \label{def:external-K-sparse-event}
		Given $v \in V$, let $E_v \coloneqq E\left(B_R (v)\right)$ and consider the discrete-time random cluster Glauber dynamics $\left(\eta_{\langle t \rangle}^{E_v}\right)_{t \in \mathbb{N}}$ with parameters $(p, q)$, that is, the dynamics where all edges in $E_v$ are wired open. 
		
		For all $F \subseteq \partial B_R (v)$, we define $S^F_{\langle t \rangle} \left(B_R(v)\right)$ as the event that, in the graph $\left(V,\eta_{\langle t \rangle}^{E_v}  \setminus E_v\right)$, the vertices in $\partial B_R (v)$ which are in non-trivial components are exactly $F$. 
		
		We define $S_{\langle t \rangle} \left(B_R(v), K\right)$ as the event that $S^F_{\langle t \rangle} \left(B_R(v)\right)$ holds for some $F \subseteq \partial B_R (u)$ such that $|F| \leq K$. Furthermore, we define $S_{\langle t \rangle} (R, K) \coloneqq \cap_{v \in V} S_{\langle t \rangle} (B_R(v), K)$. For $t'' > t' > 0$, by $S_{\langle[t', t'']\rangle} (R, K)$ we denote the event that $S_{\langle t \rangle} (R, K)$ holds for all $t \in [t', t'']$.
		
		By  $S^F_t \left(B_R (v)\right)$, $S_t \left(B_R (v), K\right)$, $S_t (R, K)$ and $S_{[t', t'']} (R, K)$ we will denote the continuous-time analogues of each event, respectively.
	\end{definition}
	
	The following result, implicit from the proof of Proposition 2 and Theorem 5 in \cite{Blanca2021}, asserts that after a sufficiently long ``burn-in" for the discrete-time random cluster Glauber dynamics, for any fixed $\delta, c_1, c_2 > 0$ there exists a constant $K > 0$ such that with probability $1 - O(n^{-c_1})$ over the choice of graph, for a given ball of radius $R$ around a vertex, with probability $1 - O(n^{-c_2})$ the number of vertices in non-trivial boundary components of this ball is at most $K$.
	
	\begin{restatable}[\cite{Blanca2021}, Theorem 5]{restatabletheorem}{sparseBurnIn} \label{thm:k_sparse_burn_in}
		Let $c_1, c_2 > 0$. There exist constants $D_{\mathrm{burn}}, K > 0$ such that the following holds: For all $n \in \mathbb{N}$ sufficiently large and for all $t \in \mathbb{N}$ such that $t \geq D_\mathrm{burn} n \log n$, with probability $1 - O\left(n^{-c_1}\right)$ over $G \sim \rrg$, for all $\eta' \in \{0, 1\}^E$, \[\mathbb{P} \left[\neg S_{\langle t \rangle} (R, K) \mid \eta_{\langle 0 \rangle} = \eta'\right] = O\left(n^{-c_2}\right).\]
	\end{restatable}
	
	\begin{proof}
		Here we briefly remark why we can state the results of \cite{Blanca2021} in this slightly different form. For a boundary condition $\eta$ on $\partial B_R(v)$, by $\mathcal{J}_{B_R (v), \eta}$ we will denote the set of vertices in $B_R (v)$ which belong to non-singleton components of the partition $\eta\left[\partial B_R (v)\right]$.
		
		First, their proof relies on the discrete-time joint-revealing procedure on the ball $B_R(v)$ and its boundary. In particular, their procedure does not reveal any information on the random variables associated with the edges inside $B_R (v)$ up to time $t$. 
		
		Instead, they consider the interior with all edges wired open, hence maximising connectivity through the ball, and only consider the random variables associated with the exterior edges for the dynamics. This gives a configuration $\tilde{\eta}_{\langle t \rangle}$ which dominates the true configuration $\eta_{\langle t \rangle}$, and hence sparsity on $\tilde{\eta}_{\langle t \rangle}$ via the exterior implies sparsity on $\eta_{\langle t \rangle}$ via the exterior. 
		
		Furthermore, in their proof they show that, jointly over the configuration model and $\tilde{\eta}_{\langle t \rangle}$, for any $k > 1$, the probability of $\left\{\big|\mathcal{J}_{B_R (v), \eta_{\langle t \rangle}}\big| < k\right\}$ is dominated by the probability of the event $S_{\langle t \rangle} \left(B_R (v), k\right)$, which in turn is bounded above by $n^{-\delta k \wedge 4}$.
		
		The constant $4$ that appears in the exponent can be carefully controlled. We desire it to be $c_1 + c_2 + 2$; let $M \geq \left\lceil\frac{c_1 + c_2 + 2}{\delta}\right\rceil$. We can then set $\Lambda$ sufficiently large in the proof of their Proposition 2 (dependent on $M$, among other parameters), such that the contribution from the term $\mathbb{P}\left[\sum_{k \leq k_\emptyset} |\mathcal{V}_k| \geq \Lambda \left|\mathcal{V}_0\right|\right]$ is bounded above by $O(n^{-M})$ by their Lemma 6. The proof of Proposition 2 then proceeds as usual, with the final Markov inequality over the configuration model applied with $n^{-(c_2 + 1)}$ instead of $n^{-2}$.
		
		This then gives that, with probability $1 - O\left(n^{-(c_1 + 1)}\right)$ over $G$ sampled from the configuration model, $\sup\limits_{v \in V} \mathbb{P} \left[\neg S_{\langle t \rangle}(B_R(v), K) \mid \eta_{\langle 0 \rangle} = \eta'\right] \leq n^{-(c_2 + 1)}$. From this, one can deduce the analogue of their Theorem 5 by union bounding over the vertices and conditioning on simple graphs in the configuration model.
	\end{proof}
	
	The following result gives a continuous-time analogue of Theorem \ref{thm:k_sparse_burn_in}. 
	
	\begin{lemma}\label{lem:sparse_sup_prob}
		There are constants $\Cburn, K > 0$ such that the following holds: With probability $1 - O(1/n)$ over $G \sim \rrg$, for all $\mu \geq \varepsilon \log n $, for all $t' > t \geq \Cburn$ such that $t' - t \leq (4/\pmin) \log_{d-1} n$ and for all $\eta \in \{0, 1\}^E$,  
		\[\mathbb{P} [\neg S_{[t' - \Delta, t']} (R, K) \mid \eta_0 = \eta] = O
		\Big(\frac{1}{n}\Big(\frac{\pmin}{4 d}\Big)^{(4/{\pmin}) \log_{d - 1} n}\Big)\]
		where $\Delta \coloneqq (t' - t) \left(\frac{\varepsilon \log n}{\mu}\right)$.
	\end{lemma}
	
	\begin{proof}
		Recall that we will use $S_t (R, K)$ and $S_{\langle t \rangle} (R, K)$ to differentiate between the continuous and discrete-time analogues, respectively, of the external sparsity event in Definition \ref{def:external-K-sparse-event}. 
		
		Let $\rho \coloneqq \left(\frac{4}{\pmin}\right) \frac{\log \left({4d}/{\pmin}\right)}{\log (d - 1)}$. By Theorem \ref{thm:k_sparse_burn_in}, there exist constants $D_\mathrm{burn}, K > 0$ such that for all $T \in \mathbb{N}$ satisfying $T \geq D_\mathrm{burn} n \log n$: With probability $1 - O\left(n^{-3}\right)$ over $G \sim \rrg$,
		\begin{equation} \label{eq:ssp1}
			\mathbb{P} \left[\neg S_{\langle T \rangle} (R, K) \mid \eta_{\langle 0 \rangle} = \eta\right] = O\left(n^{-(\rho + 3)}\right), \qquad \forall \eta \in \{0, 1\}^E.
		\end{equation}
		
		Let $C > 0$ be a constant. Let the edge clocks refresh at rate $\mu \geq\varepsilon \log n $. Let $t' > t \geq C$ such that $t' - t \leq (4/\pmin) \log_{d-1} n$. Define $\Delta \coloneqq (t' - t) \left(\frac{\varepsilon \log n}{\mu}\right)$ and $t_0 \coloneqq t' - \Delta - C \left(\frac{\varepsilon \log n}{\mu}\right)$; observe that $t' - \Delta = t_0 + C \left(\frac{\varepsilon \log n}{\mu}\right)$ and that for all $\mu \geq \varepsilon \log n$ we have that $t_0 = t' - (t' - t + C) \left(\frac{\varepsilon \log n}{\mu}\right) \geq t' - (t' - t + C) \geq t - C \geq 0$ since $t \geq C$.

        By the Markov property, we have that, for all $\eta \in \{0, 1\}^E$,
        \begin{equation}
            \mathbb{P} [\neg S_{[t' - \Delta, t']} (R, K) \mid \eta_0 = \eta] = \sum_{\eta' \in \{0, 1\}^E} \mathbb{P} [\neg S_{[t' - \Delta, t']} (R, K) \mid \eta_{t_0} = \eta'] \mathbb{P}\left[\eta_{t_0} = \eta' \mid \eta_0 = \eta\right]
        \end{equation}
        and therefore it suffices to show that, for all $\eta' \in \{0, 1\}^E$, $\mathbb{P} \left[\neg S_{[t' - \Delta, t']} \left(R, K\right) \mid \eta_{t_0} = \eta'\right] = O\left(n^{-(\rho + 1)}\right)$.
		
		During the interval $\left[t_0, t' - \Delta\right)$, which has length $C \left(\frac{\varepsilon \log n}{\mu}\right)$, the number of edge refreshes is $N_1 (C) \sim \mathsf{Poisson}\left(\left({C d \varepsilon}/{2}\right) n \log n\right)$. Let $T_\mathrm{burn} \coloneqq \left\lceil D_\mathrm{burn} n \log n\right\rceil$. Since $\mathbb{E}\left[N_1 (C)\right] = \left({C d \varepsilon}/{2}\right) n \log n$, then by concentration there exists $\Cburn > 0$ such that $\mathbb{P}\left[N_1 (\Cburn) < T_\mathrm{burn} \right] \leq O\left(n^{-(\rho + 1)}\right)$. Also, there exists $C_1 > 0$ such that $\mathbb{P}\left[N_1(\Cburn) > C_1 n \log n\right] \leq O\left(n^{-(\rho + 1)}\right)$. Define $T_{1, \mathrm{max}} \coloneqq \left\lfloor C_1 n \log n \right\rfloor$. 
		
		During the interval $\left[t' - \Delta, t'\right]$, the number of edge refreshes is $N_2 \sim \mathsf{Poisson}\left((d \varepsilon / 2) (t' - t) n \log n)\right)$. Since $t' - t \leq (4 / \pmin) \log_{d-1} n$, we have that $\mathbb{E}\left[N_2\right] \leq \left(\frac{2 d \varepsilon}{\pmin \log(d - 1)}\right) n (\log n)^2$, and by concentration there exists $C_2 > 0$ such that $\mathbb{P}\left[N_2 > C_2 n (\log n)^2 \right] \leq O\left(n^{-(\rho + 1)}\right)$. Let $T_{2, \mathrm{max}} \coloneqq \left\lfloor C_2 n (\log n)^2\right\rfloor$.

        Note that, since Poisson clocks are independent over disjoint time intervals, $N_1 (\Cburn)$ and $N_2$ are independent. Given $T_1, T_2 \in \mathbb{N}$ and $\eta' \in \{0, 1\}^E$, conditioned on $N_1 (\Cburn) = T_1$ and $N_2 = T_2$, the continuous-time dynamics on $[t_0, t']$, started at $\eta_{t_0} = \eta'$, is distributed as the discrete-time chain run for $T_1 + T_2$ steps from the initial state $\eta_{\langle 0 \rangle} = \eta'$, where at each step an edge is chosen independently and uniformly at random. Define $T_\mathrm{max} \coloneqq T_{1, \mathrm{max}} + T_{2, \mathrm{max}}$. By a simple union bound, for all $\eta' \in \{0, 1\}^E$,
		\begin{align}
			& \ \mathbb{P} \left[\neg S_{[t' - \Delta, t']} \left(R, K\right) \mid \eta_{t_0} = \eta'\right] \nonumber \\
			=& \sum_{T_1 = 0}^\infty \sum_{T_2 = 0}^\infty \mathbb{P} \left[\neg S_{\langle[T_1, T_1 + T_2]\rangle} \left(R, K\right) \mid \eta_{\langle 0 \rangle} = \eta'\right] \mathbb{P}\left[N_1 (\Cburn) = T_1\right] \mathbb{P}\left[N_2 = T_2\right] \nonumber \\
			\leq& \sum_{T_1 = 0}^\infty \sum_{T_2 = 0}^\infty \sum_{T = T_1}^{T_1 + T_2} \mathbb{P} \left[\neg S_{\langle T \rangle} \left(R, K\right) \mid \eta_{\langle 0 \rangle} = \eta'\right] \mathbb{P}\left[N_1 (\Cburn) = T_1\right] \mathbb{P}\left[N_2 = T_2\right] \nonumber \\
			\leq& \ \mathbb{P}\left[N_1 (\Cburn) < T_\mathrm{burn}\right] + \mathbb{P}\left[N_1 (\Cburn) > T_{1, \mathrm{max}}\right] + \mathbb{P}\left[N_2 > T_{2, \mathrm{max}}\right] \nonumber \\
			& \ + \sum_{T_1 = T_\mathrm{burn}}^{T_{1, \mathrm{max}}} \sum_{T_2 = 0}^{T_{2, \mathrm{max}}} \sum_{T = T_1}^{T_1 + T_2} \mathbb{P} \left[\neg S_{\langle T \rangle} \left(R, K\right) \mid \eta_{\langle 0 \rangle} = \eta'\right] \mathbb{P}\left[N_1 (\Cburn) = T_1\right] \mathbb{P}\left[N_2 = T_2\right] \nonumber \\
			\leq & \ \mathbb{P}\left[N_1 (\Cburn) \notin \left[T_\mathrm{burn}, T_{1, \mathrm{max}}\right]\right] + \mathbb{P}\left[N_2 > T_{2, \mathrm{max}}\right] + \sum_{T = T_\mathrm{burn}}^{T_\mathrm{max}} \mathbb{P} \left[\neg S_{\langle T \rangle} \left(R, K\right) \mid \eta_{\langle 0 \rangle} = \eta'\right].  \label{eq:ssp2}
		\end{align}
		
		For all $T \in \mathbb{N}$ such that $T_\mathrm{burn} \leq T \leq T_\mathrm{max}$, let $\mathcal{G}_{\langle T \rangle}$ be the family of graphs satisfying (\ref{eq:ssp1}). Since $T_\mathrm{max} = \Theta\left(n (\log n)^2\right)$, then by a simple union bound, with probability $1 - O\left(n^{-1}\right)$ over $G \sim \rrg$, $G$ is in $\cap_{T = T_\mathrm{burn}}^{T_\mathrm{max}} \mathcal{G}_{\langle T \rangle}$, since for our choice of $K$, the probability that $G \notin \mathcal{G}_{\langle T \rangle}$ is $O\left(n^{-3}\right)$. Now for every such graph $G$ we have that for all $T \in \mathbb{N}$ such that $T_\mathrm{burn} \leq T \leq T_\mathrm{max}$ and for all $\eta' \in \{0, 1\}^E$, $\mathbb{P} \left[\neg S_{\langle T \rangle} (R, K) \mid \eta_{\langle 0 \rangle} = \eta'\right] = O\left(n^{-(\rho + 3)}\right)$. Furthermore, by our choices of $\Cburn$, $T_{1, \mathrm{max}}$ and $T_{2, \mathrm{max}}$, we have that $\mathbb{P}\left[N_1 (\Cburn) \notin \left[T_\mathrm{burn}, T_{1, \mathrm{max}}\right]\right] = O\left(n^{-(\rho + 1)}\right)$ and $\mathbb{P}\left[N_2 > T_{2, \mathrm{max}}\right] = O\left(n^{-(\rho + 1)}\right)$.
		
		Therefore, from these facts in conjunction with (\ref{eq:ssp2}),
		\begin{equation*}
			\mathbb{P} \left[\neg S_{[t' - \Delta, t']} \left(R, K\right) \mid \eta_{t_0} = \eta'\right] = O\left(n^{-(\rho + 1)}\right) + \sum_{T = T_\mathrm{burn}}^{T_\mathrm{max}} O\left(n^{-(\rho + 3)}\right) = O\left(n^{-(\rho + 1)}\right)
		\end{equation*}
		noting that $T_\mathrm{max} = \Theta\left(n (\log n)^2\right)$. The result follows.
	\end{proof}

    \section{Lower bounding walker transition events: Proofs of Lemmas \ref{lem:transition_lw_bdd} and \ref{lem:stationary_lw_bdd}} \label{sec:transition_event_lw_bdds}
	
	We next turn our attention back to our full chain, with both the walker and environment dynamics. Consider a $d$-regular graph $G = (V, E)$ on $n$ vertices. Recall that for an edge $e \in E$ and $t' > t \geq 0$, by $\Ce (t, t')$ we denote the event that $\Ce$ rings in $(t, t')$, and, conditioned on the event $\Ce(t, t')$, let $\tau_e (t, t')$ denote the last ring time of $\Ce$ in $(t, t')$. Given $e = \{u, u'\} \in E$ and $t' > t \geq 0$, recall from Definition \ref{def:transition-event} that $\xi_{u \to u'} (t, t')$ is the event that all of the following hold: $V_u (t') = u'$, $\Ce(t, t')$ and $e$ is refreshed open at time $\tau_e (t, t')$. Also, from Definition \ref{def:stationary-transition-event}, recall that  $\xi_{u \not\to u'} (t, t')$ is the event that all of the following hold: $V_u (t') = u'$, $\Ce(t, t')$ and $e$ is refreshed closed at time $\tau_e (t, t')$.
    
    We will prove Lemmas \ref{lem:transition_lw_bdd} and \ref{lem:stationary_lw_bdd}, which establish lower bounds on the probability of the events $\xi_{u \to u'} (t, t')$ and $\xi_{u \not\to u'} (t, t')$, respectively. 
	
	Before proceeding further, we note the following useful fact about Poisson clocks.
	
	\begin{lemma} \label{poisson_pdf_last_arrival}
		Let $\mathsf{C}$ be a Poisson clock of rate $\mu > 0$. Let $T > 0$ and let $\tau$ be the last time $\mathsf{C}$ rings on $(0, T)$, conditioned on the event $\mathsf{C}(0, T)$. Then $\tau$ has cumulative distribution function
		\begin{eqnarray*}
			\mathcal{F}_{\tau \vert \mathsf{C}(0, T)} (x) = \mathbb{P}\left[\tau \leq x \mid \mathsf{C}(0, T)\right] = \frac{e^{\mu x} - 1}{e^{\mu T} - 1},& &0 < x < T
		\end{eqnarray*}
		and probability density function $f_{\tau \vert \mathsf{C}(0, T)} (x) = \frac{\mu e^{\mu x}}{e^{\mu T} - 1}$, $0 < x < T$.
	\end{lemma}
	
	\begin{proof}
		First note that $\mathcal{F}_{\tau \vert \mathsf{C}(0, T)} (x) = \mathbb{P}\left[\tau \leq x \mid \mathsf{C}(0, T)\right] = \frac{\mathbb{P}\left[\mathsf{C}(0, x]\right]\mathbb{P}\left[\neg\mathsf{C}(x, T)\right]}{\mathbb{P}\left[\mathsf{C}(0, T)\right]} = \frac{e^{\mu x} - 1}{e^{\mu T} - 1}$. Taking the derivative with respect to $x$, we obtain that $f_{\tau \vert \mathsf{C}(0, T)} (x) = \dfrac{\mu e^{\mu x}}{e^{\mu T} - 1}$.
	\end{proof}
	
	\subsection{Proof of Lemma \ref{lem:transition_lw_bdd}}
	
	\transitionLwBdd*
	
	\begin{proof}
		Let $e \coloneqq \{u, u'\}$. If the edge $e$ refreshes during $(t, t')$, i.e. the event $\Ce (t, t')$ holds, and during the last ring $\tau_e (t, t')$ it holds that $U_e \left(\tau_e (t, t')\right) < \pmin$, then $e$ is refreshed open independent of whether it is a cut edge or not at time $\tau_e (t, t')$. Therefore, the events $V_u (t') = u'$, $\Ce (t, t')$ and $U_e \left(\tau_e (t, t')\right) < \pmin$ imply the event $\xi_{u \to u'} (t, t')$. Hence,
        \[\mathbb{P} \left[\xi_{u \to u'} \left(t,t'\right) \mid (\eta_t, X_t) = (\eta, v)\right] \geq \mathbb{P} \left[V_{u}\left(t'\right) = u', \Ce\left(t,t'\right), U_e \left(\tau_e(t, t')\right) < \pmin \mid (\eta_t, X_t) = (\eta, v)\right].\]

        The events $V_u (t') = u'$, $\Ce (t, t')$ and $U_e \left(\tau_e (t, t')\right) < \pmin$ concern the ringing of Poisson clocks and sampling of uniforms during the time interval $(t, t')$, which are independent of the state at time $t$. Furthermore, the random variable $V_u$ is independent of the environment random variables, namely $U_e$ and $\Ce$. 
        
        We therefore have that,
		\begin{align*}
			&\ \mathbb{P} \left[V_{u}\left(t'\right) = u', \Ce\left(t,t'\right), U_e \left(\tau_e(t, t')\right) < \pmin \mid (\eta_t, X_t) = (\eta, v)\right] \\
            =&\ \mathbb{P}\left[V_{u}\left(t'\right) = u'\right] \mathbb{P}\left[U_e \left(\tau_e(t, t')\right) < \pmin \mid \Ce\left(t,t'\right)\right] \mathbb{P}\left[\Ce \left(t,t'\right)\right].
		\end{align*}
		
		Since $V_{u}\left(t'\right)$ is uniform on the $d$ neighbours of $u$, then $\mathbb{P}\left[V_{u}\left(t'\right) = u'\right] = 1/d$. By the probability of a Poisson clock ringing in a given time interval, we have $\mathbb{P}\left[\Ce \left(t,t'\right)\right] = 1 - e^{-\mu \left(t' - t\right)}$. Lastly, since $U_e$ is uniform on $(0, 1)$, we have $\mathbb{P}\left[U_e \left(\tau_e(t, t')\right) < \pmin \mid \Ce\left(t,t'\right)\right] = \pmin$. The result follows.
	\end{proof}
	
	\subsection{Controlling cut edge event probabilities} \label{sec:st_events}

    In order to obtain a good lower bound on the probability of $\xi_{u \not\to u'}(t, t')$, we need to carefully control the probability that $e = \{u, u'\}$ is a cut edge in the configuration at its last refresh time $\tau_e (t, t')$, conditioned on the event $\Ce(t, t')$.

An edge $e = \{u, v\}$ fails to be a cut edge in a configuration if there exists an open path in the configuration between $u$ and $v$ 
that does not use the edge $e$. If, at the last refresh time of the edge $e$, 
before it is examined by the walker, there were no such open paths between $u$ and $v$, then $e$ was a cut edge at its last refresh time. It therefore suffices to control the probability of open paths between endpoints of edges incident to the walker vertex at the last time of refresh for the edge examined. We proceed to define a series of useful events which together guarantee that an edge is a cut edge in a given configuration. Given an edge $e' \in E$ and $t'' > t \geq 0$, let $\mathcal{E}_{e'} (t, t'')$ be the event that $\mathsf{C}_{e'} (t, t'')$ and $U_{e'}(\tau_{e'}(t, t'')) \geq \pmax = p$ both hold. We therefore have that, 
	\begin{equation} \label{eq:ste1}
		\mathbb{P}\left[\neg \mathcal{E}_{e'}(t, t'')\right] = 1 - (1-p)(1-e^{-\mu (t'' - t)}) = p + (1-p)e^{-\mu (t'' - t)}.
	\end{equation} 

    In particular, if $\mathcal{E}_{e'}(t, t'')$ holds then the edge $e'$ is refreshed during $(t, t'')$ and at its last refresh it is closed independently of whether it was a cut edge or not, and remains closed until time $t''$.
	
	\begin{remark} \label{rem:edge_refresh_events}
		For any two edges $e', e'' \in E$, the events $\mathcal{E}_{e'} (t, t'')$ and $\mathcal{E}_{e''} (t, t'')$ are independent, by the independence of the corresponding clocks and random variables. More generally, $\mathcal{E}_{e'} (t, t'')$ is independent of any other event concerning the clocks and uniforms of edges distinct from $e'$ during $(t, t'')$. Moreover, since $\mathcal{E}_{e'} (t, t'')$ concerns the ringing of Poisson clocks and sampling of uniforms during the time interval $(t, t'')$, which are independent of the state at time $t$, then $\mathcal{E}_{e'} (t, t'')$ is independent of the state at time $t$.
	\end{remark}
	
	Given a vertex $u$ and a cycle $\mathbf{c}$ in the underlying graph $G$, let $E_{u, r} (\mathbf{c})$ be the set of edges on $\mathbf{c}$ which are distance at most $r$ away from $u$ (i.e. for a given edge $e$ on $\mathbf{c}$, the shortest path from $u$ to each end-vertex of $e$ has length at most $r$). If a cycle $\mathbf{c}$ passing through $u$ has length greater than $r$, then $\left|E_{u, r} (\mathbf{c})\right| \geq r$. Also, given an edge $e \in E$, let $C_e (R)$ be the set of cycles in the underlying graph $G$ with length strictly less than $2R$ that contain $e$. Then given a vertex $u$ and edge $e$ incident to $u$, for all $t'' > t \geq 0$ define the event
	\begin{equation*}
		\xi_{u, e}^{\mathrm{acyclic}}\left(t, t''\right) \coloneqq \bigcap_{\mathbf{c} \in C_e (R)}  \bigcup_{\substack{e' \in E_{u, r} (\mathbf{c}) \\ e' \neq e}} \!\! \mathcal{E}_{e'} (t, t'').
	\end{equation*}
	which is the event that during the time interval $(t, t'')$, for every cycle in $C_e (R)$ at least one edge $e' \neq e$ on the cycle, distance at most $r$ away from $u$, was refreshed and the last time this edge was refreshed in $(t, t'')$ it was closed (independently of whether it was a cut edge or not). 
	
	\begin{remark} \label{rem:short_cycle_cut}
		In particular, if $\xi_{u, e}^{\mathrm{acyclic}}\left(t, t''\right)$ holds then for any cycle containing $e$ which is entirely contained in the ball $B_R(u)$ (i.e. has length less than $2R$), at least one edge, distinct from $e$, on such a cycle must be closed at time $t''$.
	\end{remark}
	
	We next prove the following lower bound for the probability of the event $\xi_{u, e}^{\mathrm{acyclic}}\left(t, t + x\right)$. 
	
	\begin{lemma} \label{lem:acyclic_prob}
		For every $d$-regular graph $G = (V, E)$ and for all $\mu > 0$, $t \geq 0$, $x > 0$ and $e = \{u, u'\} \in E$ such that $u$ does not lie on a cycle of length less than $r$ and $B_R(u)$ has at most one cycle, then $\mathbb{P} \left[\xi_{u, e}^{\mathrm{acyclic}}\left(t, t + x\right) \right] \geq 1 - \left(p + (1-p)e^{-\mu x}\right)^{r - 1}$.
	\end{lemma}
	
	\begin{proof}
		Since there is at most one cycle in $B_R (u)$, then $\left|C_e (R)\right| \leq 1$. Moreover, if such a cycle containing $e$ exists in $B_R (u)$, its length must be at least $r$ (otherwise $u$ lies on a cycle of length less than $r$) and hence $\left|E_{u, r}(\mathbf{c})\right| \geq r$. By complementation and a simple union bound (noting the independence of the $\mathcal{E}$ events for distinct edges in Remark \ref{rem:edge_refresh_events}), 
		\begin{equation*}
			\mathbb{P} \left[\neg \xi_{u, e}^{\mathrm{acyclic}}\left(t, t + x\right) \right] \leq \sum_{\mathbf{c} \in C_e(R)} \prod_{\substack{e' \in E_{u, r}(\mathbf{c}) \\ e' \neq e}} \mathbb{P} \left[\neg \mathcal{E}_{e'} (t, t + x)\right] \leq \left(p + (1-p)e^{-\mu x}\right)^{r - 1}
		\end{equation*}
		where the last inequality follows using $\left|C_{e} (R)\right| \leq 1$, the fact that any such cycle must have at least $r - 1$ edges distinct from $e$ that are a distance at most $r$ from $u$, and the upper bound on $\mathbb{P}\left[\neg \mathcal{E}_{e'}(t, t + x)\right]$ given in (\ref{eq:ste1}). The result follows.
	\end{proof}
	
	Given a vertex $u$ and a path $\mathbf{p}$, let $E_{u, (r, R)}(\mathbf{p})$ be the set of edges on $\mathbf{p}$ which are at distance at least $r$ but at most $R$ away from $u$ (i.e. for a given edge $e$ on $\mathbf{p}$, the shortest path from $u$ to each end-vertex of $e$ has length at least $r$ but at most $R$). If a path $\mathbf{p}$ has one end-vertex at $u$ and the other in $\partial B_R(u)$, then $\left|E_{u, (r, R)}(\mathbf{p})\right| \geq R - r$. Furthermore, given $F \subseteq \partial B_R (u)$, let $P\left(B_R (u), F\right)$ be the set of paths in the graph from $u$ to $F$ which are contained entirely in $B_R(u)$. Then, for any edge $e$ incident to $u$, any $t'' > t \geq 0$, and any $F \subseteq \partial B_R (u)$, define the event
	\begin{equation*}
		\mathcal{P}^F_{u, e} (t, t'') \coloneqq \bigcap_{\mathbf{p} \in P\left(B_R (u), F\right)} \ \bigcup_{\substack{e' \in E_{u, (r, R)}(\mathbf{p}) \\ e' \neq e}}  \mathcal{E}_{e'} (t, t'')
	\end{equation*}
	which is the event that during the interval $(t, t'')$, for every path in $P\left(B_R (u), F\right)$, at least one edge $e' \neq e$ on the path was refreshed and the last time this edge was refreshed in $(t, t'')$ it was closed (independently of whether it was a cut edge or not). 
	
	Next, recall the definition of the external sparsity event $S^F_t \left(B_R (u)\right)$ from Definition \ref{def:external-K-sparse-event}. Given $K > 0$ as in Lemma \ref{lem:sparse_sup_prob}, define the event
	\begin{equation*}
		\xi_{u, e}^{\mathrm{path}}\left(t, t''\right) \coloneqq \bigcup_{\substack{F \subseteq \partial B_R (u) \\ |F| \leq K}} \mathcal{P}^{F}_{u, e} (t, t'') \cap S^F_{t''} \left(B_R (u)\right).
	\end{equation*}
	
	\begin{remark} \label{rem:long_cycle_cut}
		In particular, if $\xi_{u, e}^{\mathrm{path}}\left(t, t''\right)$ holds then for any cycle in the underlying graph which contains $e$ and is not entirely contained in the ball $B_R(u)$ (i.e. has length greater than $2R$), there is at least one edge, distinct from $e$, which is closed in the configuration at time $t''$. Suppose that $F \subseteq \partial B_R (u)$ is the set of non-trivial boundary vertices on $B_R(u)$ at time $t''$. Every such cycle must pass through the boundary of the ball $B_R (u)$ through at least two distinct vertices. If the cycle does not pass through the boundary at a vertex in $F$, then at least one edge (distinct from $e$) outside the ball on the cycle must be closed in the configuration at time $t''$ (otherwise there is an open path outside of the ball between non-trivial boundary vertices). Otherwise, if the cycle passes through the boundary at a vertex in $F$, then the segment of the cycle from $u$ to this vertex which is contained in $B_R(u)$ must have had an edge (distinct from $e$) refreshed closed during $(t, t'')$, since $\mathcal{P}^{F}_{u, e} (t, t'')$ holds.
	\end{remark}
	
	We will next prove the following lemma.
	
	\begin{lemma} \label{lem:path_prob}
		Let $n \in \mathbb{N}$ be sufficiently large and $G = (V, E)$ a $d$-regular graph on $n$ vertices. For all $\mu > 0$, $t \geq 0$, $x > 0$, $e = \{u, u'\} \in E$ such that $B_R(u)$ has at most one cycle, and $(\eta, v) \in \{0, 1\}^E \times V$, then
		\begin{align*}
			& \mathbb{P}\left[\xi_{u, e}^{\mathrm{path}}\left(t, t + x\right) \mid \left(\eta_t, X_t\right) = \left(\eta, v\right)\right] \\
			\geq&\ \left(1 - 2K \left(p + (1-p)e^{-\mu x}\right)^{{R}/{2}}\right) \mathbb{P}\left[S_{t+x} \left(R, K\right) \mid \left(\eta_t, X_t\right) = \left(\eta, v\right)\right].
		\end{align*}
	\end{lemma}
	
	\begin{proof}
		First note that, by disjointness of events over $F \subseteq \partial B_R (u)$, we have that
		\begin{equation*} 
			\mathbb{P}\left[\xi_{u, e}^{\mathrm{path}}\left(t, t + x\right) \mid \left(\eta_t, X_t\right) = \left(\eta, v\right)\right] = \sum_{\substack{F \subseteq \partial B_R (u) \\ |F| \leq K}} \mathbb{P}\left[\mathcal{P}^{F}_{u, e} (t, t + x) \cap S^F_{t + x} \left(B_R (u)\right) \mid \left(\eta_t, X_t\right) = \left(\eta, v\right)\right].
		\end{equation*}
		
		Now, by Remark \ref{rem:edge_refresh_events}, the event $\mathcal{P}^{F}_{u, e} (t, t + x)$ is independent of the state at time $t$. Furthermore, recall that the external sparsity event $S^F_{t + x} \left(B_R (u)\right)$ is independent of the edge clocks and uniforms associated with edges inside $B_R(u)$ up to time $t + x$. On the other hand, $\mathcal{P}^{F}_{u, e} (t, t + x)$ is an event concerning the clocks and uniforms associated with edges inside $B_R(u)$ during $(t, t + x)$. Hence $S^F_{t + x} \left(B_R (u)\right)$ and $\mathcal{P}^{F}_{u, e} (t, t + x)$ are also independent. Consequently, we have that
		\begin{equation*} 
			\mathbb{P}\left[\xi_{u, e}^{\mathrm{path}}\left(t, t + x\right) \mid \left(\eta_t, X_t\right) = (\eta', v)\right] = \sum_{\substack{F \subseteq \partial B_R (u) \\ |F| \leq K}} \!\! \mathbb{P}\left[\mathcal{P}^{F}_{u, e} (t, t + x)\right] \mathbb{P}\left[S^F_{t + x} \left(B_R (u)\right) \mid \left(\eta_t, X_t\right) = \left(\eta, v\right)\right].
		\end{equation*}

        Since there is at most one cycle in $B_R(u)$, then given $F \subseteq \partial B_R (u)$ such that $|F| \leq K$, the number of paths inside $B_R (u)$ from $u$ to a vertex in $F$ is at most $2K$ and hence $\left|P\left(B_R(u), F\right)\right| \leq 2K$. Indeed, suppose that there is some vertex $v$ in $F$ such that there are three or more paths in $B_R (u)$ from $u$ to $v$; then there are at least two cycles in $B_R(u)$, which is a contradiction.
        
		Next, by complementation and a simple union bound (noting the independence of the $\mathcal{E}$ events for distinct edges in Remark \ref{rem:edge_refresh_events}), we have that
		\begin{equation*}
			\mathbb{P} \left[\neg \mathcal{P}^{F}_{u, e} (t, t + x) \right] \leq \sum_{\mathbf{p} \in P\left(B_R(u), F\right)} \prod_{\substack{e' \in E_{e, (r, R)}(\mathbf{p}) \\ e' \neq e}} \mathbb{P} \left[\neg \mathcal{E}_{e'} (t, t + x)\right] \text{d} x \leq 2K \left(p + (1-p)e^{-\mu x}\right)^{R - r - 1}
		\end{equation*}
		where the last inequality follows from $\left|P\left(B_R(u), F\right)\right| \leq 2K$, the fact that for any $\mathbf{p} \in P\left(B_R(u), F\right)$ we have $\left|E_{e, (r, R)} (\mathbf{p})\right| \geq R - r$, and the upper bound on $\mathbb{P}\left[\neg \mathcal{E}_{e'}(t, t + x)\right]$ given in (\ref{eq:ste1}).
		
		For $n$ sufficiently large we have $R - r - 1 \geq {R}/{2}$, since $R = \Theta(\log n)$ and $r = \Theta(\log \log n)$. Therefore, for $n$ sufficiently large, $\mathbb{P} \left[\neg \mathcal{P}^{F}_{u, e} (t, t + x) \right] \leq 2K \left(p + (1-p)e^{-\mu x}\right)^{{R}/{2}}$. From Definition \ref{def:external-K-sparse-event}, we have that $S_{t+x} \left(B_R(u), K\right)$ is the disjoint union of the events $S^F_{t+x} \left(B_R (u)\right)$ for all $F \subseteq \partial B_R (u)$ such that $|F| \leq K$. Combining everything together, we have that
		\begin{align*} 
			&\ \mathbb{P}\left[\xi_{u, e}^{\mathrm{path}}\left(t, t + x\right) \mid \left(\eta_t, X_t\right) = (\eta', v)\right] \\
			\geq& \left(1 - 2K \left(p + (1-p)e^{-\mu x}\right)^{{R}/{2}}\right) \sum_{\substack{F \subseteq \partial B_R (u) \\ |F| \leq K}} \mathbb{P}\left[S^F_{t + x} \left(B_R (u)\right) \mid \left(\eta_t, X_t\right) = \left(\eta, v\right)\right] \\
            =& \left(1 - 2K \left(p + (1-p)e^{-\mu x}\right)^{{R}/{2}}\right) \mathbb{P}\left[S_{t+x} \left(B_R (u), K\right) \mid \left(\eta_t, X_t\right) = \left(\eta, v\right)\right] \\
			\geq& \left(1 - 2K \left(p + (1-p)e^{-\mu x}\right)^{{R}/{2}}\right) \mathbb{P}\left[S_{t+x} \left(R, K\right) \mid \left(\eta_t, X_t\right) = \left(\eta, v\right)\right]
		\end{align*}
        where the last inequality follows from the fact that $S_{t+x} \left(R, K\right) \subseteq S_{t+x} \left(B_R (u), K\right)$. The result follows.
	\end{proof}
	
	If the events $\xi_{u, e}^{\mathrm{acyclic}}\left(t, t''\right)$ and $\xi_{u, e}^{\mathrm{path}}\left(t, t''\right)$ both hold, then by Remarks \ref{rem:short_cycle_cut} and \ref{rem:long_cycle_cut} all cycles passing through the edge $e$ in the underlying graph have at least one edge, distinct from $e$, closed at time $t''$. Hence $e$ is a cut edge in the configuration at that time. We therefore define the following so-called ``cut edge event".
	
	\begin{definition}[$\xi_{u, e}^{\mathrm{cut}} (t, t'')$] \label{def:cut edge-event}
		Let $G = (V, E)$ be a $d$-regular graph. Let $t'' > t \geq 0$ and $u, u' \in V$ such that $e = \{u, u'\} \in E$. The event $\xi_{u, e}^{\mathrm{cut}} (t, t'')$ is the event that $\xi_{u, e}^{\mathrm{acyclic}}\left(t, t''\right)$ and $\xi_{u, e}^{\mathrm{path}}\left(t, t''\right)$ both hold.
	\end{definition}
	
	Before lower bounding the probability of our cut edge events, we will need the following lemma. 
	
	\begin{lemma} \label{integral_bound}
		For all $n \in \mathbb{N}$ sufficiently large, $\mu > 0$ and $\Delta > 0$ such that $\mu \Delta > 2 \log 2$,
		\begin{equation*}
			\int_{0}^{\Delta} \frac{\mu e^{\mu x}}{e^{\mu \Delta} - 1} \left(p + (1-p)e^{-\mu x}\right)^{r - 1} \text{d} x \leq e^{-{\mu \Delta}/{2}} + O\left((\log n)^{-3}\right).
		\end{equation*}
	\end{lemma}
	
	\begin{proof}
		Since $e^{-y}$ is strictly decreasing in $y$ and for all $y \geq 0$,  $e^{-y} \leq 1$, we have that, in particular, $p + (1-p)e^{-\mu x} \leq 1$ for $x \in [0, {\Delta}/{2}]$, and $p + (1-p)e^{-\mu x} \leq p + (1-p)e^{-{\mu \Delta}/{2}}$ for $x \in [{\Delta}/{2}, \Delta]$. 
		
		Also, since $\mu \Delta > 2 \log 2$, then $p + (1-p) e^{-{\mu \Delta}/{2}} < {(1+p)}/{2} \leq {(1+\pu)}/{2}$, which is also strictly less than $1$. In particular, by our choice of $r$, namely $r \coloneqq \dfrac{3 \log_{d-1} \log n}{\log_{d-1} (2 /  (1 + \pu))}$, there exists some $C > 0$ such that for all $n$ sufficiently large, $\left({(1+\pu)}/{2}\right)^{r - 1} < C(\log n)^{-3}$. 
        
        Therefore, splitting the integral over these two regions of $x$, we have that, for all $n$ sufficiently large,
		\begin{align*}
			\int_{0}^{\Delta} \frac{\mu e^{\mu x}}{e^{\mu \Delta} - 1} \left(p + (1-p)e^{-\mu x}\right)^{r - 1} \text{d} x &\leq \frac{1}{e^{{\mu \Delta}/{2}} + 1} + \int_{{\Delta}/{2}}^{\Delta} \frac{\mu e^{\mu x}}{e^{\mu \Delta} - 1} \left(\frac{1+\pu}{2}\right)^{r - 1} \text{d} x \nonumber \\ 
			&\leq e^{-{\mu \Delta}/{2}} + C(\log n)^{-3}
		\end{align*}
		and hence the result follows.
	\end{proof}
	
	We are now in a position to prove the following result.
	
	\begin{lemma} \label{lem:cut_edge_prob}
		Let $n \in \mathbb{N}$ be sufficiently large and $G = (V, E)$ a $d$-regular graph on $n$ vertices. For all $\mu > 0$, $t \geq 0$, $\Delta > 0$ such that $\mu \Delta > 2 \log 2$, $e = \{u, u'\} \in E$ such that $u$ does not lie on a cycle of length less than $r$ and $B_R(u)$ has at most one cycle, and $(\eta, v) \in \{0, 1\}^E \times V$,
		\begin{align*}
			&\ \mathbb{P}\left[\xi_{u, e}^{\mathrm{cut}} (t, \tau_e(t, t + \Delta))\mid \Ce\left(t, t + \Delta\right), (\eta_t, X_t) = (\eta, v)\right] \\
			\geq&\ \mathbb{P}\left[S_{[t, t + \Delta]} \left(R, K\right) \mid \left(\eta_t, X_t\right) = \left(\eta, v\right)\right] \left(1 - (2K + 1) e^{-{\mu \Delta}/{2}} - O\left((\log n)^{-3}\right)\right).
		\end{align*}
	\end{lemma}
	
	\begin{proof}
		First note that, for $n$ sufficiently large, we have that $R > r > 1$ and hence for all $0 \leq x \leq \Delta$, the event $\xi_{u, e}^{\mathrm{cut}} (t, t + x)$ is independent of the clock $\Ce$ and uniform $U_e$ associated with the edge $e$. Indeed, $\xi_{u, e}^{\mathrm{cut}} (t, t + x)$ is the event that $\xi_{u, e}^{\mathrm{acyclic}}\left(t, t + x\right)$ and $\xi_{u, e}^{\mathrm{path}}\left(t, t + x\right)$ both hold, where $\xi_{u, e}^{\mathrm{acyclic}}\left(t, t + x\right)$ is an event about clocks and uniforms during $(t, t + x)$ for edges in $E\left(B_r (u)\right)$ distinct from $e$ while $\xi_{u, e}^{\mathrm{path}}\left(t, t + x\right)$ is an event about the clocks and uniforms up to time $t + x$ for edges in $E\left(B_R (u)\right) \setminus E\left(B_r (u)\right)$. 
		
		Hence for all $0 \leq x \leq \Delta$, we have that the events $\xi_{u, e}^{\mathrm{cut}} (t, t + x)$ and $\Ce(t, t + \Delta)$ are independent, the events $\xi_{u, e}^{\mathrm{acyclic}}\left(t, t + x\right)$ and $\xi_{u, e}^{\mathrm{path}}\left(t, t + x\right)$ are independent, and lastly the event $\xi_{u, e}^{\mathrm{acyclic}}\left(t, t + x\right)$ is independent of the state at time $t$. Therefore, in conjunction with Lemma \ref{poisson_pdf_last_arrival}, 
		\begin{align}
			&\ \mathbb{P} \left[\xi_{u, e}^{\mathrm{cut}}  \left(t, \tau_e(t, t + \Delta)\right) \mid \mathsf{C}_{e}\left(t, t + \Delta\right), (\eta_t, X_t) = (\eta, v)\right] \nonumber \\ 
			=& \int_{0}^{\Delta} \!\! \frac{\mu e^{\mu x}}{e^{\mu \Delta} - 1} \ \mathbb{P}\left[\xi_{u, e}^{\mathrm{cut}} (t, t + x) \mid (\eta_t, X_t) = (\eta, v) \right] \text{d} x \nonumber \\
			=& \int_{0}^{\Delta} \!\! \frac{\mu e^{\mu x}}{e^{\mu \Delta} - 1} \ \mathbb{P}\left[\xi_{u, e}^{\mathrm{acyclic}} (t, t+x) \right] \mathbb{P}\left[\xi_{u, e}^{\mathrm{path}}(t, t+x) \mid (\eta_t, X_t) = (\eta, v) \right] \text{d} x. \label{clp:eq1}
		\end{align}
		
		Recall from Definition \ref{def:external-K-sparse-event} that $S_{[t, t + \Delta]} (R, K)$ is the event that $S_{t+x} (R, K)$ holds for all $0 \leq x \leq \Delta$. Then, $\mathbb{P}\left[S_{t+x} \left(R, K\right) \mid \left(\eta_t, X_t\right) = \left(\eta, v\right)\right] \geq \mathbb{P}\left[S_{[t, t + \Delta]} \left(R, K\right) \mid \left(\eta_t, X_t\right) = \left(\eta, v\right)\right]$. 
        
        Together with Lemmas \ref{lem:acyclic_prob} and \ref{lem:path_prob}, we have that
		\begin{align}
			&\ \mathbb{P}\left[\xi_{u, e}^{\mathrm{acyclic}} (t, t+x) \right] \mathbb{P}\left[\xi_{u, e}^{\mathrm{path}}(t, t+x) \mid (\eta_t, X_t) = (\eta, v) \right] \nonumber \\ 
			\geq&\ \mathbb{P}\left[S_{t+x} \left(R, K\right) \mid \left(\eta_t, X_t\right) = \left(\eta, v\right)\right] \left(1 - 2K \left(p + (1-p)e^{-\mu x}\right)^{{R}/{2}}\right) \left(1 - \left(p + (1-p)e^{-\mu x}\right)^{r-1} \right). \nonumber \\
			>&\ \mathbb{P}\left[S_{[t, t + \Delta]} \left(R, K\right) \mid \left(\eta_t, X_t\right) = \left(\eta, v\right)\right] \left(1 - (2K + 1) \left(p + (1-p)e^{-\mu x}\right)^{r - 1}\right)\label{clp:eq2}
		\end{align}
		where the last inequality follows from the facts that $(1 - z)(1 - y) > 1 - z - y$ for $z, y > 0$, and $R/2 \geq r - 1$ for sufficiently large $n$.
		
		Hence, by (\ref{clp:eq1}) and (\ref{clp:eq2}), we obtain that
		\begin{align*}
			&\ \mathbb{P} \left[\xi_{u, e}^{\mathrm{cut}}  \left(t, \tau_e\right) \mid \mathsf{C}_{e}\left(t, t + \Delta\right), (\eta_t, X_t) = (\eta, v)\right] \\ 
			\geq& \ \mathbb{P}\left[S_{[t, t + \Delta]} \left(R, K\right) \mid \left(\eta_t, X_t\right) = \left(\eta, v\right)\right] \left(1 - (2K + 1) \int_{0}^{\Delta} \!\! \frac{\mu e^{\mu x}}{e^{\mu \Delta} - 1} \ \left(p + (1-p)e^{-\mu x}\right)^{r - 1} \ \mathrm{d} x \right).
		\end{align*}
		
		The result follows by Lemma \ref{integral_bound}.
	\end{proof}
	
	\subsection{Proof of Lemma \ref{lem:stationary_lw_bdd}}
	
	We are finally in a position to lower bound the probability of the event $\xi_{u \not\to u'} (t, t')$.

    \stationaryLwBdd*
	
	\begin{proof}
        Let $e \coloneqq \{u, u'\}$. If the edge $e$ refreshes during $(t, t')$, i.e. the event $\Ce(t, t')$ holds, and during the last ring $\tau_e (t, t')$ both $U_e (\tau_e (t, t')) > \pmin$ and $\xi_{u, e}^{\mathrm{cut}} \left(t, \tau_e (t, t')\right)$ hold, then $e$ is a cut edge in the configuration at time $\tau_e (t, t')$, then $e$ is refreshed closed at time $\tau_e (t, t')$ and remains closed until time $t'$.

        Therefore, the events $V_u (t') = u'$, $\Ce(t, t')$, $U_e (\tau_e (t, t')) > \pmin$ and $\xi_{u, e}^{\mathrm{cut}} \left(t, \tau_e (t, t')\right)$ imply the event $\xi_{u \not\to u'} (t, t')$. Hence,
        \begin{align*}
            &\ \mathbb{P}\left[\xi_{u \not\to u'} (t, t') \mid (\eta_t, X_t) = (\eta, v)\right] \\
            \geq&\ \mathbb{P} \left[V_u (t') = u', \Ce(t, t'), U_e (\tau_e (t, t')) > \pmin, \xi_{u, e}^{\mathrm{cut}} \left(t, \tau_e (t, t')\right) \mid (\eta_t, X_t) = (\eta, v)\right]
        \end{align*}

        The events $V_u (t') = u'$, $\Ce (t, t')$ and $U_e \left(\tau_e (t, t')\right) > \pmin$ concern the ringing of Poisson clocks and sampling of uniforms during $(t, t')$, which are independent of the state at time $t$. Furthermore, the random variable $V_u$ is independent of all environment random variables, namely $U_e$ and $\Ce$, and in particular is independent of the event $\xi_{u, e}^{\mathrm{cut}} \left(t, \tau_e (t, t')\right)$ which only concerns environment random variables. Also, the event $U_e \left(\tau_e (t, t')\right) > \pmin$ is independent of $\xi_{u, e}^{\mathrm{cut}} \left(t, \tau_e (t, t')\right)$, since $\xi_{u, e}^{\mathrm{cut}} \left(t, \tau_e (t, t')\right)$ by construction is an event about the ringing of clocks and sampling of uniforms associated with edges distinct from $e$ up to the last ring time $\tau_e (t, t')$ of $e$ in the time interval $(t, t')$. We therefore have that,
		\begin{align*}
            &\ \mathbb{P} \left[V_u (t') = u', \Ce(t, t'), U_e (\tau_e (t, t')) > \pmin, \xi_{u, e}^{\mathrm{cut}} \left(t, \tau_e (t, t')\right) \mid (\eta_t, X_t) = (\eta, v)\right] \\
            =&\ \mathbb{P}\left[V_{u}\left(t'\right) = u'\right] \mathbb{P}\left[U_e \left(\tau_e (t, t')\right) > \pmin \mid \Ce\left(t,t'\right)\right]\mathbb{P}\left[\Ce \left(t,t'\right)\right] \\
			&\ \mathbb{P}\left[\xi_{u, e}^{\mathrm{cut}} \left(t, \tau_e (t, t')\right) \mid \Ce \left(t, t'\right), (\eta_t, X_t) = (\eta, v)\right].
		\end{align*}
		
		Since $V_{u}\left(t'\right)$ is uniform on the $d$ neighbours of $u$, then $\mathbb{P}\left[V_{u}\left(t'\right) = u'\right] = 1/d$. By the probability of a Poisson clock ringing in a time interval, we have $\mathbb{P}\left[\Ce \left(t,t'\right)\right] = 1 - e^{-\mu \left(t' - t\right)}$. Also, since $U_e$ is uniform on $(0, 1)$, $\mathbb{P}\left[U_e \left(\tau_e(t, t')\right) > \pmin \mid \Ce\left(t,t'\right)\right] = 1 - \pmin$. Lastly, by Lemma \ref{lem:cut_edge_prob}, noting that all necessary conditions are met, we have that for $n$ sufficiently large,
		\begin{align*}
			&\ \mathbb{P}\left[\xi_{u, e}^{\mathrm{cut}} (t, \tau_e)\mid \Ce\left(t, t'\right), (\eta_t, X_t) = (\eta, v)\right] \\
			\geq&\ \mathbb{P}\left[S_{[t, t']} \left(R, K\right) \mid \left(\eta_t, X_t\right) = \left(\eta, v\right)\right] \left(1 - \frac{1}{(\log n)^2} - (2K + 1) e^{-\mu (t' - t)/2}\right).
		\end{align*}
		
		The result follows by combining all these lower bounds together, noting that \[(1 - e^{-a})(1 - c - (2b + 1) e^{-a/2}) \geq 1 - c - 2(b + 1)e^{-a/2}\] for all $a, b > 0$ and $0 < c < 1$.
	\end{proof}

    \section{Global properties of random regular graphs: Lemmas \ref{lem:phase1_walk_counts} and \ref{lem:sparse_good_set_size}}
    \label{sec:phase1_walk_counts}\label{sec:sparse_good_set_size}
	
	In this section we study several properties of random regular graphs, in order to carefully control the number of walks (under some constraints) which avoid small cycles in the graph, namely those of order $\rr$. In particular, we will prove Lemmas \ref{lem:phase1_walk_counts} and \ref{lem:sparse_good_set_size}.

    \subsection{Count of \texorpdfstring{$(h, i)$}{(h, i)}-constrained walks in infinite \texorpdfstring{$d$}{d}-regular and \texorpdfstring{$(d-1)$}{(d-1)}-ary trees}
	
    By $\mathcal{T}$ we denote the infinite $d$-regular tree with a designated vertex $\rho$ as the root. Also, by $\widetilde{\mathcal{T}}$ we denote the infinite $(d-1)$-ary tree with a designated vertex $\widetilde{\rho}$ as the root.
	
	We will study the counts of $(h, i)$-constrained walks as in Definition \ref{def:h-i-constrained-walk}. We first note the following observation.
	
	\begin{remark} \label{walk_count_root_vertex_equiv}
		By symmetry in the infinite $d$-regular tree $\mathcal{T}$, given any vertex $x \in V(\mathcal{T})$ the number of walks from $x$ of length $h+2i$ which end at a distance $h$ away from $x$ is equal to the number of such walks from the root $\rho$, namely $\omega_{h,i}$.
	\end{remark}
	
	We have the following lemma on the number of $(h, i)$-constrained walks on the $d$-regular and $(d-1)$-ary infinite trees.
	
	\begin{lemma} \label{lem:walk_counts}
		For all $h, i \in \mathbb{N}$, $\omega_{h,i} > \widetilde{\omega}_{h,i} = \frac{h+1}{h+i+1} \binom{h + 2i}{h+i} (d-1)^{h+i}$.
	\end{lemma}
	
	\begin{proof}
		First note that we can embed the infinite $(d-1)$-ary tree into the infinite $d$-regular tree, with the roots of both trees identified. Hence every $(h, i)$-constrained walk on $\widetilde{\mathcal{T}}$ is uniquely mapped to an $(h, i)$-constrained walk on $\mathcal{T}$ under this embedding. Therefore, $\omega_{h,i} > \widetilde{\omega}_{h,i}$. 
		
		We will next determine $\widetilde{\omega}_{h,i}$. Starting at the root $\widetilde{\rho}$ of $\widetilde{\mathcal{T}}$, to reach depth $h$ in a walk of length $h +2i$, exactly $h + i$ transitions must increase the depth by $1$, and $i$ transitions must decrease the depth by $1$. The number of such walks is equivalent to the number of walks of length $h + 2i$ on $\mathbb{Z}$ starting at $0$ and ending at $h$ which never become negative. By Bertrand's Ballot Theorem, the number of such walks is known to be $\frac{h+1}{h+i+1} \binom{h + 2i}{h+i}$. For each of the $h + i$ downward transitions we have $d-1$ choices, since each vertex has exactly $d-1$ children.
		
		Therefore the total number of such walks is $\frac{h+1}{h+i+1} \binom{h + 2i}{h+i} (d-1)^{h+i}$. The result follows.
	\end{proof}
	
	\begin{remark} \label{rem:walk_path_ratio}
		By symmetry in $\mathcal{T}$ and the fact that in a tree there is a unique path of length $h$ from the root to a vertex at depth $h$, we note that given a set $\mathcal{P}$ of paths from the root of $\mathcal{T}$ to depth $h$, the number of $(h, i)$-constrained walks which follow a path in $\mathcal{P}$ is $\frac{|\mathcal{P}|}{d(d-1)^{h-1}} \ \omega_{h,i}$. 
		
		Analogous counts hold for $\widetilde{\mathcal{T}}$.
	\end{remark}
	
	\subsection{Counts of \texorpdfstring{$(h, i)$}{(h, i)}-constrained walks avoiding small cycles in regular graph} \label{sec:rrg_prop2}
	
	Recall that we will view random walks on a regular graph starting at a vertex $u$ as walks on the non-backtracking walk tree (NBWT) $\mathcal{T}_u$ starting at the root, as introduced in Section \ref{sec:extended_proof_overview}. We  will use NBWTs as a convenient tool to bound the number of $(h,i)$-constrained walks from a given vertex in a regular graph that avoid a given subset of vertices, namely vertices which lie on small cycles of length at most $\rr$. 
	
	For any two vertices $u$ and $v$ in a graph $G$, let $w_{h,i}(v, u)$ be the set of $(h, i)$-constrained walks in the NBWT $\mathcal{T}_v$ which pass through a vertex labelled $u$. The following lemma upper bounds the total number of $(h,i)$-constrained walks that pass through a vertex labelled $u$ across all NBWTs of $G$.
	
	\begin{lemma} \label{path_count_lem}
		Let $h, i \in \mathbb{N}$ and $u \in V$. Then, \[\sum_{v \in V} |w_{h,i}(v,u)| \leq \frac{1}{2}(h+i+1)(h+3i+2) \ \omega_{h,i}.\]
	\end{lemma}
	
	\begin{proof}
		Fix a vertex $v$ and a walk in $\mathcal{T}_v$ of length $h + 2i$ from the root to depth $h$ that passes through a vertex $\hat\rho$ labelled $u$. Re-rooting the tree $\mathcal{T}_v$ at $\hat{\rho}$ results in another labelled tree isomorphic to $\mathcal{T}_u$, with the image of the walk passing through the new root $\hat{\rho}$.
		
		In particular, the walks $w_{h,i}(v,u)$ in $\mathcal{T}_v$ correspond to walks in $\mathcal{T}_u$ of length $h + 2i$ which:
		\begin{enumerate}[i.]
			\item pass through the root $\hat{\rho}$ of $\mathcal{T}_u$,
			\item start at some vertex $x \in V(\mathcal{T}_u)$ labelled $v$,
			\item and end at a vertex distance $h$ away from $x$ in the tree $\mathcal{T}_u$.
		\end{enumerate}
		
		Note that a walk in $w_{h,i}(v,u)$ on $\mathcal{T}_v$ may correspond to multiple such walks on $\mathcal{T}_u$, depending on how many times a vertex labelled $u$ appears on the walk. 
		
		Therefore to upper bound $\sum_{v \in V} |w_{h,i} (v,u)|$, it suffices to count the number of walks of length $h+2i$ in $\mathcal{T}_u$ which start at any vertex, pass through the root, and end at a vertex distance $h$ away from the starting vertex. Let $\W$ be the set of all such walks in $\mathcal{T}_u$.
		
		Given $r \in \mathbb{N}$ and $x \in V(\mathcal{T}_u)$, let $\partial B(x,r)$ be the set of vertices in $\mathcal{T}_u$ which are distance $r$ away from $x$. Then $|\partial B(x, r)| = d(d-1)^{r-1}$ by regularity, and hence the cardinality is invariant of $x$. Also, given $t \in \mathbb{N}$ and $x, y \in V(\mathcal{T}_u)$, let $\W_{x, t} (y)$ be the set of all walks of length $h+2i$ in $\mathcal{T}_u$ starting at $x$ which pass through $y$ at step $t$.
		
		Now, any walk in $\W$ must start at most a distance $h+i$ away from the root $\hat{\rho}$ of $\mathcal{T}_u$, since they have to pass through the root and end a distance $h$ from the start. Therefore, for every walk in $\mathcal{W}$ there exists $r \in \mathbb{N}$ such that $0 \leq r \leq h + i$, the walk starts at a vertex $x \in \partial B(\hat\rho, r)$, and the walk passes through the root $\hat{\rho}$ at step $t \in \mathbb{N}$ where $r \leq t \leq h + 2i$. In particular we have that 
		\begin{equation} \label{eq:pcl1}
			\W = \bigcup_{r = 0}^{h+i} \bigcup_{t = r}^{h+2i} \bigcup_{x \in \partial B(\hat\rho, r)} \W_{x,t}(\hat\rho).
		\end{equation}
		
		Let $\W_{x,t}^r$ be the union of $\W_{x,t} (y)$ for all $y \in \partial B(x,r)$. Note that for any two distinct vertices $y, z \in V(\mathcal{T}_u)$, the sets $\W_{x,t} (y)$ and $\W_{x,t} (z)$ are disjoint, since the walks differ at step $t$. Therefore, 
		\begin{equation} \label{eq:pcl2}
			\left|\W_{x,t}^r\right| = \sum_{y \in \partial B(x,r)} \left|\W_{x,t} (y)\right|.
		\end{equation}
		
		By symmetry in $\mathcal{T}_u$, for any two vertices $y, z \in \partial B(x,r)$, the number of walks of length $h+2i$ which start at $x$ and end at distance $h$ from $x$ while passing through $y$ and $z$ respectively at step $t$ is equal, which is to say that $|\W_{x,t} (y)| = |\W_{x,t} (z)|$. Therefore in conjunction with (\ref{eq:pcl2}) and the fact that $|\partial B(x, r)| = d(d-1)^{r-1}$, for all $y \in \partial B(x,r)$ we have that
		\begin{equation} \label{eq:pcl3}
			|\W_{x,t}^r| = d(d-1)^{r-1} |\W_{x,t} (y)|.
		\end{equation}
		
		Furthermore, $\W_{x,t}^r$ is a subset of all walks of length $h+2i$ in $\mathcal{T}_u$ starting at $x$ which end a distance $h$ away from $x$, and hence by Remark \ref{walk_count_root_vertex_equiv} we have that 
		\begin{equation} \label{eq:pcl4}
			|\W_{x,t}^r| \leq \omega_{h,i}.
		\end{equation}
		
		In particular, if the root $\hat{\rho}$ is in $\partial B (x,r)$ then by (\ref{eq:pcl3}) and (\ref{eq:pcl4}) we have that
		\begin{equation} \label{eq:pcl5}
			|\W_{x,t} (\hat\rho)| \leq \frac{\omega_{h,i}}{d(d-1)^{r-1}}.
		\end{equation}
		
		Hence from (\ref{eq:pcl1}) and (\ref{eq:pcl5}) along with the fact that $|\partial B(\hat\rho, r)| = d(d-1)^{r-1}$, we have that 
		\begin{align*}
			|\W| \leq \sum_{r = 0}^{h+i} \sum_{t = r}^{h+2i} \sum_{x \in \partial B(\hat\rho, r)} |\W_{x,t} (\hat\rho)| \leq \sum_{r = 0}^{h+i} \sum_{t = r}^{h+2i} \sum_{x \in \partial B(\hat\rho, r)} \frac{\omega_{h,i}}{d(d-1)^{r-1}} = \frac{1}{2}(h+i+1)(h+3i+2) \ \omega_{h,i}
		\end{align*}
		as required.
	\end{proof}
	
	Recall that a walk on a graph is said to be $r$-acyclic (Definition \ref{def:r-acyclic-walk}) if it never visits a cycle of length less than $\rr$. We say that a vertex $u$ is $(h, i)$-sparse (Definition \ref{def:h-i-sparse}) if at least $\left(1 - {1}/{(\log n)^3}\right) \omega_{h, i}$ of the $(h, i)$-constrained walks from $u$ are $r$-acyclic. 
	
	The previous lemma, in conjunction with an upper bound on the number of vertices on cycles of length less than $\rr$, allows us to establish bounds on the number of $(h, i)$-sparse vertices in a random regular graph, for appropriate values of $h$ and $i$. However, we first require upper bounds on the number of vertices contained in such small cycles. Given a random $d$-regular graph, by $\mathcal{C}_k$ for $k \geq 3$ we will denote the number of cycles of length $k$. The following result establishes that for $r, d \geq 3$ such that $\sqrt{r}(d-1)^{3r/2 - 1} = o(n)$, the number of cycles $\mathcal{C}_3, \dots, \mathcal{C}_r$ are asymptotically jointly distributed as independent Poisson random variables $\mathcal{Z}_3, \dots, \mathcal{Z}_r$ with means $\frac{(d-1)^k}{2k}$ for $3 \leq k \leq r$.
	
	\begin{theorem}[Theorem 11, \cite{Johnson_2015}] \label{thm:cycle_size_count_dist}
		Let $G$ be a random $d$-regular graph on $n$ vertices with cycle counts $\left(\mathcal{C}_k \colon k \geq 3\right)$. Let $\left(\mathcal{Z}_k \colon k \geq 3\right)$ be independent Poisson random variables with $\mathbb{E}\left[\mathcal{Z}_k\right] = \frac{(d-1)^k}{2k}$. For any $n \geq 1$ and $r, d \geq 3$, $\TVD{\left(\mathcal{C}_3, \dots, \mathcal{C}_r\right)}{\left(\mathcal{Z}_3, \dots, \mathcal{Z}_r\right)} = O\left(\frac{\sqrt{r} (d-1)^{3r/2 - 1}}{n}\right)$.
	\end{theorem}
	
	The following corollary follows immediately from Theorem \ref{thm:cycle_size_count_dist}. 
	
	\begin{corollary}\label{cor:cycle_size_max_vertex_count}
		Let $G$ be a random $d$-regular graph on $n$ vertices with $\mathcal{X}_{\leq k} (G)$ vertices on cycles of length at most $k$. For any $n \geq 1$ and $r, d \geq 3$ such that $\sqrt{r}(d-1)^{3r/2 - 1} = o(n)$, with probability $1 - o(1)$ over $G \sim \rrg$, $\mathcal{X}_{\leq r} (G) \leq (\log n) (d-1)^r$.
	\end{corollary}
	
	We are now in a position to prove the main result for this section.
	
	\rrgGlblGeom*
	
	\begin{proof}
		Since $r = O(\log \log n)$ then $\sqrt{r} (d-1)^{3r/2 - 1} = o(n)$ and by Corollary \ref{cor:cycle_size_max_vertex_count}, the number of vertices on cycles of length at most $r$ is, with probability $1 - o(1)$ over $G \sim \rrg$, at most $(\log n)(d-1)^r$. Let $\mathcal{B}$ be the set of vertices contained in such cycles. 
		
		Fix $h, i \in \mathbb{N}$ such that $h, i \leq \left({4}/{\pmin}\right) \log_{d - 1} n$. Given a vertex $u \in \mathcal{B}$, for every vertex $v \in V$ let $w_{h,i}(v, u)$ be the set of $(h, i)$-constrained walks from $v$ which pass through $u$. 
		
		For a vertex $v$ to be $(h, i)$-sparse, it suffices that $\sum_{u \in \mathcal{B}} |w_{h,i}(v,u)| < ({1}/{(\log n)^3}) \omega_{h,i}$, which is to say that the $(h, i)$-constrained walks from $v$ which pass through a vertex in $\mathcal{B}$ only make up at most a ${1}/{(\log n)^3}$ fraction of all possible $(h, i)$-constrained walks from $v$. 
		
		Define $V_\textrm{bad} (h, i) \coloneqq \{v \in V \colon \sum_{u \in \mathcal{B}} |w_{h,i}(v,u)| \geq ({1}/{(\log n)^3}) \omega_{h,i}\}$. By the previous remark, every vertex in $V \setminus V_\textrm{bad} (h, i)$ is $(h, i)$-sparse. Hence it suffices to upper bound the size of $V_\textrm{bad} (h, i)$. By Lemma \ref{path_count_lem} we have that
		\begin{equation*}
			\sum_{u \in \mathcal{B}} \sum_{v \in V} |w_{h,i}(v,u)| \leq  |\mathcal{B}| (1/2)(h+i+1)(h+3i+2) \ \omega_{h,i} \leq \left(1 + {8}/{\pmin}\right)^2 |\mathcal{B}| (\log_{d - 1} n)^2 \omega_{h,i}
		\end{equation*}
		where the last inequality follows from $h, i \leq ({4}/{\pmin}) \log_{d - 1} n$. We also have the following lower bound,
		\begin{equation*}
			\sum_{u \in \mathcal{B}} \sum_{v \in V} |w_{h,i}(v,u)| \geq \sum_{v \in V_\mathrm{bad} (h, i)} \sum_{u \in \mathcal{B}} |w_{h,i}(v,u)| \geq |V_\textrm{bad} (h, i)| ({1}/{(\log n)^3}) \omega_{h,i}
		\end{equation*}
		where the last inequality follows by the definition of $V_\textrm{bad} (h, i)$.
		
		Therefore $|V_\textrm{bad} (h, i)| ({1}/{(\log n)^3}) \omega_{h,i} \leq \left(1 + {8}/{\pmin}\right)^2 |\mathcal{B}| (\log_{d - 1} n)^2 \omega_{h,i}$ and hence \[|V_\textrm{bad} (h, i)| \leq \left(\frac{1 + {8}/{\pmin}}{\log(d - 1)}\right)^2 (d - 1)^r (\log n)^6\] since $\left|\mathcal{B}\right| \leq (\log n) (d-1)^r$. Let $V_{\mathrm{bad}} = \bigcup\limits_{h, i \leq ({4}/{\pmin}) \log_{d-1} n} V_{\mathrm{bad}} (h, i)$. Then by the upper bound on each $V_{\mathrm{bad}} (h, i)$, we have that $\left|V_{\mathrm{bad}}\right| = O\left((d - 1)^r (\log n)^8\right)$. In particular, all the vertices in $V \setminus V_{\mathrm{bad}}$ are $(h, i)$-sparse for all $h, i \leq ({4}/{\pmin}) \log_{d - 1} n$ and hence the result follows.
	\end{proof}
	
	\subsection{Proof of Lemma \ref{lem:phase1_walk_counts}}
    \label{sec:rrg_prop}
	
	From Lemma \ref{lem:rrg_glbl_geom} we can prove the first of our two main results for this section.

    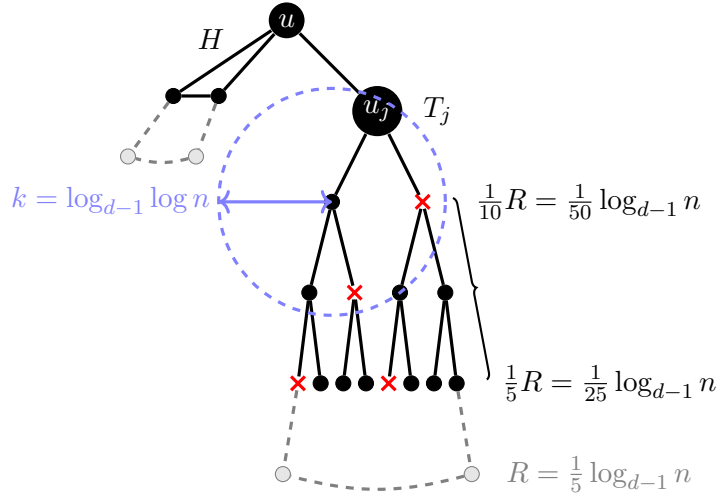
\begin{figure}[ht!]
			\centering
			\input{Figures/fig_vgood_walks}
			\vspace*{-1cm}
			\caption{Illustration of the paths on $T_j$ reaching $k$-roots which are also in the set $S$. The ball radius $\log_{d-1} \log n$ highlighted in blue illustrates that at suitable depths every vertex is a $k$-root. Furthermore, the vertices highlighted by \textcolor{red}{$\times$} are not in $S$.}
			\label{fig:vgood_walks}
		\end{figure}
	
	\phaseOneWalkCounts*
	
	\begin{proof}
		By Lemma \ref{lem:rrg_glbl_geom}, with probability $1 - o(1)$ over $G \sim \rrg$, there exist $n - O\left(\mathrm{polylog}(n)\right)$ vertices which are $(h, i)$-sparse for all $h, i \in \mathbb{N}$ such that $h, i \leq \left({4}/{\pmin}\right) \log_{d - 1} n$. 
        
        Furthermore, by Lemma 2.1 in \cite{Lubetzky_2010}, we have that, with probability $1 - o_n(1)$ over $G \sim \rrg$, every ball of radius $R$ in $G$ contains at most one cycle. Hence, with probability $1 - o_n(1)$ over $G \sim \rrg$, both of these properties hold. Fix such a graph $G$. 
        
        Let $S$ be the set of all vertices in $G$ which are $(h, i)$-sparse for all $h, i \in \mathbb{N}$ such that $h, i \leq \left({4}/{\pmin}\right) \log_{d - 1} n$.
		
		Fix a vertex $u \in V$ and consider the subgraph $B$ induced by the ball $B_R (u)$. Since there is at most one cycle in this subgraph, deleting $u$ from $B$ results in at least $d - 1$ components; in particular there must be $d - 2$ rooted trees $T_1, \dots, T_{d-2}$, whose roots are neighbours of $u$, call them $u_1, \dots, u_{d-2}$ respectively. Furthermore, the (at most) unique cycle in $B$ is entirely contained in $H \coloneqq B \setminus \cup_{j = 1}^{d-1} T_j$.
		
		Fix $1 \leq j \leq d-2$ and $h, i \in \mathbb{N}$ such that ${R}/{10} \leq h \leq {R}/{5}$ and $i \leq {R}/{5}$. The sub-tree $T_j$ corresponding to $u_j$ has degree $d - 1$ at the root and degree $d$ at all other non-leaf vertices. Hence $T_j$ is a $(d-1)$-ary tree of depth $R - 1$. Since $h + i \leq {2R}/{5} \leq R - 1$ and $(h, i)$-constrained walks never exceed depth $h + i$, the number of $(h, i)$-constrained walks on $T_j$ is equal to that on the infinite $(d-1)$-ary tree, namely $\widetilde{\omega}_{h, i}$.
		
		Consider an $(h, i)$-constrained walk on $T_j$. Suppose a vertex in this walk is contained in a cycle of length less than $r$ in $G$. Since the walk reaches a depth of at most $h + i \leq {2R}/{5}$ and the cycle has length at most $r$ (which is less than ${R}/{10}$ for all $n$ sufficiently large), then every vertex in the cycle is at distance at most ${R}/{2}$ away from $u_j$ in $G$. 
		
		Therefore, since $u_j$ is a neighbour of $u$, every vertex in the cycle is at distance at most $({R}/{2}) + 1$ away from $u$ in $G$. Therefore the cycle is entirely contained in the subgraph $B$ induced by the ball $B_R (u)$. However, this cycle contains at least one vertex in $T_j$, and the (at most) unique cycle in $B$ must be entirely contained in $H$, which is disjoint from $T_j$. Therefore such a cycle cannot exist, and hence every such walk in $T_j$ does not visit a vertex on a cycle of length less than $r$.
		
		Now for every vertex in $T_j$ at depth greater than $k = \log_{d-1} \log n$ but no more than $R - k - 1$, the ball radius $k$ around such a vertex is does not contain any cycle in $G$, and hence every such vertex is a $k$-root. This is illustrated in Figure \ref{fig:vgood_walks}. Given $n$ sufficiently large, it follows that for all $h \in \mathbb{N}$ such that ${R}/{10} \leq h \leq {R}/{5}$, all the vertices in $T_j$ at depth $h$ are $k$-roots. 
		
		Furthermore, since ${R}/{10} \leq h$ and $R = (1/5) \log_{d-1} n$, we have that $(d-1)^h = \Omega\left(n^{1/50}\right)$. Since $|V \setminus S| = O\left(\mathrm{polylog}(n)\right)$, it follows that for $n$ sufficiently large, $|V \setminus S| < \frac{1}{2} (d-1)^h$. Therefore, at least half of the vertices at depth $h$ in $T_j$ are $k$-roots and in $S$, i.e., in $\Vgood$. 
		
		Hence for all $h, i \in \mathbb{N}$ satisfying ${R}/{10} \leq h \leq {R}/{5}$ and $i \leq {R}/{5}$, at least half of the $(h, i)$-constrained walks from $u_j$ in $T_j$ end at a vertex in $\Vgood$, and we have shown that such a walk on $T_j$ never visits a cycle of length less than $r$. Since there are at least $d - 2$ such neighbours of $u$, the result follows.
	\end{proof}
	
	\subsection{Proof of Lemma \ref{lem:sparse_good_set_size}} \label{sec:rrg_prop3}
	
	For any two vertices $u, v \in V$ and $h \in \mathbb{N}$, by $\mathcal{P}_{h} (u, v)$ we denote the family of paths length $h$ between $u$ and $v$. Recall that $\Hmin \coloneqq \left\lfloor\log_{d-1} n\right\rfloor + 2 \left\lfloor\log_{d-1}\log n\right\rfloor$ and $\Hmax \coloneqq \left\lfloor\log_{d-1} n\right\rfloor + \left\lfloor (1/10)\log_{d-1} n\right\rfloor - 1$. Before proving Lemma \ref{lem:sparse_good_set_size}, we require the following result.
	
	\begin{lemma}[\cite{Lubetzky_2010}, Lemma 3.5]\label{lem:kroot_path_counts}
		With probability $1 - o_n(1)$ over $G \sim \rrg$, for any $u, v \in \Vroots$ and $h \in \mathbb{N}$ satisfying $d(u, v) > 2 k$ and $h \in [\Hmin, \Hmax]$: $\left|\mathcal{P}_{h} (u, v)\right| \geq \left(\frac{1-o_n(1)}{n}\right) d(d-1)^{h - 1}$.
	\end{lemma}
	
	We also take note of the following useful observations about the counts of paths between $k$-roots and the number of $(h, i)$-constrained walks between them.
	
	\begin{remark} \label{rem:k-root-walks-count}
		Observe that every simple path in a $d$-regular graph $G$ starting at a vertex $u$ corresponds to a simple path in the non-backtracking walk tree $\mathcal{T}_u$ starting at the root.
		
		Consequently, for any two $k$-roots $u$ and $v$ as in Lemma \ref{lem:kroot_path_counts} and $h \in \mathbb{N}$ such that $\Hmin \leq h \leq \Hmax$, each path in $\mathcal{P}_{h} (u, v)$ maps to a distinct path in $\mathcal{T}_u$ starting at the root. By Remark \ref{rem:walk_path_ratio}, for any $i \in \mathbb{N}$, the number of $(h, i)$-constrained walks which follow one of these paths in $\mathcal{T}_u$ is at least $\left(1-o(1)\right) (1/n) \omega_{h,i}$.
	\end{remark}
	
	Lastly, we will require the following result which establishes that, with probability $1 - o_n(1)$ over $G \sim \rrg$, most vertices are $k$-roots.
	
	\begin{theorem}[\cite{Lubetzky_2010}, Lemma 3.2] \label{thm:kroot_count}
		With probability $1 - o_n(1)$ over $G \sim \rrg$, there are $n - o(n)$ vertices in $G$ which are $k$-roots.
	\end{theorem}
	
	We are now in a position to prove Lemma \ref{lem:sparse_good_set_size}, which gives an analogue of Lemma \ref{lem:kroot_path_counts} for counts of $(h, i)$-constrained walks avoiding small cycles of length less than $r$. 
	
	\sparseGoodSetSize*
	
	\begin{proof}
		By Lemmas \ref{lem:kroot_path_counts} and \ref{thm:kroot_count}, with probability $1 - o(1)$ over $G \sim \rrg$, for a given $k$-root $u$ there exists a set $S'_u$ of $k$-roots with size $n - o(n)$, such that for all $h \in \mathbb{N}$ satisfying $\Hmin \leq h \leq \Hmax$ and for all $v \in S'_u$: $
		\left|\mathcal{P}_{h} (u, v)\right| \geq \left(1-o(1)\right) \left({1}/{n}\right) d(d-1)^{h - 1}$. In particular, for all $n$ sufficiently large, we have that $\left|\mathcal{P}_{h} (u, v)\right| \geq ({1}/{2n}) d(d-1)^{h - 1}$.
		
		Fix such a graph $G$, $u \in \Vgood$ and $v \in S'_u$. For all $i \in \mathbb{N}$, by Remark \ref{rem:k-root-walks-count} we have that the  paths in $\mathcal{P}_h(u,v)$
        correspond to $({1}/{2n}) \omega_{h, i}$ walks from $u$ which are $(h, i)$-constrained and end at $v$.
		
		Let $\W_{h, i} (u, v)$ correspond to this collection of walks. By the previous remarks, $\left|\W_{h, i} (u, v)\right| > ({1}/{2n}) \omega_{h, i}$. Clearly for any two distinct vertices $v, w \in S'_u$, these collections are disjoint (namely since they end at distinct vertices). 
		Let $\W_{h, i} (u)$ be the collection of all $(h, i)$-constrained walks from $u$, where $\left|\W_{h, i} (u)\right| = \omega_{h, i}$. 
		
		Next fix $h, i \in \mathbb{N}$ such that $\Hmin \leq h \leq \Hmax$ and $i \leq \left({4}/{\pmin}\right) \log_{d - 1} n$. Since $u \in \Vgood$ then it is $(h, i)$-sparse and hence at most $\left({1}/{(\log n)^3}\right) \omega_{h, i}$ of the walks in $\W_{h, i} (u)$ are not $r$-acyclic. Let this collection of ``bad" walks be $\W^\mathrm{bad}_{h, i} (u)$, and for all $v \in S'_u$ let $\W^\mathrm{bad}_{h, i} (u, v) \coloneqq \W^\mathrm{bad}_{h, i} (u) \cap \W_{h, i} (u, v)$. Note that all of these sets are pairwise disjoint and that their union is a subset of $\W^\mathrm{bad}_{h, i} (u)$.
		
		Let $S^\mathrm{bad}_u (h, i)$ be the subset of $S'_u$ such that $v \in S^\mathrm{bad}_u (h, i)$ if, and only if, $\left|\W^\mathrm{bad}_{h, i} (u, v)\right| > \frac{1}{2} \left|\W_{h, i} (u, v)\right|$. Then,
		\begin{align*}
			\left|S^\mathrm{bad}_u (h, i)\right| \frac{\omega_{h, i}}{4n} \leq \sum_{v \in S^\mathrm{bad}_u (h, i)} \frac{\left|\W_{h, i} (u, v)\right|}{2} < \sum_{v \in S^\mathrm{bad}_u (h, i)} \left|\W^\mathrm{bad}_{h, i} (u, v)\right| \leq \left|\W^\mathrm{bad}_{h, i} (u)\right| < \frac{\omega_{h, i}}{(\log n)^3}
		\end{align*}
		and hence $\left|S^\mathrm{bad}_u (h, i)\right| = O\left({n}/{(\log n)^3}\right)$. 
		
		Define $S_u \coloneqq S'_u \setminus \bigcup\limits_{h, i} S^\mathrm{bad}_u (h, i)$. Since $\left|S^\mathrm{bad}_u (h, i)\right| = O\left({n}/{(\log n)^3}\right)$ for all $\Hmin \leq h \leq \hmax$ and $i \leq ({4}/{\pmin}) \log_{d - 1} n$, then $|S_u| \geq |S'_u| - O\left({n}/{(\log n)}\right)$. Since $|S'_u| = n - o(n)$ it follows that $|S_u| = n - o(n)$. 
		
		We have therefore constructed a set $S_u$ of size $n - o(n)$, such that for all $\Hmin \leq h \leq \hmax$ and $i \leq ({4}/{\pmin}) \log_{d - 1} n$ and for all $v \in S_u$, there are at least $({1}/{4 n}) \omega_{h, i}$ walks from $u$ to $v$ which are $(h, i)$-constrained and $r$-acyclic. The result follows.
	\end{proof}

    \section{Proof of Lemma \ref{lem:simplex_integral}} \label{appendix:A}
	
	\simplexIntegralLemma*
	
	\begin{proof}
		Fix $c > 0$ and $c' > c$. We will eventually choose an appropriate value of $c$. Let $\delta \coloneqq {c'}/{(\log n)}$.  Let $x_0, \dots, x_\alpha \in [0, 1]$ be the shifted and rescaled variables satisfying $x_i = \frac{T_i - \delta}{T - (\alpha + 1)\delta}$ for all $0 \leq i \leq \alpha$, and $\sum_{i=0}^\alpha x_i = 1$. The Jacobian for this change of variables satisfies $\mathrm{d} \overrightarrow{T_\alpha} = \left(T - (\alpha + 1)\delta\right)^\alpha \ \mathrm{d} \overrightarrow{x}$, and the simplex is bijectively mapped onto the unit simplex \[\mathcal{K}_{\alpha}
		\coloneqq \left\{x_0, \dots, x_\alpha \in [0, 1] \colon x_0 + \dots + x_\alpha = 1 \right\}\] with volume $\mathrm{Vol}(\mathcal{K}_\alpha) = \frac{1}{\alpha!}$. Let $\mu_0 \coloneqq \varepsilon \log n$ and $\lambda \coloneqq \frac{1}{2} \mu_0 \left(T - (\alpha + 1) \delta\right)$. Then, for all $0 \leq i \leq \alpha$, we have that $1 - C n^{-(\varepsilon / 2) T_i} = 1 - C \exp\left(-(\mu_0 / 2) T_i\right) = 1 - C \exp\left(-(\mu_0 / 2) \delta - \lambda x_i\right)$, and therefore
		\begin{equation} \label{eq:sil1}
			\int_{\substack{T_0, \dots, T_\alpha \geq \delta \\ \sum_{i = 0}^\alpha T_i = T}} \ \prod_{i=0}^\alpha \left(1 - C n^{-(\varepsilon / 2) T_i}\right) \ \mathrm{d} \overrightarrow{T_\alpha} = \left(T - (\alpha + 1)\delta\right)^\alpha \int_{\mathcal{K}_{\alpha}} \prod_{i = 0}^\alpha \left(1 - C e^{-(\mu_0 / 2) \delta - \lambda x_i}\right) \ \mathrm{d} \overrightarrow{x}.
		\end{equation}
		
		Let $\gamma(c) \coloneqq C e^{-c \left({\varepsilon}/{2}\right)}$. Since $\mu_0 = \varepsilon \log n$ and $\delta = \frac{c'}{\log n} > \frac{c}{\log n}$, then, for all $0 \leq i \leq \alpha$, we have that $1 - C e^{-(\mu_0 / 2) \delta - \lambda x_i} \geq 1 - \gamma(c) e^{-\lambda x_i}$. By inclusion-exclusion, we also have that
		\begin{equation} \label{eq:sil2}
			\prod_{i = 0}^\alpha \left(1 - \gamma(c) e^{-\lambda x_i}\right)
			= \sum_{j = 0}^{\alpha+1} (-1)^j \left(\gamma(c)\right)^j \sum_{\substack{S \subseteq \{0, 1, \dots, \alpha\} \\ |S| = j}} e^{-\lambda \sum_{i\in S} x_i}
		\end{equation}
		
		By symmetry over the simplex, we have that each subset of size $j$ contributes equally to the summation above, and therefore we can rewrite (\ref{eq:sil2}) as 
		\begin{equation} \label{eq:sil3}
			\prod_{i = 0}^\alpha \left(1 - \gamma(c) e^{-\lambda x_i}\right)
			= \sum_{j = 0}^{\alpha+1} (-1)^j \left(\gamma(c)\right)^j \binom{\alpha + 1}{j} e^{-\lambda \sum_{i = 1}^j x_{i-1}}.
		\end{equation}
		
		Letting $M_j (\lambda) \coloneqq \left(\gamma(c)\right)^j \binom{\alpha + 1}{j} \int_{\mathcal{K}_{\alpha}} e^{-\lambda \sum_{i = 1}^j x_{i-1}} \ \mathrm{d} \overrightarrow{x}$ for $0 \leq j \leq \alpha + 1$, we therefore have that
		\begin{equation} \label{eq:sil4}
			\int_{\mathcal{K}_{\alpha}} \prod_{i = 0}^\alpha \left(1 - C e^{-(\mu_0 / 2) \delta - \lambda x_i}\right) \ \mathrm{d} \overrightarrow{x} \geq \sum_{j = 0}^{\alpha+1} (-1)^j M_j(\lambda).
		\end{equation}
		
		We will proceed to show that for $j \geq 1$, the differences $M_j (\lambda) - M_{j+1} (\lambda)$ are non-negative. For $1 \leq  j \leq \alpha$, let $X_j = x_0 + \dots + x_{j-1}$. Under the uniform distribution on $\mathcal{K}_{\alpha}$, $(x_0, \dots, x_\alpha)$ follow a $\mathrm{Dirichlet}(1, \dots, 1)$ distribution, and by the aggregation property for the Dirichlet distribution, $(X_j, 1 - X_j)$ follows a $\mathrm{Dirichlet}(j, \alpha - j + 1)$ distribution. From the marginal properties of Dirichlet distributions, we therefore have that $X_j$ follows a $\mathrm{Beta}(j, \alpha - j + 1)$ distribution with probability density function
		\begin{eqnarray*}
			f_{X_j}(x)
			= \frac{\alpha!}{(j-1)! \ (\alpha - j)!} x^{j - 1}(1 - x)^{\alpha - j}, & \forall x \in [0, 1].
		\end{eqnarray*}
		
		We therefore have that \[
		M_j (\lambda) = \mathrm{Vol}(\mathcal{K}_{\alpha}) \cdot \left(\gamma(c)\right)^j \binom{\alpha + 1}{j} \cdot \frac{\alpha!}{(j-1)! \ (\alpha - j)!} \int_0^1 e^{-\lambda x} x^{j - 1} (1 - x)^{\alpha - j} \ \mathrm{d} x.\]
		
		Next note that since $(1-x)^{\alpha-j} \leq 1$, then \[\int_0^1 e^{-\lambda x} x^{j - 1}(1 - x)^{\alpha - j} \ \mathrm{d} x \leq \int_0^1 e^{-\lambda x} x^{j - 1} \ \mathrm{d} x = \lambda^{-j} \bigl((j-1)! - \Gamma(j, \lambda)\bigr) \leq  \lambda^{-j} (j-1)!\] where $\Gamma(a, z)$ denotes the incomplete gamma function.
		
		Hence $M_j (\lambda)$ is upper bounded by $\left({\gamma(c)}/{\lambda}\right)^j \binom{\alpha + 1}{j} \frac{1}{(\alpha-j)!}$. We next seek a lower bound on $M_j (\lambda)$. First note that $\lambda  = \Omega\left((\log n)^2\right)$ and that $j = O(\log n)$. Hence for all $n$ sufficiently large, $j \lambda^{-1} \in (0, 1)$ and therefore for $x \in \left[0, j \lambda^{-1}\right]$ it follows that $1 - x \geq 1 - j \lambda^{-1}$. 
		
		We therefore have that, for all $n$ sufficiently large,
		\begin{align*}
			\int_0^1 e^{-\lambda x} x^{j - 1}(1 - x)^{\alpha - j} \ \mathrm{d} x  &\geq \int_0^{j \lambda^{-1}} e^{-\lambda x} x^{j - 1} \left(1 - j \lambda^{-1}\right)^{\alpha - j} \ \mathrm{d} x \\ &= \left(1 - j \lambda^{-1}\right)^{\alpha-j} \lambda^{-j} \bigl((j-1)! - \Gamma(j, j)\bigr).
		\end{align*}
		
		For all $n$ sufficiently large, we have that $\lambda > \frac{\varepsilon}{200 \pmin} (\log_{d-1} n) (\log n)$. Together with the fact that $j \leq \alpha \leq ({4}/{\pmin}) \log_{d - 1} n$, we have that for sufficiently large $n$, \[\left(1-j \lambda^{-1}\right)^{\alpha - j} \geq \left(1 - \frac{\left({800}/{\varepsilon}\right)}{\log n}\right)^{({4}/{\pmin}) \log_{d-1} n} \geq e^{-\frac{6400}{\varepsilon \pmin}}.\]
		
		We also have that the term $(j-1)! - \Gamma(j, j)$ is lower bounded by $\frac{(j-1)!}{2}$. Let $\theta \coloneqq \frac{1}{2} e^{-\frac{6400}{\varepsilon \pmin}}$. Then $M_j (\lambda)$ is lower bounded by $\theta \left({\gamma(c)}/{\lambda}\right)^j\binom{\alpha + 1}{j}\frac{1}{(\alpha - j)!}$. Next note that, for all $n$ sufficiently large,
		\begin{align}
			M_j (\lambda) - M_{j+1}(\lambda) &\geq \frac{\theta}{(\alpha-j)!} \left(\frac{\gamma(c)}{\lambda}\right)^j \binom{\alpha + 1}{j} - \left(\frac{\gamma(c)}{\lambda}\right)^{j+1} \frac{1}{(\alpha-j-1)!} \binom{\alpha + 1}{j+1} \nonumber \\
			&= \frac{1}{(\alpha-j)!} \left(\frac{\gamma(c)}{\lambda}\right)^j \binom{\alpha + 1}{j} \left(\theta - \frac{\gamma(c) (\alpha-j)(\alpha-j+1)}{\lambda(j+1)}\right) \nonumber \\
			&\geq \frac{1}{(\alpha-j)!} \left(\frac{\gamma(c)}{\lambda}\right)^j \binom{\alpha + 1}{j} \left(\theta - \gamma(c) \left(\frac{1600 ({4}/{\pmin} + {1}/{(\log n)})}{\varepsilon}\right)\right) \nonumber \\
			&\geq  \frac{1}{(\alpha-j)!} \left(\frac{\gamma(c)}{\lambda}\right)^j \binom{\alpha + 1}{j} \left(\theta - \gamma(c) \left(\frac{12800}{\varepsilon \pmin}\right)\right) \label{eq:sil5}.
		\end{align}
		and hence for $M_j (\lambda) - M_{j+1}(\lambda) \geq 0$ we require that $\theta \geq \gamma(c) \left(\frac{12800}{\varepsilon \pmin}\right)$. 
		
		Substituting for $\theta$ and $\gamma(c)$ then re-arranging, we require that 
		\begin{equation} \label{eq:sil6}
			e^{c \left({\varepsilon}/{2}\right)} \geq C \left(\frac{25600}{\varepsilon \pmin}\right)e^{\frac{6400}{\varepsilon \pmin}}.
		\end{equation}
		
		In particular, there exists $c > 0$ such that (\ref{eq:sil6}) is satisfied and hence $M_j (\lambda) > M_{j+1}(\lambda)$ for all $1 \leq j \leq \alpha - 1$ by (\ref{eq:sil5}). Therefore, for this choice of $c$, together with (\ref{eq:sil4}) we have that
		\begin{equation} \label{eq:sil7}
			\int_{\mathcal{K}_{\alpha}} \prod_{i = 0}^\alpha \left(1 - C e^{-(\mu_0 / 2) \delta - \lambda x_i}\right) \ \mathrm{d} \overrightarrow{x} \geq M_0 (\lambda) - M_1 (\lambda) - M_{\alpha + 1} (\lambda).
		\end{equation}
		
		Now, $M_0 (\lambda) = \frac{1}{\alpha!}$, $M_1 (\lambda) \leq \frac{\gamma(c)}{\lambda} \frac{\alpha + 1}{(\alpha - 1)!} = \frac{1}{\alpha!} \left(\gamma(c)\lambda^{-1} \alpha (\alpha + 1)\right)$ and $M_{\alpha + 1} (\lambda) = \frac{1}{\alpha!} \frac{\left(\gamma(c)\right)^{\alpha + 1}}{e^\lambda}$. Since $\gamma(c)$ is a constant independent of $n$, $\alpha = \Theta(\log n)$ and $\lambda = \Omega\left((\log n)^2\right)$, it follows that $M_{\alpha + 1} (\lambda) = \frac{o(1)}{\alpha!}$. In particular then, for $n$ sufficiently large, we have that 
		\begin{align}
			M_0 (\lambda) - M_1 (\lambda) - M_{\alpha + 1} (\lambda) &\geq \frac{1}{\alpha!} \left(1 - \frac{\gamma(c) \alpha (\alpha + 1)}{\lambda} - o_n(1)\right) \nonumber \\
			&\geq \frac{1}{\alpha!} \left(1 - o_n(1) - \gamma(c) \left(\frac{12800}{\varepsilon \pmin}\right)\right) \nonumber \\
			&\geq \frac{1}{\alpha!} \left(\frac{1}{4} + \left(\theta - \gamma(c) \left(\frac{12800}{\varepsilon \pmin}\right)\right)\right) \nonumber \\
			&\geq \frac{1}{4} \frac{1}{\alpha!} \label{eq:sil8}
		\end{align}
		where the first inequality follows from our bounds on each term, the second inequality follows from $\alpha \leq ({4}/{\pmin}) \log_{d-1} n$ and $\lambda > \frac{\varepsilon}{200 \pmin} (\log_{d-1} n)(\log n)$, the third inequality follows from the fact that $\theta < \frac{1}{2} < \frac{3}{4} - o_n(1)$ for $n$ sufficiently large, and the last inequality follows from our previous remark that our choice of $c$ is such that $\theta \geq \gamma(c) \left(\frac{12800}{\varepsilon \pmin}\right)$. The result follows from (\ref{eq:sil1}), (\ref{eq:sil7}) and (\ref{eq:sil8}).
	\end{proof}
	
	\section{Proofs of Lemmas \ref{lem:f_alpha_properties} and \ref{lem:poisson_concentration}} \label{appendix:B}
	
	\fAlphaLem*
	
	\begin{proof}
		First note that $\frac{\partial}{\partial x} f_\alpha(x, y) = \log\left(\frac{a(\alpha - x - y)}{x(1-a-b)}\right)$ and $\frac{\partial}{\partial y} f_\alpha(x, y) = \log\left(\frac{b(\alpha - x - y)}{y(1-a-b)}\right)$. Both of these are zero at $(x, y) = \left(a \alpha, b \alpha\right)$. Hence \textit{(i)} follows.
		
		Let $c = 1 - a - b$ and define $h_\alpha(x, y) \coloneqq f_\alpha\left(a \alpha, b \alpha\right) - f_\alpha\left(a (\alpha + x), b (\alpha + y)\right)$. We will show that, given $\alpha$ sufficiently large, for all $0 \leq x, y \leq \sqrt{\alpha}$ we have that $h_\alpha (x, y) \leq h_\alpha (\sqrt{\alpha}, \sqrt{\alpha})$. It suffices to show that $\left(\frac{\partial}{\partial x} h_\alpha\right) (x, y) \geq 0$ and $\left(\frac{\partial}{\partial y} h_\alpha\right) (x, y) \geq 0$ over the region $0 \leq x, y \leq \sqrt{\alpha}$. First, observe that, for all $\alpha$ sufficiently large, we have that for all $0 \leq x, y \leq \sqrt{\alpha}$:
		\begin{enumerate}[i.]
			\item $\left(\frac{\partial^2}{\partial x \partial x} h_\alpha\right) (x, y) = \frac{a}{\alpha + x} + \frac{a^2}{c \alpha - a x - by} > 0$ and $\left(\frac{\partial^2}{\partial y \partial y} h_\alpha\right) (x, y) = \frac{b}{\alpha + y} + \frac{b^2}{c \alpha - a x - by} > 0$;
			\item $\left(\frac{\partial^2}{\partial x \partial y} h_\alpha\right) (x, y) = \frac{a b}{c \alpha - a x - by} > 0$.
		\end{enumerate}	
		Also,
		\begin{align*}
			\left(\frac{\partial}{\partial x} h_\alpha\right) (x, y) &= a \left( \log\left(\frac{c}{c \alpha - a x - b y}\right) - \log\left(\frac{1}{\alpha + x}\right)\right) \\
			\mathrm{and} \quad \left(\frac{\partial}{\partial y} h_\alpha\right) (x, y) &= b \left(\log\left(\frac{c}{c \alpha - a x - b y}\right) - \log\left(\frac{1}{\alpha + y}\right)\right)
		\end{align*}
		which are both $0$ at $(x, y) = (0, 0)$. Hence, in conjunction with the second-order derivatives, we have that for $\alpha$ sufficiently large then $\left(\frac{\partial}{\partial x} h_\alpha\right) (x, y) \geq 0$ and $\left(\frac{\partial}{\partial y} h_\alpha\right) (x, y) \geq 0$ over the region $0 \leq x, y \leq \sqrt{\alpha}$, and consequently $h_\alpha (x, y) \leq h_\alpha (\sqrt{\alpha}, \sqrt{\alpha})$ over this region.
		
		All that remains to prove \textit{(ii)} is to show that for $\alpha$ sufficiently large, $h_\alpha \left(\sqrt{\alpha}, \sqrt{\alpha}\right) \leq \frac{a+b}{2(1-a-b)}$. Consider the function \[g(z) = \frac{(1-c)^2 \left((1-c + cz) \log \left(1+\frac{c z}{1-c}\right)+ c(1-z) \log (1-z)\right)}{c^2 z^2}\] where, after some simple algebraic manipulation, we have that
		\begin{align*}
			h_\alpha (\sqrt{\alpha}, \sqrt{\alpha}) &= \sqrt{\alpha} \left(\left(1+\sqrt{\alpha}\right) (1-c) \log \left(1+\frac{1}{\sqrt{\alpha}}\right)+\left(c\sqrt{\alpha} + c - 1\right) \log \left(1-\frac{1-c}{c \sqrt{\alpha}}\right)\right)\\
			&= g\left(\frac{1-c}{c \sqrt{\alpha}}\right).
		\end{align*}
		
		From the Taylor expansion of $g(z)$, we have that $g(z) = \frac{1-c}{2c} + O(z)$ and therefore $\lim\limits_{\alpha \to \infty} g\left(\frac{1-c}{c \sqrt{\alpha}}\right) = \frac{1-c}{2c}$. We also have that the Taylor expansion of $g'(z)$ is given by \[g'(z) = \frac{1-2c}{6c} + \frac{(1-c)^2}{c}\sum_{k = 1}^\infty \frac{k-1}{k(k+1)}\left(1-\frac{c^{k+2}}{(c-1)^{k+2}}\right)z^k\] which one can readily check is negative for $z$ sufficiently small given $\frac{1}{2} < c < 1$. In particular there exists $\alpha_0 > 0$ such that for all $\alpha > \alpha_0$ the function $g\left(\frac{1-c}{c \sqrt{\alpha}}\right)$ is increasing. Therefore for all $\alpha > \alpha_0$, $g\left(\frac{1-c}{c \sqrt{\alpha}}\right) < \frac{1-c}{2c}$ and hence \textit{(ii)} follows.
	\end{proof}
	
	\poissonConcentration*
	
	\begin{proof}
		For any $Z \sim \mathsf{Poisson}(T)$ for some $T > 0$, it is well known (see e.g. Proposition 2.10 in \cite{alma991014693659707026}) that for any $x > 0$, $\mathbb{P}\left[\left|Z - T\right| \geq x\right] \leq 2 \exp\left(-\frac{x^2}{2(T + x)}\right)$. Fix $T = \frac{\kmin + \kmax}{2}$ and $x = \frac{\kmax - \kmin}{2}$; observe that $\Cmin \log n \leq T \leq \Cmax \log n$ and $\left({\Cmid}/{2}\right) \log n \leq x \leq \left({\Cmax}/{2}\right) \log n$. 
		
		Letting $C \coloneqq \frac{\Cmid^2}{12 \Cmax}$, since $Y$ has rate 1 then $N_Y(T)$ is distributed as $\mathsf{Poisson}(T)$ and hence
		\begin{equation*}
			\mathbb{P} \left[\kmin \leq N_Y (T) \leq \kmax \right] \geq 1 - \mathbb{P}\left[\left|N_Y (T) - T\right| \geq x\right] \geq 1 - 2n^{-C} \geq \frac{1}{2}
		\end{equation*}
		where the last inequality follows holds for all $n$ sufficiently large, as required.
	\end{proof}
    
\end{document}

%% file: Figures/fig_vgood_walks.tex
\begin{tikzpicture}[
    every node/.style={circle, draw, fill=black, inner sep=2pt},
    level distance=1.2cm, sibling distance=1.5cm
    ]
    
    \node[fill=black, draw=none] (root) at (0,0) {\textcolor{white}{$u$}};
    
    \node (c1) at (-1.5,-1) {};
    \node (c2) at (-0.9,-1) {};
    \node[fill=none, draw=none] at (-1, -0.25) {$H$};
    \draw[very thick] (root) -- (c1) -- (c2) -- (root);
    \node[draw=gray, fill=gray!20, inner sep=2pt] (c1a) at (-2.1,-1.8) {};
    \node[draw=gray, fill=gray!20, inner sep=2pt] (c2a) at (-1.2,-1.8) {};
    \draw[dashed, gray, very thick] (c1) -- (c1a);
    \draw[dashed, gray, very thick] (c2) -- (c2a);
    \draw[dashed, gray, very thick] (c2a) to[bend left=15] (c1a);
    
    \node[fill=black, draw=none] (r1) at (1.2,-1.2) {\textcolor{white}{$u_j$}};
    \node[fill=none, draw=none] at (2, -1.2) {$T_j$};
    \draw[very thick] (root) -- (r1);
    
    \node (l1a) at (0.6,-2.4) {};
    \node[draw=red, fill=none, cross out, minimum size=4pt, inner sep=0pt, very thick] (l1b) at (1.8,-2.4) {};
    \draw[very thick] (r1) -- (l1a);
    \draw[very thick] (r1) -- (l1b);
    
    \node (l2a) at (0.3,-3.6) {};
    \node[draw=red, fill=none, cross out, minimum size=4pt, inner sep=0pt, very thick] (l2b) at (0.9,-3.6) {};
    \node (l2c) at (1.5,-3.6) {};
    \node (l2d) at (2.1,-3.6) {};
    \draw[very thick] (l1a) -- (l2a);
    \draw[very thick] (l1a) -- (l2b);
    \draw[very thick] (l1b) -- (l2c);
    \draw[very thick] (l1b) -- (l2d);
    
    \node[draw=red, fill=none, cross out, minimum size=4pt, inner sep=0pt, very thick] (l3a) at (0.15,-4.8) {};
    \node (l3b) at (0.45,-4.8) {};
    \node (l3c) at (0.75,-4.8) {};
    \node (l3d) at (1.05,-4.8) {};
    \node[draw=red, fill=none, cross out, minimum size=4pt, inner sep=0pt, very thick] (l3e) at (1.35,-4.8) {};
    \node (l3f) at (1.65,-4.8) {};
    \node (l3g) at (1.95,-4.8) {};
    \node (l3h) at (2.25,-4.8) {};
    
    \draw[very thick] (l2a) -- (l3a);
    \draw[very thick] (l2a) -- (l3b);
    \draw[very thick] (l2b) -- (l3c);
    \draw[very thick] (l2b) -- (l3d);
    \draw[very thick] (l2c) -- (l3e);
    \draw[very thick] (l2c) -- (l3f);
    \draw[very thick] (l2d) -- (l3g);
    \draw[very thick] (l2d) -- (l3h);
    
    \draw[decorate, decoration={brace, mirror}, thick]
    ($(l3h) + (0.4,0.05)$) -- ($(l1b) + (0.4,0.05)$);
    
    \node[draw=none, fill=none, right] at ($(l1b) + (0.6,0)$) {$\frac{1}{10} R = \frac{1}{50} \log_{d-1} n$};
    \node[draw=none, fill=none, right] at ($(l3h) + (0.5,0)$) {$\frac{1}{5} R = \frac{1}{25} \log_{d-1} n$};
    
    \node[draw=none, fill=none, right] at (2.8,-6.0) {\textcolor{gray}{$R = \frac{1}{5} \log_{d-1} n$}};
    
    \draw[dashed, gray, very thick] (l3h) -- +(0.2,-1.2);
    \draw[dashed, gray, very thick] (l3a) -- +(-0.2,-1.2);
    \node[draw=gray, fill=gray!20, inner sep=2pt] (l4a) at ($(l3h)+(0.2,-1.2)$) {};
    \node[draw=gray, fill=gray!20, inner sep=2pt] (l4b) at ($(l3a)+(-0.2,-1.2)$) {};
    \draw[dashed, gray, very thick] (l4a) to[bend left=15] (l4b);
    
    \draw[blue!50, dashed, very thick] (0.6,-2.4) circle (1.5);
    \draw[blue!50, <->, very thick] (0.6,-2.4) -- (-0.9, -2.4)
    node[draw=none, fill=none, left, sloped] {$k = \log_{d-1} \log n$};
\end{tikzpicture}